\RequirePackage{fix-cm}
\documentclass{svjour3}                     
\smartqed  
\makeatletter
\def\cl@chapter{\@elt {theorem}}
\makeatother

\usepackage{lipsum}
\usepackage{amsfonts}
\usepackage[final]{graphicx}
\usepackage{epstopdf}
\usepackage{algorithmic}
\usepackage{algorithm}
\usepackage{thmtools}
\usepackage{amsmath}
\usepackage{hyperref}
\usepackage{cleveref}
\Crefname{prop}{Proposition}{Propositions}
\Crefname{thm}{Theorem}{Theorems}
\Crefname{assumption}{Assumption}{Assumptions}
\Crefname{lem}{Lemma}{Lemmas}

\usepackage{csvsimple}
\usepackage{booktabs}
\usepackage{siunitx}
\usepackage{numprint}
\usepackage{amsopn}
\usepackage{amssymb}
\renewenvironment{proof}{{\itshape Proof.}}{\qed}

\ifpdf
  \DeclareGraphicsExtensions{.eps,.pdf,.png,.jpg}
\else
  \DeclareGraphicsExtensions{.eps}
\fi

\usepackage{titlesec}
\titleformat{\paragraph}[runin]
{\normalfont\normalsize\bfseries}{\theparagraph}{1em}{}
\titlespacing*{\paragraph}
{18pt}{0.35ex plus 1ex minus .2ex}{1.ex plus .2ex}

\titleformat{\section}[runin]{\normalfont\normalsize\bfseries}{\thesection.}{1em}{}[.]
\titleformat{\subsection}[runin]{\normalfont\normalsize\bfseries}{\thesubsection.}{1em}{}[.]
\titleformat{\subsubsection}[runin]{\normalfont\normalsize\bfseries}{\thesubsubsection.}{1em}{}[.]
\titleformat{\paragraph}[runin]{\normalfont\normalsize\bfseries}{\theparagraph.}{1em}{}[.]

\usepackage{ifthen}
\newboolean{longversion}
\setboolean{longversion}{true}

\newcommand{\R}{\mathbb{R}}
\newcommand{\N}{\mathbb{N}}

\newcommand{\E}{\mathbb{E}}
\newcommand{\Eone}{\mathbb{E}_1}
\newcommand{\Etwo}{\mathbb{E}_2}
\newcommand{\Ethree}{\mathbb{E}_3}

\newcommand{\n}{N}
\newcommand{\domhi}{X_i}
\newcommand{\domfi}{\dom(f_i)}
\newcommand{\domfibiconjugate}{\dom(f_i^{**})}
\newcommand{\ones}{\mathbf{1}}

\newcommand{\diam}{\text{diam}}

\newcommand{\AvgDiamSquared}{D^2}
\newcommand{\SqrtAvgDiamSquared}{D}
\newcommand{\Gtilde}{\tilde{G}}
\newcommand{\ErrTK}{R}
\newcommand{\alphabar}{\Lambda}

\newcommand{\conv}{\operatorname{conv}}

\newcommand{\epi}{\operatorname{epi}}
\newcommand{\dom}{\operatorname{dom}}

\newcommand{\ProjPos}[1]{\operatorname{\Pi_+}\left(#1\right)}

\DeclareMathOperator*{\argmax}{arg\,max}
\DeclareMathOperator*{\argmin}{arg\,min}
\usepackage{mathtools}
\newcommand{\norm}[1]{\left\lVert#1\right\rVert}
\newcommand{\abs}[1]{\left\lvert#1\right\rvert}

\newcommand{\iprod}[2]{\left\langle #1, #2 \right\rangle}

\newtheorem{assumption}{Assumption}[section]
\numberwithin{assumption}{section}
\newtheorem{thm}{Theorem}[section]
\numberwithin{thm}{section}
\newtheorem{prop}{Proposition}[section]
\numberwithin{prop}{section}
\newtheorem{lem}{Lemma}[section]
\numberwithin{lem}{section}
\newtheorem{cor}{Corollary}[section]
\numberwithin{cor}{section}

\begin{document}

\title{Two-stage stochastic algorithm for solving large-scale (non)-convex separable optimization problems under affine constraints
}


\titlerunning{Separable optimization}        

\author{Benjamin Dubois-Taine \and Laurent Pfeiffer \and Nadia Oudjane \and Adrien Seguret \and Francis Bach 
}


\institute{Benjamin Dubois-Taine \at
              Universit\'e Paris-Saclay, CNRS, CentraleSup\'elec, Inria, Laboratoire des Signaux et Syst\`emes  \\
              \email{benjamin.paul-dubois-taine@inria.fr}           
           \and
           Laurent Pfeiffer \at
             Universit\'e Paris-Saclay, CNRS, CentraleSup\'elec, Inria, Laboratoire des Signaux et Syst\`emes\\
             F\'ed\'eration de Math\'ematiques de CentraleSup\'elec France\\   \email{laurent.pfeiffer@inria.fr} 
              \and
              Nadia Oudjane \at EDF R\&D \\ \email{nadia.oudjane@edf.fr}
              \and 
              Adrien S\'eguret \at EDF R\&D\\ 
              \email{adrien.seguret@edf.fr}
              \and
              Francis Bach \at Inria, Ecole Normale Sup\'erieure, PSL Research University \\\email{francis.bach@inria.fr}
}

\date{Received: date / Accepted: date}

\maketitle

\begin{abstract}

We consider nonsmooth optimization problems under affine constraints, where the objective consists of the average of the component functions of a large number $\n$ of agents, and we only assume access to the Fenchel conjugate of the component functions. The algorithm of choice for solving such problems is the dual subgradient method, also known as dual decomposition, which requires $O(\frac{1}{\epsilon^2})$ iterations to reach $\epsilon$-optimality in the convex case. However, each iteration requires computing the Fenchel conjugate of each of the $\n$ agents, leading to a complexity $O(\frac{\n}{\epsilon^2})$ which might be prohibitive in practical applications. To overcome this, we propose a two-stage algorithm, combining a stochastic subgradient algorithm on the dual problem, followed by a block-coordinate Frank-Wolfe algorithm to obtain primal solutions. The resulting algorithm requires only $O(\frac{1}{\epsilon^2} + \frac{\n}{\epsilon^{2/3}})$ calls to Fenchel conjugates to obtain an $\epsilon$-optimal primal solution in expectation in the convex case. We extend our results to nonconvex component functions and show that our method still applies and gets (almost) the same convergence rate, this time only to an approximate primal solution recovering the classical duality gap bounds usually obtained using the Shapley-Folkman theorem.

\end{abstract}

\section{Introduction}
\label{intro}

We consider separable optimization problems of the form
\begin{align}
\tag{P}
    \begin{split}
    \mbox{minimize} & \ \frac{1}{\n} \sum_{i=1}^\n h_i(x_i)\\
    \mbox{subject to }& \ \frac{1}{\n}\sum_{i=1}^\n A_ix_i  \leq  b,\\
    &x_i \in \domhi, \ i=1, \dots, \n,
    \end{split}
    \label{eq:primal}
\end{align}
in the variables $x_i \in \R^{d_i}$, with $A_i \in \R^{m \times d_i}$, $b \in \R^m$. We assume that the component functions $h_i$ are proper, convex and lower semicontinuous with a convex compact domain $\domhi = \dom(h_i)$. We also assume we have access to the following minimization oracle: given $\gamma \geq 0$ and $\lambda \in \R^m$, the oracle returns
\begin{align}
\label{eq:oracle-definition}
\tag{O1}
    x_i^*((\gamma, \lambda)) \in \argmin_{x_i \in \domhi} \left\{ \gamma h_i(x_i) + \lambda^\top A_i x_i \right\}.
\end{align}
We shall often encounter the case $\gamma = 1$, in which case we simplify notation and simply write
\begin{align}
    x_i^*( \lambda) \in \argmin_{x_i \in \domhi} \left\{ h_i(x_i) + \lambda^\top A_i x_i \right\}.
\end{align}
Note that access to this oracle when $\gamma \not = 0$ is equivalent to computing the subgradient of the Fenchel conjugate of $h_i$, as we have
\begin{align*}
    \argmin_{x_i \in \domhi} \left\{ \gamma h_i(x_i) + \lambda^\top A_i x_i \right\} = \argmax_{x_i \in X_i} \left\{ \left(-\frac{A_i^\top \lambda}{\gamma}\right)^\top x_i - h_i(x_i) \right\} = \partial h_i^*\left(-\frac{A_i^\top \lambda}{\gamma}\right).
\end{align*}
In the case $\gamma = 0$, oracle~\eqref{eq:oracle-definition} is a linear minimization over $\domhi$, which we assume is strictly simpler to compute than the Fenchel conjugate of $h_i$.

\subsection{Background and related work}

\subsubsection{Convex case}
In order to solve optimization problems, it is common to derive the associated dual problem. Subgradient algorithms on this dual problem have been extensively used and studied to solve the original problem, or to at least obtain lower bounds on its optimal value. Starting with the works~\cite{polyak1967general,ermol1966methods}, several results for different stepsizes have been derived~\cite{polyak1987introduction,hiriart1996convex,bertsekas1997nonlinear,bertsekas2003convex}, and extended to distributed settings as well as so-called incremental versions of subgradient algorithms~\cite{kiwiel2001parallel,zhao1999surrogate,nedic2001incremental,nedic2001distributed}.

The above results typically only deal with convergence in the dual space, yet what we are eventually interested in is primal convergence. It turns that when running dual subgradient algorithms, one obtains primal points as a by-product of the computation of the subgradient of the dual function. Convergence for the average of those primal iterates has been investigated, and first results in this direction trace back to at least~\cite{nemirovski1978cezari} and~\cite{shor2012minimization}. More general settings and convex combination schemes than just the averaged iterates have also been studied in~\cite{larsson1997lagrangean,larsson1996ergodic,gustavsson2015primal,onnheim2017ergodic}. However, those results only establish primal convergence of the dual subgradient asymptotically. In~\cite{nedic2009approximate}, the authors derive the first rate of convergence for the averaged primal iterates of the dual subgradient in $O(1/\sqrt{k})$ after $k$ iterations. Specifying this result to our problem~\eqref{eq:primal} where each iteration costs $O(\n)$, this gives a total complexity of $O(\n/\epsilon^2)$ to reach an $\epsilon$-accurate primal solution. Our focus being on instances where $\n$ is large, this may be problematic. A similar rate for the last iterate of a variant of the dual subgradient algorithm was also established in~\cite{nesterov2018dual}. The $O(1/\sqrt{k})$ convergence rate being optimal for nonsmooth convex optimization, efforts to improve them require additional assumptions, for example strong convexity~\cite{beck20141}.

In the large-scale separable setting as it appears in~\eqref{eq:primal}, dual subgradient methods make strong candidates due to the structure of the dual. This has been explored, for example, in~\cite{nedic2009distributed,beck20141}, where the typical argument in favor of those methods is that the computation of the subgradient can be parallelized among the agents, which in turn reduces the overall computation time. While this is true, one cannot always afford to parallelize or distribute the computations, and it is therefore still of interest to reduce the overall complexity on $\n$.

Finally, an important class of methods to solve problems of the form~\eqref{eq:primal} rely on the augmented Lagrangian framework~\cite{ruszczynski1995convergence}. In particular, the ADMM algorithm and its variants~\cite{glowinski1975approximation,gabay1976dual} have gained popularity in recent years. We refer the reader to~\cite{boyd2011distributed} and references therein for more on ADMM-like algorithms. The important point is that those methods rely on the computation of the proximal operator of the functions, i.e., they replace the linear term in~\eqref{eq:oracle-definition} by a quadratic term. We make no such assumption, and our oracle is strictly simpler to compute than the proximal of the functions $h_i$. This is particularly true in the nonconvex case where we may even encounter discrete function domains.

\subsubsection{Nonconvex case}

When the convexity assumptions is relaxed, the well-known result from Aubin and Ekeland~\cite{aubin1976estimates} uses the Shapley-Folkman theorem to provide duality gap bounds which scale with the ratio $\frac{m}{\n}$. In particular, the duality gap vanishes in instances with a growing number of agents $\n$ and a fixed number of constraints $m$. We also refer to~\cite{bi2016refined,bertsekas1983optimal} for similar results with slightly different assumptions.

Those results however do not provide an algorithmic approach to obtain the duality gap bounds. In~\cite{udell2016bounding}, the authors provide a method to achieve those bounds with probability 1, but it requires solving an optimization problem over the full solution of the bidual of the original problem. In~\cite{dubois2025frank}, the authors provide an approach based on the Frank-Wolfe algorithm on a well-designed optimization problem and (approximate) Carath\'eodory algorithms to obtain those duality gap bounds. Our proposed method resembles their approach but, in contrast with them, we do not require knowledge of the optimal dual value, and we shall show that the overall complexity of our approach scales better with $\n$. We discuss the key differences in more detail in Section~\ref{sec:nonconvex}, once all the necessary mathematical tools have been introduced.

\subsection{Roadmap of paper and contributions}

The paper is organized as follows. Let us emphasize that we will assume convexity throughout Section~\ref{sec:preliminaries} and Section~\ref{sec:proposed-algorithm}, and will relax this assumption in Section~\ref{sec:nonconvex}.

In Section~\ref{sec:preliminaries}, we start by introducing all the necessary mathematical and algorithmic tools to study our problem. We then present the dual subgradient algorithm to solve~\eqref{eq:primal} in Subsection~\ref{sec:dual-subgradient}. This algorithm simply consists in a subgradient ascent algorithm on the dual problem, collecting one primal point as a by-product of the computation of the subgradient of the dual at each iteration. A candidate primal solution can then be computed as the average of these primal points~\cite{nedic2009approximate}. We recall the convergence rates of the resulting algorithm, and we end this section by showing that to obtain $\epsilon$-accurate solutions of both the primal problem and the dual problem, the number of calls to the Fenchel conjugates is of the order $O(\n / \epsilon^2)$, which is prohibitive in settings where $\n$ is large.

In Section~\ref{sec:proposed-algorithm} we detail our two-stage approach to improve on this complexity. At first glance it may seem like a simple stochastic version of the deterministic dual subgradient algorithm might be enough to improve on the complexity. We therefore present this stochastic version in Subsection~\ref{sec:stochastic-dual-subgradient}, for which we recall well-known convergence rates for the dual problem. Our first contribution then comes at this point, where we prove a convergence rate for primal candidates (Proposition~\ref{prop:primal-convergence-stochastic-dual-subgradient}). Unfortunately the resulting complexity is still $O(\n/\epsilon^2)$ to obtain $\epsilon$-accurate primal solutions, and thus the algorithm does not directly improve on its deterministic counterpart, justifying the need for our more evolved two-stage approach. We end the section by noting that although we do get this complexity $O(\n/\epsilon^2)$ in the primal, we only get complexity $O(1/\epsilon^2)$ to obtain an $\epsilon$-accurate dual solution in expectation. Although this last complexity result is not new, we emphasize it as it is a crucial part of our approach.

We then leverage both the primal and dual convergence rates of the stochastic dual subgradient algorithm in the second stage of our approach, which we present in Subsection~\ref{sec:BCFW}. There, we show how one can obtain primal points through a block-coordinate Frank-Wolfe algorithm~\cite{lacoste2013block} on a well-crafted smooth objective function. The main contribution of this work comes at the end of this section, where we show that putting the two stages together, the proposed algorithm (Algorithm~\ref{alg:2-stage-algorithm} in the text) outputs, in expectation, an $\epsilon$-accurate primal solution after at most $O(\frac{1}{\epsilon^2} + \frac{N}{\epsilon^{2/3}})$ calls to Fenchel conjugates of the component functions, hence dramatically improving over the $O(N/\epsilon^2)$ complexity of the dual subgradient algorithm. This is summarized in Theorem~\ref{thm:2-stage-algorithm-complexity}.

In Section~\ref{sec:nonconvex}, we extend our results to nonconvex separable problems, where the primal~\eqref{eq:primal} can be seen as the bidual of the considered nonconvex problem. We start the section by recalling the classical duality gap bounds that can be obtained using the famous Shapley-Folkman theorem. We then show that our two-stage algorithmic framework can still be applied to the bidual of nonconvex problems without any modification up until the very last output, where a final step consisting of the computation of a Carath\'eodory decomposition must be carried out in order to obtain a solution to the primal nonconvex problem.

\section{Preliminaries}
\label{sec:preliminaries}

If $f$ is a function not everywhere $+ \infty$, with an affine lower bound, the (Fenchel) conjugate of $f$ is defined as
\begin{align}
    f^*(y) = \sup_{x \in \dom f} \{ y^\top x - f(x) \}.
\end{align}

We use $\norm{\cdot}$ as the standard Euclidean norm. For a vector $x\in \R^p$, we also define $\norm{x}_+ = \sqrt{\sum_{j=1}^p \max(x_j, 0)^2}$. Similarly, the projection $\ProjPos{x}$ of $x$ onto the positive orthant is defined as
\begin{align*}
    \ProjPos{x}_j &= \max(x_j, 0), \ j=1,\dots, p.
\end{align*}

\subsection{Dual}
We define the dual function as
\begin{align*}
    d(\lambda) := -\lambda^\top b + \frac{1}{\n} \sum_{i=1}^\n \inf_{x_i \in \domhi} \left\{ h_i(x_i) + \lambda^\top A_i x_i \right\}.
\end{align*}
The dual problem of~\eqref{eq:primal} is then 
\begin{align}
    \tag{D}
    \begin{split}
    \mbox{maximize} & \ d(\lambda)\\
    \mbox{subject to }& \ \lambda \geq 0.
    \end{split}
    \label{eq:dual}
\end{align}

We make the assumption that the dual problem has a solution and that strong duality holds.

\begin{assumption}
\label{ass:existence-dual-maximizer}
    There exists $\lambda^* \in \R^m_+$ such that $\lambda^* \in \argmax_{\lambda \in \R^m_+ } d(\lambda)$. We write $d^* = d(\lambda^*)$. Moreover, strong duality holds, i.e., $d^*$ is also the optimal value of the primal problem~\eqref{eq:primal}.
\end{assumption}

The above assumption holds under quite general qualification conditions. We give one example next, whose proof can be found in~\cite[Theorem 2.6]{lemarechal2001geometric}

\begin{prop}
    Suppose that for all $i=1, \dots, \n$, there exists some $x_i \in  \domhi$ such that $\frac{1}{\n} \sum_{i=1}^\n A_i x_i < b$. Then Assumption~\ref{ass:existence-dual-maximizer} holds.
\end{prop}

In this work we will be developing algorithms to solve~\eqref{eq:primal} and~\eqref{eq:dual}. Of course, algorithms only solve such problems up to an arbitrary precision. We define this notion next.

\begin{definition}
    Let $\epsilon > 0$. We say that $x = [x_1, \dots, x_\n] \in \R^d$ is an $\epsilon$-accurate primal solution of~\eqref{eq:primal} if and only if for all $i=1,\dots, \n$, $x_i \in \domhi$ and
    \begin{align*}
        &\frac{1}{\n} \sum_{i=1}^\n h_i(x_i) \leq d^* + \epsilon, \\
        &\norm{\frac{1}{\n} \sum_{i=1}^\n A_i x_i -b }_+ \leq \epsilon.
    \end{align*}
    We say that $x = [x_1, \dots, x_\n]$ is an $\epsilon$-accurate primal solution in expectation if the above bounds hold in expectation.
\end{definition}
We have a similar definition for approximate solutions of the dual.

\begin{definition}
    Let $\epsilon > 0$. We say that $\lambda \in \R^m_+$ is an $\epsilon$-accurate dual solution of~\eqref{eq:dual} if and only if
    \begin{align*}
        d(\lambda) \geq d^* - \epsilon.
    \end{align*}
    We say that $\lambda \in \R^m_+$ is an $\epsilon$-accurate dual solution in expectation if the above bound holds in expectation.
\end{definition}

\subsection{Dual subgradient}
\label{sec:dual-subgradient}

The dual subgradient method is a classical method for solving general optimization problems, i.e., it does not require the separable structure of~\eqref{eq:primal}. The method starts from some $\lambda_0 \in \R^m_+$. At each iteration $t$, it computes a subgradient of $d$ at $\lambda_t$, takes a positive step in that direction, and projects the result back onto the positive orthant. Observe that the subdifferential of $d$ is given by
\begin{align}
\label{eq:subgradient-of-dual-function-definition}
     \partial d(\lambda) = -b + \frac{1}{\n} \sum_{i=1}^\n \argmin_{x_i \in \domhi} \left\{ h_i(x_i) + \lambda^\top A_i x_i \right\},  \ \lambda \in \R^m_+.
\end{align}
In particular, in order to compute the subgradient of $d$, we must make $\n$ calls to oracle~\eqref{eq:oracle-definition}, one for each agent. We summarize the full method in Algorithm~\ref{alg:dual-subgradient-separable-primal}. Note finally that the dual subgradient algorithm outputs the average of the outputs of the oracle as a primal candidate solution for each $i=1,\dots, \n$.

\begin{algorithm}
\caption{Dual Subgradient for problem~\eqref{eq:primal}}
\label{alg:dual-subgradient-separable-primal}
\begin{algorithmic}[1]
\STATE{\textbf{Input}: $\lambda_0 \in \R^m_+$, number of iterations $T$.} 
\FOR{$t=0, \dots, T-1$}
\FOR{$i=1, \dots, \n$}
\STATE{
Compute
\begin{align}
\label{eq:linear-oracle-within-subgradient-algorithm}
    x_i^*(\lambda_t) \in \argmin_{x_i \in \domhi} \left\{ h_i(x_i) + \lambda_t^\top A_ix_i\right\}.
\end{align}
}
\ENDFOR
\STATE{Set $g_t = \frac{1}{\n} \sum_{i=1}^\n A_ix_i^*(\lambda_t) - b \in \partial d(\lambda_t)$.}
\STATE{Pick some stepsize $\alpha_t > 0$.}
\STATE{Set $\lambda_{t+1} = \ProjPos{\lambda_t + \alpha_t g_t}$}
\ENDFOR
\STATE{\textbf{Return:} Dual candidate $\bar{\lambda}_T = \frac{1}{T} \sum_{t=0}^{T-1} \lambda_t$ and primal candidate $\bar{x}_{i, T} = \frac{1}{T} \sum_{t=0}^{T-1} x_i^*(\lambda_t)$ for all $\ i=1,\dots, \n$
}
\end{algorithmic}
\end{algorithm}

Before we move on to the analysis of Algorithm~\ref{alg:dual-subgradient-separable-primal}, we define the quantity
\begin{align}
\label{eq:G-definition}
    G := \sup_{x_i \in \domhi, i=1,\dots, \n} & \norm{\frac{1}{\n} \sum_{i=1}^\n A_i x_i - b}.
\end{align}
In particular, we see from~\eqref{eq:subgradient-of-dual-function-definition} that $G$ is an upper bound on the norm of any element of the subgradient of the dual function $d$, i.e., the function $d$ is $G$-Lipschitz.

We are now ready to provide convergence rates for Algorithm~\ref{alg:dual-subgradient-separable-primal}. We start with dual convergence in Subsection~\ref{sec:dual-subgradient-dual-convergence}, and then proceed to primal convergence in Subsection~\ref{sec:dual-subgradient-primal-convergence}.

\subsubsection{Dual convergence}
The next proposition establishes convergence of the averaged dual iterate $\bar{\lambda}_T$. Let us emphasize that from a strictly dual point of view, Algorithm~\ref{alg:dual-subgradient-separable-primal} is simply subgradient ascent on the Lipschitz concave function $d$, a setting which has been extensively studied in the literature (see for example~\cite[Theorem 3.2.2]{nesterov2013introductory}). We recall the rate and proof for completeness.

\label{sec:dual-subgradient-dual-convergence}
\begin{prop}
\label{prop:dual-subgradient-dual-convergence}
    Suppose Assumption~\ref{ass:existence-dual-maximizer} holds. Consider Algorithm~\ref{alg:dual-subgradient-separable-primal} with input $\lambda_0 = 0_m$ and constant stepsize $\alpha_t = \alpha = \frac{\alphabar}{G\sqrt{T}}$ for some $\alphabar > 0$. The output $\bar{\lambda}_T$ then satisfies
    \begin{align*}
        d^* - d(\bar{\lambda}_T) \leq \frac{G \norm{\lambda^*}^2}{2\alphabar\sqrt{T}} + \frac{G \alphabar}{2\sqrt{T}}.
    \end{align*}
\end{prop}
\begin{proof}
    We have, for any $t=0, \dots, T-1$,
    \begin{align*}
    \norm{\lambda_{t+1} - \lambda^*}^2 &= \norm{ \ProjPos{\lambda_{t} + \alpha g_t} - \ProjPos{\lambda^*}}^2 \\ 
    &\leq \norm{ \lambda_{t} + \alpha g_t - \lambda^*}^2 \tag{by contractivity of $\ProjPos{\cdot}$}\\
        &= \norm{\lambda_t - \lambda^*}^2 + 2\alpha \iprod{g_t}{\lambda_t - \lambda^*} + \alpha^2 \norm{g_t}^2 \\
        &\leq \norm{\lambda_t - \lambda^*}^2 + 2\alpha (d(\lambda_t) - d(\lambda^*)) + \alpha^2 G^2,
    \end{align*}
    by concavity of $d$. Rearranging gives us
    \begin{align*}
        d(\lambda^*) - d(\lambda_t) &\leq \frac{\norm{\lambda_t - \lambda^*}^2 - \norm{\lambda_{t+1} - \lambda^*}^2}{2\alpha} + \frac{\alpha G}{2}.
    \end{align*}
    We thus get, by concavity of $d$ and Jensen's inequality,
    \begin{align*}
        d(\lambda^*) - d(\bar{\lambda}_T) &\leq \frac{1}{T} \sum_{t=0}^{T-1} (d(\lambda^*) - d(\lambda_t)) \\
        &\leq \frac{1}{T} \sum_{t=0}^{T-1} \left(\frac{\norm{\lambda_t - \lambda^*}^2 - \norm{\lambda_{t+1} - \lambda^*}^2}{2\alpha} + \frac{\alpha G^2}{2}\right)\\
        &\leq \frac{\norm{\lambda_0 - \lambda^*}^2}{2\alpha} + \frac{ \alpha G^2}{2}.
    \end{align*}
    Plugging $\alpha = \alphabar/(G\sqrt{T})$ and $\lambda_0 = 0_m$ yields the result.
\end{proof}

\subsubsection{Primal convergence}
\label{sec:dual-subgradient-primal-convergence}

To establish primal convergence, we rely on the following important proposition.

\begin{prop}[Bounded iterates]
\label{prop:dual-subgradient-bounded-iterates}
Suppose Assumption~\ref{ass:existence-dual-maximizer} holds. Consider Algorithm~\ref{alg:dual-subgradient-separable-primal} with input $\lambda_0 \in \R^m_+$. For any $t=0, \dots T$, it holds that
    \begin{align*}
        \norm{\lambda_{t} - \lambda^*}^2 \leq \norm{\lambda_0 - \lambda^*}^2 + G^2 \sum_{j=0}^{t-1} \alpha_j^2.
    \end{align*}
\end{prop}
\begin{proof}
    For any $t=0,\dots, T-1$,
    \begin{align*}
        \norm{\lambda_{t+1} - \lambda^*}^2 &= \norm{ \ProjPos{\lambda_{t} + \alpha g_t} - \ProjPos{\lambda^*}}^2 \\ &\leq \norm{\lambda_{t} + \alpha_t g_t - \lambda^*}^2 \\
        &= \norm{\lambda_t - \lambda^*}^2 + 2\alpha_t \iprod{g_t}{\lambda_t - \lambda^*} + \alpha_t^2 \norm{g_t}^2 \\
        &\leq \norm{\lambda_t - \lambda^*}^2 + 2\alpha_t (d(\lambda_t) - d(\lambda^*)) + \alpha_t^2 G^2 \\
        &\leq  \norm{\lambda_t - \lambda^*}^2 + \alpha_t^2 G^2,
    \end{align*}
    where the first inequality holds by contractivity of the projection operator, the second one by concavity of $d$ and the last one by optimality of $\lambda^*$. Unrolling the recursion gives the result.
\end{proof}

We now state the convergence result for the primal iterates. This result is not new and was already established in~\cite{nedic2009approximate}, but we repeat its proof for completeness.

\begin{prop}
\label{prop:dual-subgradient-primal-convergence}
    Suppose Assumption~\ref{ass:existence-dual-maximizer} holds. Consider Algorithm~\ref{alg:dual-subgradient-separable-primal} with input $\lambda_0 = 0_m$ and constant stepsize $\alpha_t = \alpha = \frac{\alphabar}{G\sqrt{T}}$ for some $\alphabar > 0$. The output $\bar{x}_T$ then satisfies
    \begin{align*}
        &\norm{\frac{1}{\n} \sum_{i=1}^\n A_i \bar{x}_{i, T} - b}_+ \leq \frac{2G\norm{\lambda^*}}{\alphabar\sqrt{T}} + \frac{G}{\sqrt{T}}\\
        &\frac{1}{\n} \sum_{i=1}^\n h_i(\bar{x}_{i, T}) \leq d(\bar{\lambda}_T)  + \frac{ G \alphabar}{2\sqrt{T}}.
    \end{align*}
\end{prop}
\begin{proof}
    We have
    \begin{align*}
        \frac{1}{\n} \sum_{i=1}^\n A_i \bar{x}_{i, T} - b = \frac{1}{\n}\sum_{i=1}^\n  \frac{1}{T} A_i\sum_{t=0}^{T-1} x_i^*(\lambda_t) - b 
        = \frac{1}{T} \sum_{t=0}^{T-1} \left(\frac{1}{\n} \sum_{i=1}^\n A_i x_i^*(\lambda_t) - b \right).
        \end{align*}
Observe that $\frac{1}{\n} \sum_{i=1}^\n A_i x_i^*(\lambda_t) - b$ is precisely the subgradient $g_t$ computed at iteration $t$ in Algorithm~\ref{alg:dual-subgradient-separable-primal}. Thus, letting $\tilde{\lambda}_{t+1} = \lambda_t + \alpha g_t$, we have
\begin{align*}
        \frac{1}{\n} \sum_{i=1}^\n A_i \bar{x}_{i, T} - b &= \frac{1}{T} \sum_{t=0}^{T-1} g_t = \frac{1}{T\alpha} \sum_{t=0}^{T-1} (\tilde{\lambda}_{t+1} - \lambda_t).
    \end{align*}
    Observe that since $\lambda_{t+1} = \ProjPos{\tilde{\lambda}_{t+1}}$, we have $\tilde{\lambda}_{t+1} \leq \lambda_{t+1}$, where the inequality is to be understood component-wise. Thus,
    \begin{align}
    \label{eq:proof-dual-subgradient-telescoping-1}
        \frac{1}{\n} \sum_{i=1}^\n A_i \bar{x}_{i, T} - b &\leq \frac{1}{T\alpha} \sum_{t=0}^{T-1} (\lambda_{t+1} - \lambda_t) = \frac{\lambda_T - \lambda_0}{T\alpha} \leq \frac{\lambda_T}{T\alpha},
    \end{align}
    since $\lambda_0 \geq 0$. Observe then that $\frac{\lambda_T}{T\alpha} \geq 0$ and we can thus deduce that
    \begin{align*}
        \norm{\frac{1}{\n} \sum_{i=1}^\n A_i \bar{x}_{i, T} - b}_+ &\leq \frac{\norm{\lambda_T}}{T\alpha}\\
        &\leq \frac{\norm{\lambda_T - \lambda^*} + \norm{\lambda^*}}{T\alpha} \\
        &\leq  \frac{\norm{\lambda_0 - \lambda^*} + G \sqrt{T} \alpha + \norm{\lambda^*}}{T\alpha} \nonumber \tag{Prop.~\ref{prop:dual-subgradient-bounded-iterates}}.
    \end{align*}
    Plugging $\alpha = \alphabar/(G\sqrt{T})$ and $\lambda_0 = 0_m$ yields the infeasibility result. For primal suboptimality, we have
    \begin{align}
    \begin{split}
        \frac{1}{\n}\sum_{i=1}^\n h_i(\bar{x}_{i, T}) &= \frac{1}{\n} \sum_{i=1}^\n h_i\left(\frac{1}{T} \sum_{t=0}^{T-1} x_i^*(\lambda_t) \right) \\
        &\leq \frac{1}{T } \sum_{t=0}^{T-1} \frac{1}{\n} \sum_{i=1}^\n h_i(x_i^*(\lambda_t)) \\
        &= \frac{1}{T } \sum_{t=0}^{T-1} \underbrace{\frac{1}{\n} \sum_{i=1}^\n\left(h_i(x_i^*(\lambda_t)) + \lambda_t^\top(A_ix_i^*(\lambda_t) - b)  \right)}_{=d(\lambda_t)} - \frac{1}{T } \sum_{t=0}^{T-1} \frac{1}{\n} \sum_{i=1}^\n \lambda_t^\top (A_ix_i^*(\lambda_t)  - b)  \\
        &= \frac{1}{T} \sum_{t=0}^{T-1} d(\lambda_t) - \frac{1}{T } \sum_{t=0}^{T-1} \frac{1}{\n} \sum_{i=1}^\n  \lambda_t^\top (A_ix_i^*(\lambda_t)  - b) \\
        &\leq d(\bar{\lambda}_T)- \frac{1}{T } \sum_{t=0}^{T-1} \lambda_t^\top \left( \frac{1}{\n} \sum_{i=1}^\n A_ix_i^*(\lambda_t)  - b\right).
        \end{split}
        \label{eq:proof-dual-subgradient-primal-1}
    \end{align}
    Now, recall that $\lambda_{t+1} = \ProjPos{\lambda_t + \alpha g_t} = \ProjPos{\lambda_t + \alpha (\frac{1}{\n} \sum_{i=1}^\n A_i x_i^*(\lambda_t) - b)}$. Thus
    \begin{align*}
        \norm{\lambda_{t+1}}^2 \leq \norm{\lambda_t + \alpha g_t} \leq \norm{\lambda_t}^2 + 2 \alpha \lambda_t^\top \left(\frac{1}{\n} \sum_{i=1}^\n A_i x_i^*(\lambda_t) - b\right) + \alpha^2 G^2.
    \end{align*}
    Rearranging gives us
    \begin{align*}
        - \lambda_t^\top \left(\frac{1}{\n} \sum_{i=1}^\n A_i x_i^*(\lambda_t) - b\right) \leq \frac{\norm{\lambda_t}^2 - \norm{\lambda_{t+1}}^2}{2\alpha} + \frac{\alpha G^2}{2},
    \end{align*}
    and thus
    \begin{align}
        - \frac{1}{T } \sum_{t=0}^{T-1} \lambda_t^\top \left( \frac{1}{\n} \sum_{i=1}^\n A_ix_i^*(\lambda_t)  - b\right) &\leq \frac{1}{T} \sum_{t=0}^{T-1} \left( \frac{\norm{\lambda_t}^2 - \norm{\lambda_{t+1}}^2}{2\alpha} + \frac{\alpha G^2}{2} \right) \label{eq:proof-dual-subgradient-telescoping-sum-2}\\
        &\leq \frac{\norm{\lambda_0}^2}{2T\alpha} + \frac{\alpha G^2}{2}. 
        \nonumber
    \end{align}
    Plugging this back into~\eqref{eq:proof-dual-subgradient-primal-1}, we get
    \begin{align*}
        \frac{1}{\n}\sum_{i=1}^\n h_i(\bar{x}_{i, T}) \leq d(\bar{\lambda}_T) + \frac{\norm{\lambda_0}^2}{2T\alpha} + \frac{\alpha G^2}{2}.
    \end{align*}
    Plugging $\alpha = \alphabar/(G\sqrt{T})$ and $\lambda_0 = 0_m$ in the above finishes the proof.
\end{proof}

\subsubsection{Total complexity}
We are now ready to state the complexity to reach $\epsilon$-accurate solutions with respect to the number of calls to oracle~\eqref{eq:oracle-definition}. We do so in the next proposition. We only keep $N$ and $T$ in our complexity bounds, and hide all other terms appearing in the previous rates within the $O$-notation.

\begin{prop}
\label{prop:epsilon-complexity-dual-subgradient}
    Suppose Assumption~\ref{ass:existence-dual-maximizer} holds, and consider Algorithm~\ref{alg:dual-subgradient-separable-primal} applied to problem~\eqref{eq:primal} with constant stepsize $\alpha_t = \alpha = \frac{\alphabar}{G\sqrt{T}}$.

    \textbf{Dual complexity:} the number of calls to the oracle~\eqref{eq:oracle-definition} to reach an $\epsilon$-accurate dual solution is bounded above by
    \begin{align}
    \label{eq:epsilon-dual-complexity-dual-subgradient}
        O \left( \frac{N}{\epsilon^2}\right).
    \end{align}
    \textbf{Primal complexity:} the number of calls to the oracle~\eqref{eq:oracle-definition} to reach an $\epsilon$-accurate primal solution is bounded above by
    \begin{align}
    \label{eq:epsilon-primal-complexity-dual-subgradient}
        O \left( \frac{N}{\epsilon^2}\right).
    \end{align}
\end{prop}
\begin{proof}
    Applying Propositions~\ref{prop:dual-subgradient-dual-convergence} and~\ref{prop:dual-subgradient-primal-convergence} we have, in $O$-notation,
    \begin{align*}
    &d^* - d(\bar{\lambda}_T) \leq O\left(\frac{1}{\sqrt{T}}\right), \\
        &\norm{\frac{1}{\n} \sum_{i=1}^\n A_i \bar{x}_{i, T} - b}_+ \leq O\left(\frac{1}{\sqrt{T}}\right), \\
        &\frac{1}{\n} \sum_{i=1}^\n h_i(\bar{x}_{i, T}) \leq d^* + O\left( \frac{1}{\sqrt{T}}\right).
    \end{align*}
    In particular, we need $T \geq O(1/\epsilon^2)$ iterations to reach $\epsilon$-precision for all inequalities. Since each iteration requires $\n$ calls to oracle~\eqref{eq:oracle-definition}, this proves the claim.
\end{proof}

The complexity in terms of number of calls to the oracle~\eqref{eq:oracle-definition} scaling with $\n$ is prohibitive in many large-scale applications. The goal of this work is to achieve smaller complexity. We describe our approach on how to do so in the next section. 

\section{The two-stage algorithm}
\label{sec:proposed-algorithm}

We focus again on solving the original separable problem~\eqref{eq:primal}. We saw in the previous section that the dual subgradient algorithm requires $O(\frac{\n}{\epsilon^2})$ calls to oracle~\eqref{eq:oracle-definition} to reach an $\epsilon$-accurate solution of~\eqref{eq:primal}. Our goal is to improve on this complexity, and we do so through a two-stage method. The first stage consists in running a stochastic subgradient algorithm in order to obtain an approximation of the dual optimal value. Although this algorithm, just as its deterministic counterpart, produces primal points, we shall see that those points do not provide better candidates than in the deterministic case. However, the convergence rate for the dual is much better than in the deterministic setting. We detail this algorithm as well as its theoretical guarantees in Subsection~\ref{sec:stochastic-dual-subgradient}. Once the approximate dual optimal value is obtained, Subsection~\ref{sec:BCFW} focuses on obtaining primal points through a block-coordinate Frank-Wolfe algorithm.

\subsection{Stage 1: stochastic dual subgradient}
\label{sec:stochastic-dual-subgradient}

The stochastic dual subgradient algorithm for~\eqref{eq:primal} is a simple extension of the dual subgradient Algorithm~\ref{alg:dual-subgradient-separable-primal} where, at each iteration, instead of computing the oracle~\eqref{eq:oracle-definition} for all indices $i =1, \dots, \n$, we simply sample one index uniformly at random, compute the oracle only for that index, form the corresponding stochastic subgradient, and take a positive step in that direction.  We summarize the method in Algorithm~\ref{alg:stochastic-dual-subgradient-separable-primal}.

Let us mention one important difference here. As in the deterministic counterpart, we also obtain primal points through oracle~\eqref{eq:oracle-definition}, which we wish to combine to obtain a primal solution. Letting $I_i$ be the number of times we sample index $i$, a natural primal candidate $\bar{x}_{i, T}$ would be the average of the $I_i$ oracle outputs. The issue is that there might be some index $i$ which never gets sampled, i.e., $I_i = 0$. We circumvent this by taking a deterministic subgradient step on the last iteration, thus ensuring that we have at least one primal candidate for each index $i$.

\begin{algorithm}
\caption{Stochastic dual subgradient for problem~\eqref{eq:primal}}
\label{alg:stochastic-dual-subgradient-separable-primal}
\begin{algorithmic}[1]
\STATE{\textbf{Input}: $\lambda_0 \in \R^m$, number of iterations $T$.} 
\STATE{Set $I_i = 0$ for all $i=1, \dots, \n$.}
\FOR{$t=0, \dots, T-2$}
\STATE{Sample $i_t \in \{1, \dots, \n\}$ uniformly at random.}
\STATE{Update $I_{i_t} \leftarrow I_{i_t} + 1$.}
\STATE{Compute
\begin{align}
\label{eq:stochastic-dual-subgradient-oracle}
    x_{i_t}^*(\lambda_t) \in \argmin_{x_{i_t} \in X_{i_t}} \{ h_{i_t}(x_{i_t}) + \lambda_t^\top A_{i_t}x_{i_t}\}
\end{align}
}
\STATE{Set $g_t = A_{i_t} x_{i_t}^*(\lambda_t) - b$.}
\STATE{Pick some stepsize $\alpha_t > 0$.}
\STATE{Set $\lambda_{t+1} = \ProjPos{\lambda_t + \alpha_t g_t}$}
\ENDFOR
\STATE{\textbf{Last step:} For $i=1,\dots, \n$, compute
\begin{align*}
    x_{i}^*(\lambda_{T-1}) \in \argmin_{x_i \in \domhi} \{h_i(x_i) + \lambda_{T-1}^\top A_ix_i \}
\end{align*}
\STATE{Set $g_{T-1} = \frac{1}{\n} \sum_{i=1}^\n A_i x_{i}^*(\lambda_{T-1}) - b$.}
\STATE{Pick some stepsize $\alpha_{T-1} > 0$.}
\STATE{Set $\lambda_{T} = \ProjPos{\lambda_{T-1} + \alpha_{T-1} g_{T-1}}$}
}
\STATE{\textbf{Return:} Dual candidate $\bar{\lambda}_T = \frac{1}{T} \sum_{t=0}^{T-1} \lambda_t$ and primal candidate
\begin{align}
\label{eq:stochastic-dual-subgradient-primal-output}
    \bar{x}_{i, T} = \frac{1}{I_i + 1} \left( x_i^*(\lambda_{T-1}) +  \sum_{t\in \{0, \dots, T-2\}, \atop i_t = i} x_i^*(\lambda_t)\right), \ i=1,\dots, \n. 
\end{align}
}
\end{algorithmic}
\end{algorithm}

Before we move on to the convergence of this method, we define the following quantities of interest:
\begin{align*}
    \Gtilde &:= \sup_{i=1,\dots, \n, x_i\in \domhi} \norm{A_i x_i - b},\\
    H &:= \sup_{i=1,\dots, \n, x_i \in \domhi} \abs{h_i(x_i)}
\end{align*}

\begin{remark}
    Comparing with the definition of $G$ in~\eqref{eq:G-definition}, observe that $G$ and $\Gtilde$ have the same scale, and although we always have $G \leq \Gtilde$, we expect their values to be close in practical applications. In particular, when the matrices $A_i$ and the sets $X_i$ are identical, it holds that $G = \Gtilde$.
\end{remark}

\subsubsection{Dual convergence}

In this section we provide a full characterization of the dual convergence of Algorithm~\ref{alg:stochastic-dual-subgradient-separable-primal}. We emphasize that our analysis is not new and that the convergence of stochastic subgradient descent has been extensively studied (see for example~\cite{polyak1987introduction,nemirovski1978cezari,nemirovskij1983problem}), but we provide it for completeness. We denote $\Eone$ as the expectation with respect to the randomness induced by Algorithm~\ref{alg:stochastic-dual-subgradient-separable-primal}.

\begin{prop}
\label{prop:dual-convergence-stochastic-dual-subgradient}
    Suppose Assumption~\ref{ass:existence-dual-maximizer} holds. Consider Algorithm~\ref{alg:stochastic-dual-subgradient-separable-primal} applied to problem~\eqref{eq:primal} with input $\lambda_0 = 0_m$ and constant stepsize $\alpha_t = \alpha = \frac{\alphabar}{\Gtilde \sqrt{T}}$ for some $\alphabar > 0$. It then holds that
    \begin{align*}
        \Eone[d(\lambda^*) - d(\bar{\lambda}_T)] \leq \frac{\Gtilde \norm{ \lambda^*}^2}{2 \alphabar\sqrt{T}} + \frac{\Gtilde \alphabar}{2\sqrt{T}}.
    \end{align*}
\end{prop}
\begin{proof}
    For $t=0, \dots, T-2$, we have
    \begin{align*}
        \norm{\lambda_{t+1} - \lambda^*}^2 &=  \norm{\ProjPos{\lambda_{t} + \alpha g_t} - \ProjPos{\lambda^*}}^2 \\
        &\leq \norm{\lambda_{t} + \alpha g_t - \lambda^*}^2 \\
        &= \norm{\lambda_t - \lambda^*}^2 + 2\alpha \iprod{g_t}{\lambda_t - \lambda^*} + \alpha^2 \norm{g_t}^2 \\
        &\leq \norm{\lambda_t - \lambda^*}^2 + 2\alpha \iprod{g_t}{\lambda_t - \lambda^*} + \alpha_t^2 \Gtilde^2.
    \end{align*}
    Letting $\Eone^t$ be the expectation conditioned on the randomness up to iteration $t$, we have
    \begin{align*}
        \Eone^t[g_t] = \sum_{i=1}^\n \frac{1}{\n} A_i x_i^*(\lambda_t) - b \in \partial d(\lambda_t).
    \end{align*}
    Therefore we have, for $t=0, \dots, T-2$,
    \begin{align}
    \label{eq:proof-stochastic-dual-subgradient-dual-convergence-1}
        \Eone^t\norm{\lambda_{t+1} - \lambda^*}^2 &\leq \norm{\lambda_t - \lambda^*}^2 + 2\alpha (d(\lambda_t) - d(\lambda^*)) + \alpha^2 \Gtilde^2,
    \end{align}
    by concavity of $d$. Now, for the last iteration $t=T-1$, we have
    \begin{align*}
        \norm{\lambda_{T} - \lambda^*}^2 &=  \norm{\ProjPos{\lambda_{T-1} + \alpha g_{T-1}} - \ProjPos{\lambda^*}}^2 \\
        &\leq \norm{\lambda_{T-1} + \alpha g_{T-1} - \lambda^*}^2 \\
        &= \norm{\lambda_{T-1} - \lambda^*}^2 + 2\alpha \iprod{g_t}{\lambda_{T-1} - \lambda^*} + \alpha^2 \norm{g_{T-1}}^2 \\ 
        &\leq  \norm{\lambda_{T-1} - \lambda^*}^2 + 2\alpha (d(\lambda_{T-1}) - d(\lambda^*))  + \alpha^2 G^2,
    \end{align*}
    by concavity of $d$. Note that in that case we did not need to take expectation since $g_{T-1} \in \partial d(\lambda_{T-1})$. Since $G \leq \Gtilde$, combining with~\eqref{eq:proof-stochastic-dual-subgradient-dual-convergence-1}, we get that for any $t=0,\dots, T-1,$
    \begin{align*}
        \Eone^t\norm{\lambda_{t+1} - \lambda^*}^2 &\leq \norm{\lambda_t - \lambda^*}^2 + 2\alpha (d(\lambda_t) - d(\lambda^*)) + \alpha^2 \Gtilde^2.
    \end{align*}
    Rearranging and taking total expectation gives
    \begin{align*}
        \Eone[d(\lambda^*) - d(\lambda_t)] &\leq \frac{\Eone\norm{\lambda_t - \lambda^*}^2 - \Eone\norm{\lambda_{t+1} - \lambda^*}^2}{2\alpha} + \frac{\alpha \Gtilde^2}{2}.
    \end{align*}
    Concavity of $d$ then yields
    \begin{align*}
        \Eone[d(\lambda^*) - d(\bar{\lambda}_T)] &\leq \frac{1}{T} \sum_{t=0}^{T-1} \Eone[d(\lambda^*) - d(\lambda_t)] \\
        &\leq \frac{1}{2T} \sum_{t=0}^{T-1} \left( \frac{\Eone\norm{\lambda_t - \lambda^*}^2 - \Eone\norm{\lambda_{t+1} - \lambda^*}^2}{\alpha} + \alpha \Gtilde^2 \right) \\
        &\leq \frac{\norm{\lambda_0 - \lambda^*}^2}{2T\alpha} + \frac{\alpha \Gtilde^2}{2}.
    \end{align*}
    Plugging $\alpha = \frac{\alphabar}{\tilde{G}\sqrt{T}}$ and $\lambda_0 = 0_m$ yields the result.
\end{proof}

\subsubsection{Primal convergence}

We now turn our attention to the convergence of the primal candidate. To the best of our knowledge, this is the first result showing convergence of the primal iterates produced by a stochastic dual subgradient algorithm under our assumptions.

\begin{prop}
\label{prop:primal-convergence-stochastic-dual-subgradient}
    Suppose Assumption~\ref{ass:existence-dual-maximizer} holds. Consider Algorithm~\ref{alg:stochastic-dual-subgradient-separable-primal} applied to problem~\eqref{eq:primal} with input $\lambda_0 = 0_m$ and constant stepsize $\alpha_t = \alpha = \frac{\alphabar}{\Gtilde \sqrt{T}}$ for some $\alphabar > 0$. The primal output $\bar{x}_T$ then satisfies
    \begin{align*}
        \Eone \left[\norm{\frac{1}{\n} \sum_{i=1}^\n A_i \bar{x}_{i, T} - b }_+\right] \leq \Gtilde \sqrt{\frac{N-1}{T}} + \frac{ 2 \Gtilde\norm{\lambda^*}}{\alphabar\sqrt{T}} + \frac{\Gtilde}{\sqrt{T}} + \frac{\n \Gtilde}{T},
    \end{align*}
    and
    \begin{align*}
        \Eone \left[\frac{1}{\n} \sum_{i=1}^\n h_i(\bar{x}_{i, T}) \right] \leq d^* + H \sqrt{\frac{\n - 1}{T}} + \frac{\alphabar \Gtilde}{2\sqrt{T}} + \frac{\n \Gtilde(2\norm{\lambda^*}  + \alphabar)}{T}
    \end{align*}
\end{prop}
The above proposition is based on the following two lemmas. The first one is simply the analogue of Proposition~\ref{prop:dual-subgradient-bounded-iterates} to the stochastic setting, so we defer its proof to Appendix~\ref{sec:technical-results}. The second one bounds the deviation between the index counter $I_i+1$ and its expected value.

\begin{restatable}{lem}{SDSGDBoundIterates}
\label{lem:bound-growth-dual-iterates-stochastic-dual-subgradient}
    Consider the setting of Proposition~\ref{prop:primal-convergence-stochastic-dual-subgradient}. Then for all $t=0, \dots, T-1$,
    \begin{align*}
        \Eone \norm{\lambda_t - \lambda^*}^2 \leq \norm{\lambda_0 - \lambda^*} + t \Gtilde^2 \alpha^2.
    \end{align*}
\end{restatable}

\begin{lem}
\label{lem:concentration-bound-stochastic-dual-subgradient}
    Consider the setting of Proposition~\ref{prop:primal-convergence-stochastic-dual-subgradient}. Then for all $i=1, \dots, \n$,
    \begin{align*}
        \Eone \left[ \abs{\frac{1}{N} - \frac{ I_i+1}{T-1+\n}}\right] \leq \frac{1}{\n}\sqrt{\frac{N-1}{T}}.
    \end{align*}
\end{lem}
\begin{proof}
    First note that $\Eone[I_i + 1] = \Eone [I_i] + 1 = \frac{T-1}{\n} + 1 = \frac{T-1+\n}{\n}$. Then we have, by Jensen's inequality,
    \begin{align*}
        \Eone \left[\abs{1 - \frac{\n (I_i+1)}{T-1+\n}}\right]^2 &\leq \Eone \left[\left(1 - \frac{N(I_i+1)}{T-1+\n}\right)^2\right]  \\
        &= \text{Var}\left(1 - \frac{\n (I_i+1)}{T-1+\n}\right) + \Eone\left[1 - \frac{\n (I_i+1)}{T-1+\n}\right]^2.
    \end{align*}
    Since $\Eone [I_i+1] = \frac{T-1+\n}{\n}$, the second term is actually $0$. For the first term,
    \begin{align*}
        \text{Var}\left(1 - \frac{\n (I_i+1)}{T-1+\n}\right) &= \frac{N^2}{(T-1+\n)^2} \text{Var}(I_i) = \frac{N^2}{(T-1+\n)^2} (T-1) \frac{1}{N}(1 - \frac{1}{N}) \leq  \frac{N-1}{T}.
    \end{align*}
\end{proof}

We are now ready to start the proof of Proposition~\ref{prop:primal-convergence-stochastic-dual-subgradient}.

\begin{proof}
    Let us start by giving a sketch of the main idea of the proof. Recall that in the argument that made the deterministic proof work was the telescoping sums in equations~\eqref{eq:proof-dual-subgradient-telescoping-1} and~\eqref{eq:proof-dual-subgradient-telescoping-sum-2}. In the stochastic setting, we can only get similar inequalities if the inverse of $I_i+1$ is precisely equal to the inverse of its expected value, i.e., $\frac{1}{I_i + 1} = \frac{\n}{T - 1 + \n}$. The idea of the proof is therefore to separate both the infeasibility and suboptimality quantities in two terms. One term will depend on the `ideal' value $\frac{\n}{T - 1 + \n}$, and will lead to a telescoping sum similar to the deterministic setting. The other term will depend on the error between  $\frac{1}{I_i + 1}$ and $\frac{\n}{T - 1 + \n}$, which we will be able to bound using Lemma~\ref{lem:concentration-bound-stochastic-dual-subgradient}. We start with the infeasibility inequality. We have
\begin{align*}
    \frac{1}{\n}\sum_{i=1}^\n A_i \bar{x}_{i, T} - b =&\ \sum_{i=1}^\n \left(\frac{1}{\n (I_i + 1)} \left( (A_ix_i^*(\lambda_{T-1}) - b) +  \sum_{t\in \{0, \dots, T-2\}, \atop i_t = i} (A_i x_i^*(\lambda_t) - b) \right) \right)\\
    =& \ \underbrace{\sum_{i=1}^\n \left(\frac{1}{\n (I_i+1)} - \frac{1}{T-1+\n} \right) \left( (A_ix_i^*(\lambda_{T-1}) - b) +  \sum_{t\in \{0, \dots, T-2\}, \atop i_t = i} (A_i x_i^*(\lambda_t) - b) \right)}_{\Gamma_1} \\ &+ \underbrace{\frac{1}{T-1+\n} \sum_{i=1}^\n  \left( (A_ix_i^*(\lambda_{T-1}) - b) +  \sum_{t\in \{0, \dots, T-2\}, \atop i_t = i} (A_i x_i^*(\lambda_t) - b) \right)}_{\Gamma_2}
\end{align*}
For the first term, we have
\begin{align*}
    \Eone \norm{\Gamma_1} &\leq \Eone \left[\sum_{i=1}^\n \abs{\frac{1}{\n (I_i+1)} - \frac{1}{T-1+\n}} \left( \Gtilde +  \sum_{t\in \{0, \dots, T-2\}, \atop i_t = i}\Gtilde \right)\right]\\
    &= \Eone \left[\tilde{G} \sum_{i=1}^\n \abs{\frac{1}{\n} - \frac{I_i + 1}{T-1+\n}}\right] \leq \Gtilde \sqrt{\frac{\n-1}{T}} \tag{Lemma~\ref{lem:concentration-bound-stochastic-dual-subgradient}}.
\end{align*}
Let's look at the second term now. As in the proof of Proposition~\ref{prop:dual-subgradient-primal-convergence}, we let $\tilde{\lambda}_{t+1} = \lambda_t + \alpha g_t$ so that $\lambda_{t+1} = \ProjPos{\tilde{\lambda}_{t+1}}$. We have
\begin{align*}
    \norm{\Gamma_2}_{+} &\leq \frac{1}{T-1+\n} \left(\norm{\sum_{i=1}^\n (A_i x_i^*(\lambda_{T-1}) - b)}_+ + \norm{\sum_{t\in \{0, \dots, T-2\}, \atop i_t = i} \underbrace{(A_i x_i^*(\lambda_t) - b)}_{= \frac{\tilde{\lambda}_{t+1} - \lambda_t}{\alpha}}}_+\right) \\
    &\leq \frac{1}{ (T- 1+ \n)} \left(\n \Gtilde + \frac{1}{\alpha} \norm{\sum_{t=0}^{T-2} (\tilde{\lambda}_{t+1} - \lambda_t)}_+ \right).
\end{align*}
Now, observe that $\tilde{\lambda}_{t+1} \leq \lambda_{t+1}$, where the inequality is to be understood component-wise, and thus
\begin{align*}
    \sum_{t=0}^{T-2} (\tilde{\lambda}_{t+1} - \lambda_t) \leq \sum_{t=0}^{T-2} (\lambda_{t+1} - \lambda_t) = \lambda_{T-1} - \lambda_0 \leq \lambda_{T-1},
\end{align*}
since $\lambda_0 \geq 0$. Since $\lambda_{T-1} \geq 0$, we then get
\begin{align*}
    \norm{\sum_{t=0}^{T-2} (\tilde{\lambda}_{t+1} - \lambda_t)}_+ \leq \norm{\lambda_{T-1}}.
\end{align*}
Therefore,
\begin{align*}
    \Eone \norm{\Gamma_2}_+ &\leq 
    \frac{\n \Gtilde}{T-1+\n} + \frac{1}{\alpha(T-1+\n)} \Eone[\norm{\lambda_{T-1}}]\\
    &\leq \frac{\n \Gtilde}{T-1+\n} + \frac{1}{\alpha(T-1+\n)} \Eone[\norm{\lambda_{T-1} - \lambda^*} + \norm{\lambda^*}] \\
    &\leq \frac{N \Gtilde}{T-1 + \n} + \frac{1}{\alpha(T-1+\n)} \left(\norm{\lambda^*} + \sqrt{\Eone \norm{\lambda_{T-1} - \lambda^*}^2}\right).
    \end{align*}
Applying Lemma~\ref{lem:bound-growth-dual-iterates-stochastic-dual-subgradient}, we get
    \begin{align*}
    \Eone \norm{\Gamma_2}_+ &\leq \frac{N \Gtilde}{T-1 + \n} + \frac{1}{\alpha (T - 1 + \n)} \left( \norm{\lambda^*} + \sqrt{\norm{\lambda_0 - \lambda^*}^2 + (T-1) \Gtilde^2 \alpha^2}\right)  \\
    &\leq \frac{N \Gtilde}{T} + \frac{ \norm{\lambda_0 - \lambda^*} + \norm{\lambda^*}}{\alpha T} +\frac{\Gtilde}{\sqrt{T}},
\end{align*}
which proves the infeasibility inequality by plugging $\alpha = \alphabar/(\Gtilde\sqrt{T})$ and $\lambda_0 = 0_m$.

Now, for primal suboptimality, we have,
\begin{align}
\label{eq:proof-stochastic-dual-subgradient-primal-convergence-important-result}
\begin{split}
    \frac{1}{\n} \sum_{i=1}^\n h_i(\bar{x}_{i, T}) =& \ \frac{1}{\n} \sum_{i=1}^\n h_i\left( \frac{1}{I_i+1} \left( x_i^*(\lambda_{T-1}) +  \sum_{t\in \{0, \dots, T-2\}, \atop i_t = i} x_i^*(\lambda_t)\right)\right) \\
    \leq &\ \frac{1}{\n} \sum_{i=1}^\n \frac{1}{I_i+1}\left( h_i(x_i^*(\lambda_{T-1})) +  \sum_{t\in \{0, \dots, T-2\}, \atop i_t = i} h_i(x_i^*(\lambda_t))\right).
    \end{split}
\end{align}
Just like before, we now split the sum as
\begin{align}
\label{eq:SDSGD-opt-gap-bound1}
\begin{split}
    \frac{1}{\n} \sum_{i=1}^\n h_i(\bar{x}_{i, T}) \leq &\  \underbrace{\sum_{i=1}^\n \left(\frac{1}{\n (I_i+1)}- \frac{1}{T- 1 + \n} \right)\left( h_i(x_i^*(\lambda_{T-1})) +  \sum_{t\in \{0, \dots, T-2\}, \atop i_t = i} h_i(x_i^*(\lambda_t))\right)}_{=\Phi_1} \\
    & \ + \frac{1}{T- 1 + \n} \underbrace{\sum_{i=1}^\n  \left( h_i(x_i^*(\lambda_{T-1})) +  \sum_{t\in \{0, \dots, T-2\}, \atop i_t = i} h_i(x_i^*(\lambda_t))\right)}_{=\Phi_2}.
    \end{split}
\end{align}
For the first term, we have
\begin{align*}
    \Eone[\Phi_1] &= \Eone \left[\sum_{i=1}^\n \left(\frac{1}{\n (I_i+1)}- \frac{1}{T- 1 + \n} \right)\left( h_i(x_i^*(\lambda_{T-1})) +  \sum_{t\in \{0, \dots, T-2\}, \atop i_t = i} h_i(x_i^*(\lambda_t))\right)\right] \\ &\leq \Eone \left[\sum_{i=1}^\n \abs{\frac{1}{\n (I_i+1)}- \frac{1}{T- 1 + \n}}\left( H +  \sum_{t\in \{0, \dots, T-2\}, \atop i_t = i} H\right)\right]\\
    &= H\Eone \left[\sum_{i=1}^\n \abs{\frac{1}{\n }- \frac{I_i + 1}{T- 1 + \n}}\right] \\
    &\leq H \sqrt{\frac{\n -1}{T}}. \tag{Lem.~\ref{lem:concentration-bound-stochastic-dual-subgradient}} 
\end{align*}

For the second term, we have
\begin{align*}
\Phi_2 =& \ \sum_{i=1}^\n \left( h_i(x_i^*(\lambda_{T-1})) +  \sum_{t\in \{0, \dots, T-2\}, \atop i_t = i} h_i(x_i^*(\lambda_t))\right)\\ =& \ \sum_{i=1}^\n h_i(x_i^*(\lambda_{T-1})) +  \sum_{t=0}^{T-2} h_{i_t}(x_{i_t}^*(\lambda_t)) \\
=&\ \underbrace{\sum_{i=1}^\n \left(h_i(x_i^*(\lambda_{T-1})) + (\lambda_{T-1})^\top(A_ix_i^*(\lambda_{T-1}) - b) \right)}_{=\n d(\lambda_{T-1})} - \underbrace{\sum_{i=1}^\n \left( (\lambda_{T-1})^\top(A_ix_i^*(\lambda_{T-1}) - b) \right)}_{=\n (\lambda_{T-1})^\top g_{T-1}} \\ & + \sum_{t=0}^{T-2} h_{i_t}(x_{i_t}^*(\lambda_t)) + \lambda_t^\top (A_{i_t}x_{i_t}^*(\lambda_t) - b) - \sum_{t=0}^{T-2}  \lambda_t^\top (A_{i_t}x_{i_t}^*(\lambda_t) - b).
\end{align*}
Taking expectation we get
\begin{align*}
    \Eone [\Phi_2] =&\ \Eone \left[\n d(\lambda_{T-1}) -\n (\lambda_{T-1})^\top g_{T-1} \right]  \\
    & + \Eone \left[ \sum_{t=0}^{T-2} \underbrace{\frac{1}{\n}\sum_{i=1}^\n  h_{i}(x_{i}^*(\lambda_t)) + \lambda_t^\top (A_{i}x_{i}^*(\lambda_t) - b)}_{=d(\lambda_t)}\right] - \Eone \left[\sum_{t=0}^{T-2}  \lambda_t^\top (A_{i_t}x_{i_t}^*(\lambda_t) - b) \right] \\
    =& \ \Eone \left[\n d(\lambda_{T-1}) -\n (\lambda_{T-1})^\top g_{T-1} + \sum_{t=0}^{T-2} d(\lambda_t)\right] - \Eone \left[\sum_{t=0}^{T-2}  \lambda_t^\top (A_{i_t}x_{i_t}^*(\lambda_t) - b) \right].
\end{align*}
Now, observe that for $t=0,\dots, T-2$,
\begin{align*}
    \norm{\lambda_{t+1}}^2 &=  \norm{\ProjPos{\lambda_t + \alpha (A_{i_t}x_{i_t}^*(\lambda_t) - b )}}^2  \\
    &\leq \norm{\lambda_t + \alpha (A_{i_t}x_{i_t}^*(\lambda_t) - b )}^2  \\
    &= \norm{\lambda_t}^2 + 2\alpha \lambda_t^\top (A_{i_t}x_{i_t}^*(\lambda_t) - b) + \alpha^2\norm{A_{i_t}x_{i_t}^*(\lambda_t) - b }^2.
\end{align*}
Therefore,
\begin{align*}
    -\lambda_t^\top (A_{i_t}x_{i_t}^*(\lambda_t) - b) &\leq \frac{\norm{\lambda_t}^2 - \norm{\lambda_{t+1}}^2}{2\alpha} + \frac{\alpha \Gtilde^2}{2}.
\end{align*}
We thus get
\begin{align*}
    \Eone [\Phi_2] &\leq \Eone \left[\n d(\lambda_{T-1}) -\n (\lambda_{T-1})^\top g_{T-1} + \sum_{t=0}^{T-2} d(\lambda_t)\right] + \Eone \left[\sum_{t=0}^{T-2} \frac{\norm{\lambda_t}^2 - \norm{\lambda_{t+1}}^2}{2\alpha} + \frac{\alpha \Gtilde^2}{2} \right] \\
    &\leq \Eone \left[\n d(\lambda_{T-1}) -\n (\lambda_{T-1})^\top g_{T-1} + \sum_{t=0}^{T-2} d(\lambda_t)\right] + \frac{\norm{\lambda_0}^2}{2\alpha} + \frac{\alpha (T-1) \Gtilde^2}{2}.
\end{align*}
Now,
\begin{align*}
    \Eone [- (\lambda_{T-1})^\top g_{T-1}] &\leq \Eone [\norm{\lambda_{T-1}} \norm{g_{T-1}}] \\
    &\leq \Gtilde \Eone [\norm{\lambda_{T-1} - \lambda^*} + \norm{\lambda^*}] \\
    &\leq \Gtilde \norm{\lambda^*} + \Gtilde \sqrt{\Eone \norm{\lambda_{T-1} - \lambda^*}^2} \\
    &\leq \Gtilde \norm{\lambda^*} + \Gtilde \sqrt{\norm{\lambda_0 - \lambda^*}^2 + (T-1) \Gtilde^2 \alpha^2} \tag{Lem. \ref{lem:bound-growth-dual-iterates-stochastic-dual-subgradient}}\\
    &\leq \Gtilde(\norm{\lambda^*} + \norm{\lambda_0 - \lambda^*}) + \alpha \Gtilde^2 \sqrt{T-1}.
 \end{align*}

Therefore, 
\begin{align*}
    \Eone [\Phi_2] &\leq \Eone \left[\n d(\lambda_{T-1}) + \sum_{t=0}^{T-2} d(\lambda_t)\right] +\n\Gtilde\left(\norm{\lambda^*} + \norm{\lambda_0 - \lambda^*}\right) + \alpha \n \Gtilde^2 \sqrt{T-1} + \frac{\norm{\lambda_0}^2}{2\alpha} + \frac{\alpha (T-1) \Gtilde^2}{2}
\end{align*}
Coming back to~\eqref{eq:SDSGD-opt-gap-bound1}, recall that
\begin{align}
    \Eone \left[\frac{1}{\n} \sum_{i=1}^\n h_i(\bar{x}_{i, T}) \right] &\leq \Eone [\Phi_1] + \frac{1}{T-1 + \n} \Eone [\Phi_2]\nonumber \\
    &\leq H \sqrt{\frac{\n - 1}{T}} + \frac{1}{T-1 + \n} \Eone [\Phi_2] \label{eq:proof-stochastic-dual-subgradient-primal-convergence-1}.
\end{align}
Finally, 
\begin{align*}
    \frac{1}{T-1 + \n} \Eone [\Phi_2] \leq & \ \frac{1}{T-1 + \n}\Eone \left[\n d(\lambda_{T-1}) + \sum_{t=0}^{T-2} d(\lambda_t)\right] \\ &+ \frac{1}{T-1+\n}\left(\n\Gtilde(\norm{\lambda^*} + \norm{\lambda_0 - \lambda^*}) + \alpha \n \Gtilde^2 \sqrt{T-1} + \frac{\norm{\lambda_0}^2}{2\alpha} + \frac{\alpha (T-1) \Gtilde^2}{2} \right)\\
    \leq &\ \frac{1}{T-1 + \n}\Eone \left[\n d^* + \sum_{t=0}^{T-2} d^*\right]
    \\ &+    \frac{1}{T}\left(\n\Gtilde(\norm{\lambda^*} + \norm{\lambda_0 - \lambda^*}) + \alpha \n \Gtilde^2 \sqrt{T} + \frac{\norm{\lambda_0}^2}{2\alpha} + \frac{\alpha T \Gtilde^2}{2} \right) \\
    = & \  d^* + \frac{\n \Gtilde(\norm{\lambda^*} + \norm{\lambda_0 - \lambda^*})}{T} + \frac{\alpha\n \Gtilde^2}{\sqrt{T}} + \frac{\norm{\lambda_0}^2}{2\alpha T} + \frac{\alpha \Gtilde^2}{2}.
\end{align*}
Going back to~\eqref{eq:proof-stochastic-dual-subgradient-primal-convergence-1}, we get
\begin{align*}
    \Eone \left[\frac{1}{\n} \sum_{i=1}^\n h_i(\bar{x}_{i, T}) \right] \leq d^* + \frac{\n \Gtilde(\norm{\lambda^*} + \norm{\lambda_0 - \lambda^*})}{T} + \frac{\alpha \n \Gtilde^2}{\sqrt{T}} + \frac{\norm{\lambda_0}^2}{2\alpha T} + \frac{\alpha \Gtilde^2}{2} +  H \sqrt{\frac{\n - 1}{T}}.
\end{align*}

Plugging $\alpha = \frac{\alphabar}{\tilde{G}\sqrt{T}}$ and $\lambda_0 = 0_m$ yields the result.
\end{proof}

\subsubsection{Complexity analysis}

We can now derive the complexity in terms of number of oracle calls needed to reach $\epsilon$-precision.

\begin{prop}
\label{prop:epsilon-complexity-stochastic-dual-subgradient}
    Suppose Assumption~\ref{ass:existence-dual-maximizer} holds. Consider Algorithm~\ref{alg:stochastic-dual-subgradient-separable-primal} applied to problem~\eqref{eq:primal} with constant stepsize $\alpha_t = \alpha = \frac{\alphabar}{\Gtilde \sqrt{T}}$.
    
    \textbf{Dual complexity:} the number of calls to oracle~\eqref{eq:oracle-definition} to reach an $\epsilon$-accurate dual solution in expectation is bounded above by
    \begin{align}
        O\left(\n + \frac{1}{\epsilon^2}\right).
    \end{align}
    \textbf{Primal complexity:} the number of calls to oracle~\eqref{eq:oracle-definition} to reach an $\epsilon$-accurate primal solution in expectation is bounded above by
    \begin{align}
        O\left(\frac{N}{\epsilon^2} \right).
    \end{align}
\end{prop}
\begin{proof}
For the dual complexity, from Proposition~\ref{prop:dual-convergence-stochastic-dual-subgradient} we have
\begin{align*}
    \Eone [d(\lambda^*) - d(\lambda_t)] \leq O\left(\frac{1}{\sqrt{T}}\right),
\end{align*}
so we need $T = O\left(1/\epsilon^2\right)$ iterations to achieve an $\epsilon$-accurate dual solution in expectation. Since each iteration requires only one call to oracle~\eqref{eq:oracle-definition}, except for the last one which requires $\n$ calls, we get the result.

    For the primal complexity, from Proposition~\ref{prop:primal-convergence-stochastic-dual-subgradient} we have
    \begin{align*}
   &\Eone  \left[\norm{\frac{1}{\n} \sum_{i=1}^\n A_i \bar{x}_{i, T} - b }_+\right] \leq O\left( \sqrt{\frac{N}{T}} + \frac{\n}{T}\right), \\
        &\Eone \left[\frac{1}{\n} \sum_{i=1}^\n h_i(\bar{x}_{i, T}) -d^* \right] \leq  O\left( \frac{1}{\sqrt{T}} + \sqrt{\frac{N}{T}} + \frac{\n }{T}\right).
    \end{align*}
    In order to reach $\epsilon$-precision, we thus need to set 
    \begin{align*}
        T = O \left( \max \left( \frac{\n }{\epsilon^2}, \frac{\n}{\epsilon}\right)\right) = O\left(\frac{N}{\epsilon^2} \right).
    \end{align*}
    Since each iteration of Algorithm~\ref{alg:stochastic-dual-subgradient-separable-primal} requires only one call to oracle~\eqref{eq:oracle-definition}, except for the last one which requires $\n$ calls, we get a complexity $O\left(\n + \frac{N}{\epsilon^2} \right) = O\left(\frac{N}{\epsilon^2} \right)$ in terms of number of calls to~\eqref{eq:oracle-definition}.
\end{proof}

Comparing with the complexity of the deterministic dual subgradient algorithm (Algorithm~\ref{alg:dual-subgradient-separable-primal}) derived in Proposition~\ref{prop:epsilon-complexity-dual-subgradient}, there is no theoretical advantage in going with the stochastic version of the algorithm to obtain good primal solutions. We see however that the complexity to solve the dual is improved to $O(1/\epsilon^2)$ instead of $O(\n/\epsilon^2)$ for the deterministic algorithm.

This is one crucial aspect of our two-stage approach: we are able to `quickly' get an approximation of the dual optimal value. Once we have this value, we are ready to start the second stage of our method, which we present in Section~\ref{sec:BCFW}.

\begin{remark}
The additional $\n$ term in the dual complexity $O(\n + 1/\epsilon^2)$ is due to the last iteration of Algorithm~\ref{alg:stochastic-dual-subgradient-separable-primal}, where we take a deterministic subgradient step. Recall that we only do so in order to guarantee well-definedness of the primal candidate solution. If we took a stochastic step instead, the proof for dual convergence would still hold and we would get a complexity $O(1/\epsilon^2)$.
\end{remark}

\subsection{Stage 2: block-coordinate Frank-Wolfe}
\label{sec:BCFW}

After running the stochastic dual subgradient Algorithm~\ref{alg:stochastic-dual-subgradient-separable-primal}, we obtain a dual candidate $\bar{\lambda}_T$ and, more importantly in our case, an approximation of the optimal dual value $d(\bar{\lambda}_T)$.

Let us start with a simple observation, which follows the approach from~\cite{dubois2025frank}. Since $d^*$ is the optimal dual value and because strong duality holds by Assumption~\ref{ass:existence-dual-maximizer}, we have
\begin{align*}
    \begin{pmatrix}
        d^* \\ b
    \end{pmatrix} \in \frac{1}{\n} \sum_{i=1}^\n \left\{ \begin{pmatrix} h_i(x_i) \\ A_i x_i \end{pmatrix} \mid x_i \in \domhi \right\} + \begin{pmatrix}
        0 \\ \R^m_+
    \end{pmatrix}.
\end{align*}
Notably, the elements $x_i$ achieving the above inclusion are optimal for the primal problem. Let us then define the corresponding sets
\begin{align}
\label{eq:definition-Ci-sets}
    C_i := \left\{ \begin{pmatrix} h_i(x_i) \\ A_i x_i \end{pmatrix} \mid x_i \in \domhi \right\}, \ i=1, \dots, \n.
\end{align}
In~\cite{dubois2025frank}, the authors suggest minimizing the function $\norm{\begin{pmatrix}
    \beta \\ z
\end{pmatrix} - \begin{pmatrix}
        d^* \\ b
    \end{pmatrix}}^2_+$ over $(\beta, z) \in \frac{1}{\n} \sum_{i=1}^\n C_i$ using a Frank-Wolfe algorithm, obtaining primal points as a by-product of the linear minimization step that must be computed at each iteration. The important thing here is that, as we will see later, computing a linear minimization over $C_i$ is equivalent to making a call to oracle~\eqref{eq:oracle-definition}.
    We follow the same approach, with two important modifications:
    \begin{enumerate}
    \item As we do not have access to the optimal value $d^*$, we replace it with the output $d(\bar{\lambda}_T)$ of Algorithm~\ref{alg:stochastic-dual-subgradient-separable-primal}.
    \item Each iteration of the Frank-Wolfe algorithm from~\cite{dubois2025frank} requires $\n$ calls to the oracle~\eqref{eq:oracle-definition}, i.e., one for each of the $\n$ agents. This would defeat the purpose of this work which is to reduce the dependency on $\n$. Instead we propose to replace the classical Frank-Wolfe algorithm with its block-coordinate version from~\cite{lacoste2013block}, which only makes one call to oracle~\eqref{eq:oracle-definition} per iteration.
    \end{enumerate}

    Let us define the following function
    \begin{align}
    \label{eq:F-definition}
        F : \R^{m+1} &\rightarrow \R, \\
        \begin{pmatrix}\beta\\ z\end{pmatrix} &\rightarrow \frac{1}{2} \max(\beta - d(\bar{\lambda}_T), 0)^2 + \frac{1}{2} \norm{z - b}_+^2,
    \end{align}
    with $\beta \in \R$ and $z \in \R^m$. We then consider the following optimization problem
    \begin{align}
    \begin{split}
        \min_{\beta, z} &\ F\left((\beta, z)\right)\\
        \text{subject to } &\ \begin{pmatrix}
            \beta \\ z
        \end{pmatrix} \in \frac{1}{\n}\sum_{i=1}^\n C_i.
        \end{split}
        \label{eq:second-stage-problem-definition}
    \end{align}
    We write $F^*$ as the optimal value of~\eqref{eq:second-stage-problem-definition}, and bound it as follows.
    \begin{lem}
    \label{lem:second-stage-optimal-value-bound}
        It holds that
        \begin{align*}
            F^* \leq \frac{1}{2} (d^* - d(\bar{\lambda}_T))^2.
        \end{align*}
    \end{lem}
    \begin{proof}
        This is clear from the already mentioned fact that
        \begin{align*}
            \begin{pmatrix}
                d^* \\ b
            \end{pmatrix} \in \frac{1}{\n}\sum_{i=1}^\n C_i + \begin{pmatrix}
                0 \\ \R^m_+
            \end{pmatrix},
        \end{align*}
        implying that there exists $(\beta^*, z^*) \in \frac{1}{\n}\sum_{i=1}^\n C_i$ such that $d^* = \beta^*$ and $z^* \leq b$. In particular,
        \begin{align*}
            F^* \leq F((\beta^*, z^*)) = \frac{1}{2} \max(\beta^* - d(\bar{\lambda}_T), 0)^2 + \frac{1}{2} \norm{z^* - b}^2_+ = \frac{1}{2}(d^* - d(\bar{\lambda}_T))^2.
        \end{align*}
    \end{proof}

We use the block-coordinate Frank-Wolfe (BCFW) algorithm from~\cite{lacoste2013block} to solve~\eqref{eq:second-stage-problem-definition}. For all $i=1,\dots, \n$, the algorithm starts from some $x_i^0 \in \domhi$ and forms $\beta^0 = \sum_{i=1}^\n \beta^0_i$ and $z^0 = \sum_{i=1}^\n z^0_i$ where
\begin{align*}
    \beta^0_i = \frac{1}{\n} h_i(x_i^0) \ \text{ and }\ 
    z^0_i = \frac{1}{\n} A_i x_i^0.
\end{align*}
Observe then that $(\beta^0, z^0) \in \sum_{i=1}^\n C_i$. At each iteration $k$, the gradient $(\gamma^k, g^k) \in \R_+ \times \R^m_+$ of $F$ is computed as
\begin{align*}
    \gamma^k &= \max(\beta^k - d(\bar{\lambda}_T), 0) \ \text{ and }\ 
    g^k = \max(z^k - b, 0_m),
\end{align*}
where the maximum is taken component wise. In the classical Frank-Wolfe algorithm, this gradient is used to minimize the first-order approximation of $F$ around $(\beta^k, z^k)$, i.e., it computes
\begin{align*}
    (\eta^k, s^k) \in \argmin_{(\eta, s) \in \sum_{i=1}^\n C_i} \eta \gamma^k + s^\top g^k.
\end{align*}
This comes down to solving, for each $i=1, \dots, \n$,
\begin{align}
\label{eq:FW-second-stage-oracle}
    x_i^k = x_i^*((\gamma^k, g^k)) \in \argmin_{x_i \in \domhi} \left\{ \gamma^k h_i(x_i) + (g^k)^\top A_i x_i\right\},
\end{align}
and setting
\begin{align*}
    \eta^k = \frac{1}{\n} \sum_{i=1}^n h_i(x_i^k) \  \text{ and } \ s^k = \frac{1}{\n} \sum_{i=1}^\n A_i x_i^k,
\end{align*}
and then updating $(\beta^k, z^k)$ as a convex combination of itself and $(\eta^k, s^k$).
Importantly here, we see that once again the main computational bottleneck is oracle~\eqref{eq:oracle-definition}. We also see that, as in the deterministic dual subgradient, each iteration of the Frank-Wolfe algorithm requires $\n$ calls to oracle~\eqref{eq:oracle-definition}. We thus resort instead to the block-coordinate version, where at each iteration the algorithm picks an index $i_k \in \{1, \dots, \n\}$ uniformly at random, computes~\eqref{eq:FW-second-stage-oracle} only for that index, and updates only that index in the next iterate. We summarize the full method in Algorithm~\ref{alg:BCFW}. Note that the sets $C_i$ are not necessarily convex, and thus it might be that the iterates of Algorithm~\ref{alg:BCFW} are not in the feasible set. We shall see that this however not an issue, as we are only truly interested in the final output $\hat{x}^K = (\hat{x}_1^K, \dots, x_\n^K)$, and we know that $\hat{x}_i^K \in \domhi$ for all $i$ by convexity of $\domhi$.

\begin{algorithm}
\caption{Block coordinate Frank-Wolfe}
\label{alg:BCFW}
\begin{algorithmic}[1]
\STATE{\textbf{Input}: $d(\bar{\lambda}_T)$ and $x_i^0 \in \domhi$ for all $i=1, \dots, \n$, number of iterations $K$.} 
\STATE{Compute $\beta_i^0 = \frac{1}{\n} h_i(x_i^0)$ and $z_i^0 = \frac{1}{\n} A_i x_i^0$ for all $i=1, \dots, \n$.}
\STATE{Set $\beta^0 = \sum_{i=1}^\n \beta_i^0$ and $z^0 = \sum_{i=1}^\n z^0_i$.}
\STATE{Set current estimate $\hat{x}_i^0 = x_i^0$ for all $i=1,\dots, \n$.}
\FOR{$k=0, \dots, K-1$}
\STATE{Compute the gradient $(\gamma^k, g^k)$ of $F$:
\begin{align*}
    &\gamma^k = \max(\beta^k - d(\bar{\lambda}_T), 0) \\
    &g^k = \max(z^k - b, 0_m)
\end{align*}
}
\STATE{Sample $i_k \in \{1, \dots, \n\}$ uniformly at random.}
\STATE{Compute
\begin{align}
\label{eq:bcfw-oracle}
    x_{i_k}^{k+1} = 
        \argmin_{x_{i_k} \in X_{i_k}} \{ \gamma^k h_{i_k}(x_{i_k}) + (g^k)^\top A_{i_k}x_{i_k} \}
\end{align}
}
\STATE{Pick some stepsize $\rho_k \in [0, 1]$ and set
\begin{align}
\label{eq:BCFW-update-beta-z}
    \begin{pmatrix}\beta^{k+1}_i \\ z_i^{k+1}\end{pmatrix} = \begin{cases} (1- \rho_k) \begin{pmatrix}\beta^k_i \\ z^k_i \end{pmatrix} + \rho_k \frac{1}{\n} \begin{pmatrix}h_i(x_{i}^{k+1}) \\A_i x_i^{k+1} \end{pmatrix} &\text{ if } i = i_k, \\
        \begin{pmatrix}\beta_i^k \\ z_i^k \end{pmatrix} &\text{ otherwise.}
    \end{cases}
\end{align}
and
\begin{align}
\label{eq:BCFW-update-x-hat}
    \hat{x}_{i}^{k+1} = \begin{cases}(1- \rho_k) \hat{x}_{i}^k + \rho_k x_{i}^{k+1} &\text{ if } i = i_k, \\ 
    \hat{x}_i^{k} &\text{ otherwise.}
    \end{cases}
\end{align}
}
\ENDFOR
\STATE{\textbf{Return:} $(\beta^K, z^K)$ and $\hat{x}^K = (\hat{x}^{K}_1, \dots, \hat{x}_\n^K)$.
}
\end{algorithmic}
\end{algorithm}

Before we move on to the analysis of the algorithm, we define the constant $\AvgDiamSquared$ as the average of the squared diameter of the convex hull of the sets $C_i$, i.e.,
\begin{align*}
    \AvgDiamSquared &:= \frac{1}{\n} \sum_{i=1}^\n \diam(\conv(C_i))^2.
\end{align*}

The next theorem states the convergence result of Algorithm~\ref{alg:BCFW}. The result is almost identical to~\cite[Theorem 2]{lacoste2013block}, with a slightly better dependence on the initialization error. We denote $\Etwo$ as the expectation of Algorithm~\ref{alg:BCFW} conditioned on its input, and in particular conditioned on the value of $d(\bar{\lambda}_T)$.

\begin{thm}
\label{thm:BCFW-convergence}
    Consider Algorithm~\ref{alg:BCFW} with stepsize $\rho_k = \frac{2\n}{k + 2\n}$. After $K$ iterations, it holds that
    \begin{align*}
        \Etwo [F((\beta^K, z^K)) - F^*] \leq \frac{4\AvgDiamSquared}{K } + \frac{4\n^2}{K^2}(F((\beta^0, z^0)) - F^*).
    \end{align*}
\end{thm}
\begin{proof}
    Observe that the objective function $F$ is $1$-smooth, implying that for all $k=0, \dots, K-1$,
    \begin{align*}
        F((\beta^{k+1}, z^{k+1})) &\leq F((\beta^k, z^k)) + \iprod{\nabla F((\beta^k, z^k))}{\begin{pmatrix}\beta^{k+1}\\ z^{k+1}\end{pmatrix} - \begin{pmatrix}\beta^k \\  z^k\end{pmatrix}} + \frac{1}{2}\norm{\begin{pmatrix}\beta^{k+1}\\ z^{k+1}\end{pmatrix} - \begin{pmatrix}\beta^k \\  z^k\end{pmatrix}}^2 \\
        &= F((\beta^k, z^k)) + \iprod{\begin{pmatrix}
            \gamma^k \\ g^k
        \end{pmatrix}}{\begin{pmatrix}\beta^{k+1}\\ z^{k+1}\end{pmatrix} - \begin{pmatrix}\beta^k \\  z^k\end{pmatrix}} + \frac{1}{2}\norm{\begin{pmatrix}\beta^{k+1}\\ z^{k+1}\end{pmatrix} - \begin{pmatrix}\beta^k \\  z^k\end{pmatrix}}^2.
    \end{align*}
    Now observe that
    \begin{align*}
        \begin{pmatrix}\beta^{k+1}\\ z^{k+1}\end{pmatrix} - \begin{pmatrix}\beta^k \\  z^k\end{pmatrix} &= \rho_k \begin{pmatrix}
        \frac{1}{\n}h_{i_k}(x_{i_k}^{k+1}) - \beta_{i_k}^{k} \\  
        \frac{1}{\n}A_{i_k}x_{i_k}^{k+1} - z^k_{i_k}\end{pmatrix}. 
    \end{align*}
Since $\begin{pmatrix}
    \beta^k_{i_k} \\ z^k_{i_k}
\end{pmatrix} \in \frac{1}{\n} \conv(C_{i_k})$, we have
\begin{align*}
    \norm{\begin{pmatrix}
        \frac{1}{\n}h_{i_k}(x_{i_k}^{k+1}) - \beta_{i_k}^{k} \\  
        \frac{1}{\n}A_{i_k}x_{i_k}^{k+1} - z^k_{i_k}\end{pmatrix}}^2 &\leq \frac{1}{\n^2} \diam(\conv(C_{i_k}))^2.
\end{align*}
Letting $\Etwo^k$ be the expectation conditioned on the randomness up to iteration $k$, we then get
\begin{align}
    \Etwo^k[F((\beta^{k+1}, z^{k+1}))] 
        &\leq F((\beta^k, z^k)) + \rho_k \Etwo^k\left[\iprod{\begin{pmatrix}
            \gamma^k \\ g^k
        \end{pmatrix}}{\begin{pmatrix}
        \frac{1}{\n}h_{i_k}(x_{i_k}^{k+1}) - \beta_{i_k}^{k} \\  
        \frac{1}{\n}A_{i_k}x_{i_k}^{k+1} - z^k_{i_k}\end{pmatrix}}\right] + \frac{\rho_k^2}{\n^2} \Etwo^k[\diam(\conv(C_{i_k}))^2]  \nonumber \\
        &= F((\beta^k, z^k)) + \rho_k \Etwo^k\left[\iprod{\begin{pmatrix}
            \gamma^k \\ g^k
        \end{pmatrix}}{\begin{pmatrix}
        \frac{1}{\n}h_{i_k}(x_{i_k}^{k+1}) - \beta_{i_k}^{k} \\  
        \frac{1}{\n}A_{i_k}x_{i_k}^{k+1} - z^k_{i_k}\end{pmatrix}}\right] + \frac{\rho_k^2\AvgDiamSquared}{\n^2}.
        \label{eq:proof-bcfw-1}
\end{align}
We now claim that
\begin{align}
    \label{eq:proof-bcfw-3}\Etwo^k\left[\iprod{\begin{pmatrix}
            \gamma^k \\ g^k
        \end{pmatrix}}{\begin{pmatrix}
        \frac{1}{\n}h_{i_k}(x_{i_k}^{k+1}) - \beta_{i_k}^{k} \\  
        \frac{1}{\n}A_{i_k}x_{i_k}^{k+1} - z^k_{i_k}\end{pmatrix}}\right] \leq \frac{1}{\n} (F^* - F((\beta^k, z^k))).
\end{align}
To do so, recall that $x_{i_k}^{k+1}$ is obtained through solving the oracle~\eqref{eq:bcfw-oracle}, so that
\begin{align*}
    x_{i_k}^{k+1} = x_{i_k}^*(\gamma^k, g^k) = \argmin_{x_{i_k} \in X_{i_k}} \{ \gamma^k h_{i_k}(x_{i_k}) + (g^k)^\top A_{i_k}x_{i_k} \}.
\end{align*}
Therefore, 
\begin{align}
    \Etwo^k\left[\iprod{\begin{pmatrix}
            \gamma^k \\ g^k
        \end{pmatrix}}{\begin{pmatrix}
        \frac{1}{\n}h_{i_k}(x_{i_k}^{k+1}) - \beta_{i_k}^{k} \\  
        \frac{1}{\n}A_{i_k}x_{i_k}^{k+1} - z^k_{i_k}\end{pmatrix}}\right] &=  \iprod{\begin{pmatrix}
            \gamma^k \\ g^k
        \end{pmatrix}}{\sum_{i=1}^\n \frac{1}{\n} \begin{pmatrix}
        \frac{1}{\n}h_i(x_i^*(\gamma^k, g^k)) - \beta^k_i \\ \frac{1}{\n}A_i x_i^*(\gamma^k, g^k) - z^k_i \\
    \end{pmatrix}} \nonumber \\
    &= \frac{1}{\n} \iprod{\begin{pmatrix}
            \gamma^k \\ g^k
        \end{pmatrix}}{\sum_{i=1}^\n  \begin{pmatrix}
        \frac{1}{\n}h_i(x_i^*(\gamma^k, g^k)) \\ \frac{1}{\n}A_i x_i^*(\gamma^k, g^k) \\
    \end{pmatrix}  - \begin{pmatrix}
        \beta^k \\ z^k
    \end{pmatrix}}. \label{eq:proof-bcfw-2}
\end{align}
Now, let $(\beta^*, z^*)$ be the minimizer of $F$ over $\frac{1}{\n} \sum_i \conv(C_i)$. Note that by the relationship
\begin{align*}
    x_{i}^*(\gamma^k, g^k) = \argmin_{x_{i} \in X_{i}} \{ \gamma^k h_{i}(x_{i}) + (g^k)^\top A_{i}x_{i} \},
\end{align*}
we actually have that
\begin{align*}
    \sum_{i=1}^\n  \begin{pmatrix}
        \frac{1}{\n}h_i(x_i^*(\gamma^k, g^k)) \\ \frac{1}{\n}A_i x_i^*(\gamma^k, g^k) \\
    \end{pmatrix} \in \argmin_{(\beta, z) \in \sum_{i=1}^\n C_i} \iprod{\begin{pmatrix}
        \gamma^k \\ g^k
    \end{pmatrix}}{\begin{pmatrix}
        \beta \\ z
    \end{pmatrix}} \subset   \argmin_{(\beta, z) \in \sum_{i=1}^\n \conv(C_i)} \iprod{\begin{pmatrix}
        \gamma^k \\ g^k
    \end{pmatrix}}{\begin{pmatrix}
        \beta \\ z
    \end{pmatrix}},
\end{align*}
by linearity. Therefore,
\begin{align*}
    \iprod{\begin{pmatrix}
            \gamma^k \\ g^k
        \end{pmatrix}}{\sum_{i=1}^\n  \begin{pmatrix}
        \frac{1}{\n}h_i(x_i^*(\gamma^k, g^k)) \\ \frac{1}{\n}A_i x_i^*(\gamma^k, g^k) \\
    \end{pmatrix}} \leq \iprod{\begin{pmatrix}
            \gamma^k \\ g^k
        \end{pmatrix}}{   \begin{pmatrix}
        \beta^* \\ z^*\\
    \end{pmatrix}}.
\end{align*}
Going back to~\eqref{eq:proof-bcfw-2}, 
\begin{align*}
    \Etwo^k\left[\iprod{\begin{pmatrix}
            \gamma^k \\ g^k
        \end{pmatrix}}{\begin{pmatrix}
        \frac{1}{\n}h_{i_k}(x_{i_k}^{k+1}) - \beta_{i_k}^{k} \\  
        \frac{1}{\n}A_{i_k}x_{i_k}^{k+1} - z^k_{i_k}\end{pmatrix}}\right]  &\leq \frac{1}{\n}  \iprod{\begin{pmatrix}
            \gamma^k \\ g^k
        \end{pmatrix}}{   \begin{pmatrix}
        \beta^* \\ z^*\\
    \end{pmatrix} - \begin{pmatrix}
        \beta^k \\ z^k
    \end{pmatrix} } \\
    &\leq \frac{1}{\n} \left( F^* - F(\beta^k, z^k)\right),
\end{align*}
by first-order characterization of the convexity of $F$, thus proving~\eqref{eq:proof-bcfw-3}. Plugging this back into~\eqref{eq:proof-bcfw-1}, we get
\begin{align*}
    \Etwo^k[F((\beta^{k+1}, z^{k+1}))] 
        &\leq F((\beta^k, z^k)) - \frac{\rho_k}{\n}\left( F(\beta^k, z^k) - F^*\right)  + \frac{\rho_k^2\AvgDiamSquared}{\n^2}.
\end{align*}
Letting $\delta_k = \Etwo[F(\beta^k, z^k) - F^*]$, taking total expectation gives
\begin{align*}
    \delta_{k+1} \leq (1- \frac{\rho_k}{\n}) \delta_k + \frac{\rho_k^2 \AvgDiamSquared}{\n^2}.
\end{align*}
Unrolling gives
\begin{align*}
    \delta_K \leq \delta_0 \prod_{k=0}^{K-1}(1- \frac{\rho_k}{\n}) + \frac{\AvgDiamSquared}{\n^2} \sum_{k=0}^{K-1} \rho_k^2 \prod_{j=k+1}^{K-1} (1 - \frac{\rho_j}{\n}).
\end{align*}
Recalling that $\rho_k = \frac{2\n }{k + 2\n}$, we have
\begin{align*}
    \prod_{k=0}^{K-1}(1- \frac{\rho_k}{\n}) &= \prod_{k=0}^{K-1} \frac{k+2\n - 2}{k + 2\n} = \frac{(2\n - 2)(2\n -1)}{(2\n + K - 2) (2\n + K -1)} &\leq \frac{4\n^2}{K^2},
\end{align*}
and
\begin{align*}
    \sum_{k=0}^{K-1} \rho_k^2 \prod_{j=k+1}^K (1 - \frac{\rho_j}{\n}) &= \sum_{k=0}^{K-1} \frac{4\n^2}{(k + 2\n)^2} \prod_{j=k+1}^{K-1} \frac{2\n + j - 2}{2\n + j} \\
    &= \sum_{k=0}^{K-1} \frac{4\n^2}{(k + 2\n)^2} \frac{(2\n + k -1)(2\n + k)}{(2\n + K-2)(2\n + K-1)} \\
    &\leq \sum_{k=0}^{K-1} \frac{4\n^2}{(2\n + K-2)(2\n + K-1)} \\
    &= \frac{4\n^2 K}{(2\n + K-2)(2\n + K-1)} \leq \frac{4\n^2}{K}.
\end{align*}
This ends the proof.
\end{proof}

From the above theorem and the structure of $F$, we can derive rates for the difference between $\beta^K$ and $d^*$, as well as between $z^K$ and $b$. 

\begin{thm}
\label{thm:convergence-bcfw-specified}
    Consider Algorithm~\ref{alg:BCFW} with stepsize $\rho_k = \frac{2\n}{k + 2\n}$. After $K$ iterations, it holds that
    \begin{align*}
        &\Etwo\left[\beta^K -d^* \right] \leq (d^* - d(\bar{\lambda}_T)) + \frac{2\SqrtAvgDiamSquared}{\sqrt{K }} + \frac{2\n}{K} \sqrt{F((\beta^0, z^0)) - F^*},\\
        &\Etwo\left[\norm{z^K - b }_+\right] \leq (d^* - d(\bar{\lambda}_T)) + \frac{2\SqrtAvgDiamSquared}{\sqrt{K }} + \frac{2\n}{K} \sqrt{F((\beta^0, z^0)) - F^*}.
    \end{align*}
\end{thm}
\begin{proof}
We have
\begin{align*}
        \Etwo[\beta^K - d^*] &\leq \Etwo [\beta^K - d(\bar{\lambda}_T)] \\
        &\leq \Etwo[\max(\beta^K - d(\bar{\lambda}_T), 0)] \\
        &\leq \sqrt{\Etwo[\max(\beta^K - d(\bar{\lambda}_T), 0)^2]} \tag{Jensen's inequality} \\
        &\leq \sqrt{\Etwo[2F((\beta^K, z^K))]} \\
        &\leq \sqrt{2F^* + \frac{4}{K} \AvgDiamSquared + \frac{4\n^2}{K^2} (F((\beta^0, z^0)) - F^*)} \tag{Theorem~\ref{thm:BCFW-convergence}} \\
        &\leq \sqrt{2 F^*} + \frac{2\SqrtAvgDiamSquared}{\sqrt{K}} + \frac{2\n}{K} \sqrt{F((\beta^0, z^0)) - F^*}  \\
        &\leq (d^* - d(\bar{\lambda}_T)) + \frac{2\SqrtAvgDiamSquared}{\sqrt{K}} + \frac{2\n}{K} \sqrt{F((\beta^0, z^0)) - F^*} \tag{Lem.~\ref{lem:second-stage-optimal-value-bound}},
    \end{align*}
    which proves the claim for the first inequality. For the second one, we have
    \begin{align*}
       \Etwo \left[\norm{z^K - b}_+\right] 
        &\leq \sqrt{\Etwo \norm{z^K - b}_+^2} \\
        &\leq \sqrt{\Etwo [2F((\beta^K, z^K))]} \\
        &\leq \sqrt{2F^* + \frac{4}{K} \AvgDiamSquared + \frac{4\n^2}{K^2} (F((\beta^0, z^0)) - F^*)} \tag{Theorem~\ref{thm:BCFW-convergence}} \\
        &\leq \sqrt{2 F^*} + \frac{2\SqrtAvgDiamSquared}{\sqrt{K}} + \frac{2\n}{K} \sqrt{F((\beta^0, z^0)) - F^*}  \\
        &\leq (d^* - d(\bar{\lambda}_T)) + \frac{2\SqrtAvgDiamSquared}{\sqrt{K}} + \frac{2\n}{K} \sqrt{F((\beta^0, z^0)) - F^*} \tag{Lem.~\ref{lem:second-stage-optimal-value-bound}}.
    \end{align*}
\end{proof}

To derive a convergence rate for the output $\hat{x}^K$ of Algorithm~\ref{alg:BCFW}, we only need to relate it to $\beta^K$ and $z^K$. Note that from~\eqref{eq:BCFW-update-x-hat}, we have that $\hat{x}_i^K$ is a convex combination of $x_i^0$ and of the outputs $x_i^{k+1}$ of oracle~\eqref{eq:bcfw-oracle} on iterations $k$ such that $i_k = i$. In other words, letting $J_i = \{0\} \cup  \{k\in \{1, \dots, K\} \mid i_{k-1} =i\}$, we have
\begin{align*}
    \hat{x}_i^K = \sum_{k \in J_i} \omega_i^k  x_i^k,
\end{align*}
where $\omega_i \geq 0$ and $\sum_{k\in J_i} \omega_i^k = 1$. Similarly, from~\eqref{eq:BCFW-update-beta-z},
\begin{align}
\label{eq:bcfw-betaK-zK-output}
    \begin{pmatrix}\beta_i^K \\ z_i^K \end{pmatrix} = \sum_{k \in J_i} \omega_i^k \frac{1}{\n} \begin{pmatrix} h_i(x_i^k) \\ A_i x_i^k \end{pmatrix}, \ i =1, \dots, \n. 
\end{align}

We are now able to derive convergence guarantees for $\hat{x}^K$.

\begin{cor}
    \label{cor:stage2-convergence}
    Consider Algorithm~\ref{alg:BCFW} with stepsize $\rho_k = \frac{2\n}{k + 2\n}$. After $K$ iterations, the output $\hat{x}^K$ satisfies
    \begin{align*}
        &\Etwo \left[\frac{1}{\n}\sum_{i=1}^\n h_i(\hat{x}_i^K) - d^* \right] \leq (d^* - d(\bar{\lambda}_T)) + \frac{2\SqrtAvgDiamSquared}{\sqrt{K }} + \frac{2\n}{K} \sqrt{F((\beta^0, z^0)) - F^*},\\
        &\Etwo \norm{\frac{1}{\n} \sum_{i=1}^\n A_i \hat{x}_i^K - b }_+ \leq (d^* - d(\bar{\lambda}_T)) + \frac{2\SqrtAvgDiamSquared}{\sqrt{K }} + \frac{2\n}{K} \sqrt{F((\beta^0, z^0)) - F^*}.
    \end{align*}
\end{cor}
\begin{proof}
    By convexity of $h_i$, note that we have
\begin{align*}
    \frac{1}{\n}\sum_{i=1}^\n h_i(\hat{x}_i^K) &= \frac{1}{\n}\sum_{i=1}^\n h_i\left(\sum_{k \in J_i} \omega_i^k x_i^k\right) 
    \leq \frac{1}{\n}\sum_{i=1}^\n \sum_{k \in J_i} \omega_i^k h_i(x_i^k)
    = \sum_{i=1}^\n \beta_i^K = \beta^K
\end{align*}
Moreover, we also have
    \begin{align*}
        \frac{1}{\n}\sum_{i=1}^\n A_i\hat{x}_i^K &= \frac{1}{\n}\sum_{i=1}^\n A_i\sum_{k \in J_i} \omega_i^k x_i^k = z^K.
    \end{align*}
Applying Theorem~\ref{thm:convergence-bcfw-specified} yields the result.
\end{proof}

In the next section, we combine this stage 2 algorithm with the stochastic dual subgradient from stage 1 to analyze a full method to solve~\eqref{eq:primal}.

\subsection{Combining stage 1 and stage 2}
 The guarantees of the BCFW algorithm derived in Corollary~\ref{cor:stage2-convergence} depend on $d^* - d(\bar{\lambda}_T)$ and on $F((\beta^0, z^0))$, i.e., on the input to this algorithm. We have already derived bounds on $d^* - d(\bar{\lambda}_T)$ in Proposition~\ref{prop:dual-convergence-stochastic-dual-subgradient}. Moreover, we shall also derive bounds on $F((\beta^0, z^0))$ by initializing each $x^0_i$ of Algorithm~\ref{alg:BCFW} with the output $\bar{x}_{i, T}$ of Algorithm~\ref{alg:stochastic-dual-subgradient-separable-primal}. We summarize the full method in Algorithm~\ref{alg:2-stage-algorithm}, and derive its convergence guarantees in the next theorem. We denote $\E$ as the expectation with respect to both Algorithm~\ref{alg:stochastic-dual-subgradient-separable-primal} and~\ref{alg:BCFW}.

\begin{algorithm}
\caption{Two-stage algorithm to solve~\eqref{eq:primal}.}
\label{alg:2-stage-algorithm}
\begin{algorithmic}[1]
\STATE{\textbf{Input}: $\lambda_0 \in \R^m_+$, number of iterations $T$ of stochastic dual subgradient algorithm, stepsize sequence $\{\alpha_t\}_{t=0}^{T-1}$, number of iterations $K$ of BCFW.}
\STATE{Run Algorithm~\ref{alg:stochastic-dual-subgradient-separable-primal} with input $\lambda_0$ for $T$ iterations with stepsize $\{\alpha_t\}_{t=0}^{T-1}$.}
\STATE{Obtain output $\bar{\lambda}_T$ and $\bar{x}_{i, T}$, for $i=1, \dots, \n$.}
\STATE{Run Algorithm~\ref{alg:BCFW} with input $d(\bar{\lambda}_T)$ and $x_i^0 = \bar{x}_{i, T}$ for all $i=1, \dots, \n$.}
\STATE{Obtain output $\hat{x}^K = (\hat{x}_1^K, \dots, \hat{x}_N^K)$.}
\STATE{\textbf{Return:} $\hat{x}^K$.}
\end{algorithmic}
\end{algorithm}

\begin{thm}
    \label{thm:convergence-2-stage-algorithm}
Suppose Assumption~\ref{ass:existence-dual-maximizer} holds and consider Algorithm~\ref{alg:2-stage-algorithm} with input $\lambda_0 = 0_m$ and constant stepsize $\alpha_t = \alpha = \frac{\alphabar}{\Gtilde \sqrt{T}}$. The output $\hat{x}^K$ satisfies
    \begin{align*}
        &\E\left[\frac{1}{\n}\sum_{i=1}^\n h_i(\hat{x}_i^K) - d^*\right] \leq \ErrTK(T, K), \\
        &\E\norm{\frac{1}{\n} \sum_{i=1}^\n A_i \hat{x}_i^K - b }_+\leq \ErrTK(T, K),
    \end{align*}
    where $\ErrTK(T, K)$ is defined as
    \begin{align*}
        \ErrTK(T, K) := P_1 \frac{1}{\sqrt{T}} + P_2 \frac{1}{\sqrt{K}} + \frac{2 \n}{K} \left( P_3 \sqrt{\frac{\n-1}{T}} + P_4 \frac{1}{\sqrt{T}} + P_5 \frac{\n}{T}\right),
    \end{align*}
    where $P_1, P_2, P_3, P_4$ and $P_5$ are constants defined as
    \begin{align*}
        P_1 & := \Gtilde \left(\frac{\norm{\lambda^*}^2}{2 \alphabar } + \frac{\alphabar}{2}\right) , \ \quad 
        P_2 : = 2 D, \ \quad 
        P_3 := \frac{H + ( \norm{\lambda^*} + 1) \Gtilde}{\sqrt{2}}, \\
        P_4 &:= \left( \frac{\alphabar}{\sqrt{2}} + \frac{\norm{\lambda^*}^2}{2\sqrt{2} \alphabar} +  \frac{\norm{\lambda^*} + 1}{\sqrt{2}}\left(1 + \frac{2\norm{\lambda^*}}{\alphabar}\right)\right), \quad
        P_5  := \Gtilde \left(\frac{3\norm{\lambda^*} + 1 + \alphabar }{\sqrt{2}} \right).
    \end{align*}
\end{thm}
\begin{proof}
First observe that by the tower property,
\begin{align*}
    \E\left[\frac{1}{\n}\sum_{i=1}^\n h_i(\hat{x}_i^K) - d^*\right] = \Eone \left[ \Etwo \left[\frac{1}{\n}\sum_{i=1}^\n h_i(\hat{x}_i^K) - d^* \right] \right].
\end{align*}
By Corollary~\ref{cor:stage2-convergence}, we thus get
\begin{align*}
    \E\left[\frac{1}{\n}\sum_{i=1}^\n h_i(\hat{x}_i^K) - d^*\right] \leq \Eone \left[ (d^* - d(\bar{\lambda}_T)) + \frac{2\SqrtAvgDiamSquared}{\sqrt{K }} + \frac{2\n}{K} \sqrt{F((\beta^0, z^0)) - F^*}\right].
\end{align*}
By the same argument, we also get
    \begin{align*}
    \E \norm{\frac{1}{\n} \sum_{i=1}^\n A_i \hat{x}_i^K - b }_+ \leq  \Eone \left[ (d^* - d(\bar{\lambda}_T)) + \frac{2\SqrtAvgDiamSquared}{\sqrt{K }} + \frac{2\n}{K} \sqrt{F((\beta^0, z^0)) - F^*}\right].
    \end{align*}
    Therefore, to obtain the result all we have to do is bound $\Eone [d^* - d(\bar{\lambda}_T)]$ and $\Eone [\sqrt{F((\beta^0, z^0)) - F^*}]$,
    where this time expectation is taken with respect to the randomness induced by Algorithm~\ref{alg:stochastic-dual-subgradient-separable-primal}. First recall Proposition~\ref{prop:dual-convergence-stochastic-dual-subgradient} where we showed
    \begin{align*}
        \Eone [d^* - d(\bar{\lambda}_T)] \leq \frac{\Gtilde \norm{\lambda^*}^2}{2\alphabar\sqrt{T}} + \frac{\Gtilde \alphabar}{2\sqrt{T}}.
    \end{align*}
    Let us then look at $\Eone [\sqrt{F((\beta^0, z^0)) - F^*}]$. Since $F^* \geq 0$ so we will only bound $\Eone [\sqrt{F((\beta^0, z^0))}]$. Next, recall that we initialize the second stage with the output of the first stage, i.e., we have
    \begin{align*}
        \begin{pmatrix}
            \beta^0 \\ z^0 
        \end{pmatrix} = \sum_{i=1}^\n \frac{1}{\n} \begin{pmatrix}
            h_i(\bar{x}_{i, T}) \\ A_i \bar{x}_{i, T}
        \end{pmatrix},
    \end{align*}
    where $\bar{x}_{i, T}$ is the output of Algorithm~\ref{alg:stochastic-dual-subgradient-separable-primal}. We then have
    \begin{align}
        \Eone [\sqrt{F((\beta^0, z^0))}] &= \Eone  
        \left[\sqrt{\frac{1}{2} \max( \frac{1}{\n} \sum_{i=1}^\n h_i(\bar{x}_{i, T}) - d(\bar{\lambda}_T), 0)^2 + \frac{1}{2} \norm{\frac{1}{\n} \sum_{i=1}^\n A_i \bar{x}_{i, T} - b }_+^2} \right] \nonumber \\
        &\leq \frac{1}{\sqrt{2}} \left(\Eone \left[\max( \frac{1}{\n} \sum_{i=1}^\n h_i(\bar{x}_{i, T}) - d(\bar{\lambda}_T), 0)\right] + \Eone \left[\norm{\frac{1}{\n} \sum_{i=1}^\n A_i \bar{x}_{i, T} - b }_+\right] \right).
        \label{eq:proof-2-stage-convergence-1}
    \end{align}
    We can directly control the second term in the above sum using Proposition~\ref{prop:primal-convergence-stochastic-dual-subgradient}. The first term is however trickier, as we have the max operator which prevents us from directly applying Proposition~\ref{prop:primal-convergence-stochastic-dual-subgradient}. We circumvent this next. We have
    \begin{align*}
        \frac{1}{\n} \sum_{i=1}^\n h_i(\bar{x}_{i, T}) - d(\bar{\lambda}_T) &= \underbrace{\frac{1}{\n} \sum_{i=1}^\n h_i(\bar{x}_{i, T}) + (\lambda^*)^\top(A_i \bar{x}_{i, T} - b)}_{\geq d(\lambda^*)} - \frac{1}{\n} \sum_{i=1}^\n (\lambda^*)^\top( A_i \bar{x}_{i, T} - b) -  d(\bar{\lambda}_T) \\
        & \geq d^* - d(\bar{\lambda}_T) - (\lambda^*)^\top \ProjPos{\frac{1}{\n} \sum_{i=1}^\n A_i\bar{x}_{i, T} - b} \tag{since $\lambda^* \geq 0$} \\
        &\geq - (\lambda^*)^\top \ProjPos{\frac{1}{\n} \sum_{i=1}^\n A_i\bar{x}_{i, T} - b}.
    \end{align*}
     Now, for any $a, c \in \R$ for which $a \geq c$, it holds that $\max(a, 0) \leq a + \abs{c}$ (see Lemma~\ref{lem:max-inequality} for a proof). Therefore,
    \begin{align*}
        \max( \frac{1}{\n} \sum_{i=1}^\n h_i(\bar{x}_{i, T}) - d(\bar{\lambda}_T), 0) &\leq \frac{1}{\n} \sum_{i=1}^\n h_i(\bar{x}_{i, T}) - d(\bar{\lambda}_T) + \abs{(\lambda^*)^\top \ProjPos{\frac{1}{\n} \sum_{i=1}^\n A_i\bar{x}_{i, T} - b}} \\
        &\leq \frac{1}{\n} \sum_{i=1}^\n h_i(\bar{x}_{i, T}) - d(\bar{\lambda}_T) + \norm{\lambda}^* \norm{\frac{1}{\n} \sum_{i=1}^\n A_i\bar{x}_{i, T} - b}_+,
    \end{align*}
    by the Cauchy-Schwarz inequality. Going back to~\eqref{eq:proof-2-stage-convergence-1}, we get
    \begin{align*}
        \Eone [\sqrt{F((\beta^0, z^0))}] \leq& \ \frac{1}{\sqrt{2}} \left( \Eone \left[\frac{1}{\n} \sum_{i=1}^\n h_i(\bar{x}_{i, T}) - d(\bar{\lambda}_T) +  \norm{\lambda^*} \norm{\frac{1}{\n} \sum_{i=1}^\n A_i \bar{x}_{i,T} - b}_+\right] + \Eone \norm{\frac{1}{\n} \sum_{i=1}^\n A_i \bar{x}_{i,T} - b}_+\right) \\
        =&\ \frac{1}{\sqrt{2}} \Eone \left[\frac{1}{\n} \sum_{i=1}^\n h_i(\bar{x}_{i, T}) -d^*\right] +  \frac{1}{\sqrt{2}} \Eone[d^* - d(\bar{\lambda}_T)] +\frac{\norm{\lambda^*} + 1}{\sqrt{2}} \Eone \norm{\frac{1}{\n} \sum_{i=1}^\n A_i \bar{x}_{i,T} - b}_+ \\
        \leq &\ \frac{1}{\sqrt{2}} \left( H \sqrt{\frac{N-1}{T}} + \frac{\alphabar \Gtilde }{2 \sqrt{{T}}}  + \frac{\n \Gtilde( 2\norm{\lambda^*} + \alphabar)}{T}\right) + \frac{1}{\sqrt{2}} \left( \frac{\Gtilde \norm{\lambda^*}^2}{2\alphabar \sqrt{T}} + \frac{\Gtilde \alphabar}{2 \sqrt{T}}\right)  \\ &+ \frac{\norm{\lambda^*} + 1}{\sqrt{2}}  \left( \Gtilde \sqrt{\frac{N-1}{T}} + \frac{ 2\Gtilde \norm{\lambda^*}} {\alphabar\sqrt{T}} + \frac{\Gtilde}{\sqrt{T}}  + \frac{\n \Gtilde }{T} \right),
    \end{align*}
    by applying Proposition~\ref{prop:dual-convergence-stochastic-dual-subgradient} and Proposition~\ref{prop:primal-convergence-stochastic-dual-subgradient}. Putting everything together,
    \begin{align*}
        \Eone [\sqrt{F((\beta^0, z^0))}] \leq &\  \sqrt{\frac{N-1}{T}} \left( \frac{H + (\norm{\lambda^*} + 1) \Gtilde}{\sqrt{2}}\right) \\  &+ \frac{\Gtilde}{\sqrt{T}}\left( \frac{\alphabar}{2\sqrt{2}} + \frac{\norm{\lambda^*}^2}{2\sqrt{2} \alphabar} + \frac{\alphabar}{2\sqrt{2}}+  \frac{\norm{\lambda^*} + 1}{\sqrt{2}}\left(1 + \frac{2\norm{\lambda^*}}{\alphabar}\right)\right) \\
        &+ \frac{\n \Gtilde}{T} \left(\frac{3\norm{\lambda^*} + 1 + \alphabar }{\sqrt{2}}\right).
    \end{align*}
\end{proof}

We are then ready to state the full complexity of our algorithm, which is the main result of this work.

\begin{thm}
\label{thm:2-stage-algorithm-complexity}
    Let $\epsilon > 0$. Suppose Assumption~\ref{ass:existence-dual-maximizer} holds and consider Algorithm~\ref{alg:2-stage-algorithm} with $T = K$. In order to for the output $\hat{x}^K$ to be an $\epsilon$-accurate primal solution in expectation, the total number of calls to oracle~\eqref{eq:oracle-definition} is bounded above by
    \begin{align*}
        O\left( \frac{1}{\epsilon^2} + \frac{\n}{\epsilon^{2/3}}\right).
    \end{align*}
\end{thm}
\begin{proof}
    From Theorem~\ref{thm:convergence-2-stage-algorithm}, we have
    \begin{align*}
    &\E[\frac{1}{\n}\sum_{i=1}^\n h_i(\hat{x}_i^K) - d^*] \leq \ErrTK(T, K)\ \  \text{ and }\ \  \E\norm{\frac{1}{\n} \sum_{i=1}^\n A_i \hat{x}_i^K - b }_+\leq \ErrTK(T, K),
    \end{align*}
    where $\ErrTK(T, K)$ is such that
    \begin{align*}
        \ErrTK(T, K) &= O\left(  \frac{1}{\sqrt{T}} + \frac{1}{\sqrt{K}} + \frac{\n}{K} \left( \sqrt{\frac{N}{T}} + \frac{1}{\sqrt{T}} + \frac{\n }{T} \right) \right) \\
        &= O\left(\frac{1}{\sqrt{T}} + \frac{N^{3/2}}{T^{3/2}}  + \frac{\n^2 }{T^2}\right). \tag{since $T=K$}
        \end{align*}
        To get an $\epsilon$-accurate primal solution, we need $R(T, K) \leq \epsilon$, i.e.,
        \begin{align*}
            T = K = O\left( \frac{1}{\epsilon^2} + \frac{\n }{\epsilon^{2/3}} + \frac{\n }{\epsilon^{1/2}}\right) = O\left( \frac{1}{\epsilon^2} + \frac{\n }{\epsilon^{2/3}}\right).
        \end{align*}
        Since each iteration of Algorithm~\ref{sec:stochastic-dual-subgradient} requires one call to~\eqref{eq:oracle-definition}, except for the last one which requires $\n$, and each iteration of Algorithm~\ref{alg:BCFW} requires one call to~\eqref{eq:oracle-definition}, we get that the complexity in terms of number of calls to~\eqref{eq:oracle-definition} is
        \begin{align*}
            O\left( \n + \frac{1}{\epsilon^2} + \frac{\n }{\epsilon^{2/3}}\right) = O\left( \frac{1}{\epsilon^2} + \frac{\n }{\epsilon^{2/3}}\right).
        \end{align*}
\end{proof}

Recall that the complexity bound to obtain $\epsilon$-accurate primal solutions of the deterministic dual subgradient from Theorem~\ref{prop:epsilon-complexity-dual-subgradient} was $O(\n/\epsilon^2)$. Comparing with the result above, we see that our algorithm allows to no longer pay a factor $\n$ in the $\frac{1}{\epsilon^2}$ term of the complexity. Instead we only pay the factor $\n$ scaled by $\frac{1}{\epsilon^{2/3}}$ which, for small values of $\epsilon$, is much smaller than $\frac{1}{\epsilon^2}$.

\section{The nonconvex case}
\label{sec:nonconvex}

We now turn our attention to separable problems of the form
\begin{align}
\tag{NC-P}
    \begin{split}
    \mbox{minimize} & \ \frac{1}{\n} \sum_{i=1}^\n f_i(x_i)\\
    \mbox{subject to }& \ \frac{1}{\n}\sum_{i=1}^\n A_ix_i  \leq  b,\\
    &x_i \in \domfi, \ i=1, \dots, \n,
    \end{split}
    \label{eq:nonconvex-primal}
\end{align}
in the variables $x_i \in \R^{d_i}$. We make the following assumptions on the functions $f_i$.
\begin{assumption}
For all $i \in \{1, \dots, \n\}$, the function $f_i$ is closed, proper and has an affine minorant. Moreover, its domain $\domfi$ is compact.
\end{assumption}
Note that we make no assumptions on the convexity of the functions $f_i$ or their domains. In this section we will be particularly interested in cases where $m$, the dimension of the coupling constraints, is much smaller than $\n$, the number of agents. This is relevant, for example, in energy management problems~\cite{bertsekas1983optimal,vujanic2016decomposition}, supply chain optimization~\cite{vujanic2014large} and portfolio optimization~\cite{baumann2013portfolio}.

In this case, the bidual of~\eqref{eq:nonconvex-primal} is given by (see~\cite{dubois2025frank})
\begin{align}
\tag{NC-BD}
    \begin{split}
    \mbox{minimize} & \ \frac{1}{\n} \sum_{i=1}^\n f_i^{**}(x_i)\\
    \mbox{subject to }& \ \frac{1}{\n}\sum_{i=1}^\n A_ix_i  \leq  b,\\
    &x_i \in \conv(\domfi), \ i=1, \dots, \n,
    \end{split}
    \label{eq:nonconvex-bidual}
\end{align}

The functions $f_i^{**}$ are closed, proper and convex, and their domain $\conv(\domfi)$ are convex compact, so that bidual~\eqref{eq:nonconvex-bidual} directly fits within the setting of~\eqref{eq:primal} with $h_i = f_i^{**}$ and $\domhi = \conv(\domfi)$. Moreover, the dual problem of~\eqref{eq:nonconvex-bidual} reads
\begin{align}
\tag{NC-D}
    \begin{split}
        \mbox{maximize} & \ \left\{ d(\lambda) := -\lambda^\top b + \frac{1}{\n} \sum_{i=1}^\n \inf_{x_i \in \conv(\domfi)} \{ f_i^{**}(x_i) + \lambda^\top A_i x_i\}\right\}\\
        \mbox{subject to} &\ \lambda \geq 0.
    \end{split}
    \label{eq:nonconvex-dual}
\end{align}
Just like in the convex case, we makes the general assumption that the dual has a solution and that strong duality holds between the dual and the bidual.

\begin{assumption}
\label{ass:nonconvex-existence-dual-maximizer}
    There exists $\lambda^* \in \R^m_+$ such that $\lambda^* \in \argmax_{\lambda \in \R^m_+ } d(\lambda)$. We write $d^* = d(\lambda^*)$. Moreover, strong duality holds between \eqref{eq:nonconvex-dual} and \eqref{eq:nonconvex-bidual}, i.e., $d^*$ is also the optimal value of the bidual problem~\eqref{eq:nonconvex-bidual}.
\end{assumption}

This assumption will allow us to directly apply our two-stage algorithm to the bidual problem~\eqref{eq:nonconvex-bidual}. Before doing so, we provide some background on nonconvex duality gap estimation in the next subsection.

\subsection{Background and recent results on nonconvexity and duality gap estimation}

We start with the notion of nonconvexity of the functions~$f_i$.
\begin{definition}
\label{def:function-nonconvexity}
    The nonconvexity of the function $f_i : \R^{d_i} \rightarrow \R$ is defined as~\cite{aubin1976estimates}
\begin{align}
\label{eq:function-nonconvexity-definition}
    \rho(f_i) := \sup \left\{ f_i\left( \sum_{j=1}^p \alpha_j x^j \right) - \sum_{j=1}^p \alpha_j f_i(x^j) \mid p \in \N, x^j \in \domfi, \alpha \geq 0, \sum_j \alpha_j = 1 \right\}. 
\end{align}
\end{definition}
From the above definition, we can see that  $\rho(f_i) = 0$ if $f_i$ is convex. Equivalently, the nonconvexity of $f_i$ can be written (see~\cite{bi2016refined} for a proof)
\begin{align*}
    \rho(f_i) = \sup_{x\in \domfibiconjugate} \{f_i(x) - f_i^{**}(x) \}.
\end{align*}


Duality gap bounds for~\eqref{eq:nonconvex-primal} are typically derived using the Shapley-Folkman theorem, which states that the sum of arbitrary sets gets increasingly close to its convex hull. We recall the theorem next, but omit its proof as it is not essential for the current discussion. This theorem, established in private communications between Shapley and Folkman, was first published in~\cite{starr1969quasi}.

\begin{thm}[Shapley-Folkman]
    Let $V_i \subset \R^p$, $i=1,\dots, n$ and $x \in \sum_{i=1}^n \conv(V_i)$. Then 
    \begin{align*}
        x \in \sum_{i\in S} \conv(V_i) + \sum_{i \not \in S} V_i,
    \end{align*}
    for some set $S\subset \{1, \dots, n\}$ such that $|S| \leq p$.
\end{thm}

Based on this theorem, duality gap bounds on separable nonconvex problems in the form~\eqref{eq:nonconvex-primal} can be derived. The most widely known result was first published by Aubin and Ekeland in~\cite{aubin1976estimates}. 

\begin{thm}
\label{thm:duality-gap-estimation-convex-domains}
    Let $d^*$ be the optimal dual value and $\rho(f_i)$ the nonconvexity of the functions $f_i$ as defined in~\eqref{eq:function-nonconvexity-definition}. Suppose that $\domfi$ is convex for all $i=1, \dots, \n$. Then there exists $x^* \in \R^d$ such that $x_i^* \in \domfi$ for all $i=1, \dots, \n$ and
    \begin{align*}
        d^* \leq \frac{1}{\n} \sum_{i=1}^\n f_i(x_i^{*}) \leq d^* + \frac{m+1}{\n} \max_i \rho(f_i).
    \end{align*}
\end{thm}
\begin{proof}
    See~\cite{bertsekas2014constrained},~\cite{ekeland1999convex} or~\cite[Theorem 2.5]{dubois2025frank} for a more modern exposition.
\end{proof}

The above theorem tells us that when the number of agents $\n$ increases but the size $m$ of the coupling constraint and the nonconvexity of the functions $f_i$ remain bounded, the duality gap vanishes. When the function domains are not convex, similar bounds also hold based on a more general notion of nonconvexity and under slightly stronger assumptions. We will discuss one such result at the very end of this section, and simply refer to~\cite{bertsekas2014constrained,dubois2025frank} for more on this topic.

The important thing to emphasize at this point is that although the above theorem provides strong duality gap estimations for~\eqref{eq:nonconvex-primal} when $\n$ is much greater than $m$, it does not provide a constructive way to get a primal solution achieving those bounds. Constructive approaches for doing so have been studied, notably in~\cite{udell2016bounding} and~\cite{dubois2025frank}. In~\cite{udell2016bounding}, the authors suggest a randomized approach to solve~\eqref{eq:nonconvex-primal}, and achieve the bounds from Theorem~\ref{thm:duality-gap-estimation-convex-domains} with probability 1. Their approach however requires optimizing over the full solution of the bidual problem~\eqref{eq:nonconvex-bidual}, which limits the application of the method in practical settings.

The approach we develop in this section is closer to the approach from~\cite{dubois2025frank}, with a few important differences. While in~\cite{dubois2025frank}, the authors assume access to the optimal dual value $d^*$, our first stage approximates it using the stochastic dual subgradient method. In the second stage, in~\cite{dubois2025frank} the authors define the sets $C_i$ and the function~$F$ just like we do in~\eqref{eq:definition-Ci-sets} and~\eqref{eq:F-definition}. They then optimize $F$ using the Frank-Wolfe algorithm. Using properties of the oracle~\eqref{eq:oracle-definition} (which we will explore next), they show that the output of the Frank-Wolfe algorithm is a conic combination of primal points which approximates the optimal dual value. Finally, they build a primal solution recovering the duality gap bounds of Theorem~\ref{thm:duality-gap-estimation-convex-domains} by computing a Carath\'eodory decomposition of the output of the Frank-Wolfe algorithm. Our overall approach is similar, except that, as we saw previously, we instead run the block-coordinate version of Frank-Wolfe, which dramatically improves the overall complexity for large $\n$. In the end, our approach will yield similar bounds as in~\cite{dubois2025frank} but (i) without assuming knowledge of $d^*$ and (ii) by reducing the overall complexity on $\n$.

The rest of this section is dedicated to applying our framework to~\eqref{eq:nonconvex-bidual} and obtaining solutions of~\eqref{eq:nonconvex-primal}. In Section~\ref{sec:oracle-nonconvex} we develop the first crucial element of our method, where we show that we can assume that the oracle~\eqref{eq:oracle-definition} for~\eqref{eq:nonconvex-bidual} returns a point in $\domfi$ for which the functions $f_i$ and $f_i^{**}$ match. In Section~\ref{sec:nonconvex-stochastic-dual-subgradient} we discuss how we can still apply the stochastic dual subgradient to obtain an approximation of $d^*$. In Section~\ref{sec:nonconvex-bcfw} we discuss the application of the block coordinate Frank-Wolfe algorithm, and how the properties of the oracle allow the output to be a convex combination of primal points for each $i=1,\dots, \n$. We then show how to make this convex combination trivial for at least $\n - m$ indices by using Carath\'eodory decomposition algorithms.

\subsection{Oracle and extreme points}
\label{sec:oracle-nonconvex}

In the case of~\eqref{eq:nonconvex-bidual}, the oracle~\eqref{eq:oracle-definition} reads,
\begin{align}
\tag{NC-O1}
\label{eq:nonconvex-oracle}
    x_i^*((\gamma, \lambda)) \in \argmin_{x_i \in \conv(\domfi)} \{ \gamma f_i^{**}(x_i) + \lambda^\top A_i x_i\}.
\end{align}
The important aspect of our approach in this section is that we can actually restrict the oracle to the original nonconvex functions. This result was proved in~\cite[Lemma 3.1]{dubois2025frank}, and we provide the proof for completeness in Appendix~\ref{app:nonconvex-results}.

\begin{restatable}{prop}
{ExtremePointsMatching}
\label{prop:extreme-points-match}
Given $\gamma \geq 0, \lambda \in \R^m$, consider the oracle
\begin{align}
\label{eq:oracle_nc}
\tilde{x}_i \in \argmin_{x_i \in  \domfi} \{\gamma f_i(x_i) + \lambda^\top A_i x_i\}.
\end{align}
Assume that, if $\gamma = 0$, the above oracle returns an extreme point of $\dom (f_i)$.
Then $\tilde{x}_i$ is a solution to oracle~\eqref{eq:nonconvex-oracle} and crucially, $f_i(\tilde{x}_i) = f_i^{**}(\tilde{x}_i)$.
\end{restatable} 

We therefore assume from now on that oracle~\eqref{eq:nonconvex-oracle} returns $x_i^*((\gamma, \lambda)) \in \domfi$ such that $f_i(x_i^*((\gamma, \lambda))) = f_i^{**}(x_i^*((\gamma, \lambda)))$.

We are then ready to apply the previously derived two-stage method to the the bidual~\eqref{eq:nonconvex-bidual} of our nonconvex problem.

\subsection{Stochastic dual subgradient}
\label{sec:nonconvex-stochastic-dual-subgradient}

In the first stage, we approximate $d^*$, the optimal value of the dual $d$ as defined in~\eqref{eq:nonconvex-dual} via the stochastic dual subgradient algorithm (Algorithm~\ref{alg:stochastic-dual-subgradient-separable-primal}). At the end of this stage, the approximation $d(\bar{\lambda}_T)$ of the optimal dual value satisfies, by Proposition~\eqref{prop:dual-convergence-stochastic-dual-subgradient},
\begin{align*}
    \E[d^* - d(\bar{\lambda}_T)] \leq \frac{\Gtilde \norm{ \lambda^*}^2}{2 \alphabar\sqrt{T}} + \frac{\Gtilde \alphabar}{2\sqrt{T}}.
\end{align*}

\subsection{Second stage: block-coordinate Frank-Wolfe}
\label{sec:nonconvex-bcfw}

The second stage consists, as before, in running the BCFW algorithm (Algorithm~\ref{alg:BCFW}). As input, it requires $d(\bar{\lambda}_T)$ and some $(\beta_i^0, z_i^0)$ for each $i=1,\dots, \n$. In the convex case, we used $\beta_i^0 = \frac{1}{\n}h_i(\bar{x}_{i, T})$,
where $\bar{x}_{i, T}$ was the output of Algorithm~\ref{alg:stochastic-dual-subgradient-separable-primal}, i.e.,
\begin{align*}
    \bar{x}_{i, T} = \frac{1}{I_i + 1} \left( x_i^*(\lambda_{T-1}) +  \sum_{t\in \{0, \dots, T-2\}, \atop i_t = i} x_i^*(\lambda_t)\right), \ i=1,\dots, \n. 
\end{align*}
This would lead to convergence guarantees on the bidual problem similar to Proposition~\ref{prop:primal-convergence-stochastic-dual-subgradient}, by convexity of the functions $f_i^{**}$. However, we are looking for bounds on the primal problem, and thus must resort to something else to account for the nonconvexity of the functions $f_i$. We therefore suggest setting
\begin{align*}
\beta_i^0 = \frac{1}{I_i + 1} \left( \frac{1}{\n}f_i^{**}\left(x_i^*(\lambda_{T-1})\right) +  \sum_{t\in \{0, \dots, T-2\}, \atop i_t = i} \frac{1}{\n}f_i^{**}\left( x_i^*(\lambda_t)\right)\right), \ i=1,\dots, \n.
\end{align*}
Recalling that the $x_i^*(\lambda_t)$ are output of our oracle (namely equation~\eqref{eq:stochastic-dual-subgradient-oracle}), by the discussion above and Proposition~\ref{prop:extreme-points-match} we actually have
\begin{align}
\label{eq:beta0-nonconvex-definition}
    \beta_i^0 = \frac{1}{I_i + 1} \left( \frac{1}{\n}f_i\left(x_i^*(\lambda_{T-1})\right) +  \sum_{t\in \{0, \dots, T-2\}, \atop i_t = i} \frac{1}{\n}f_i\left( x_i^*(\lambda_t)\right)\right), \ i=1,\dots, \n.
\end{align}
For $z_i^0$, we do not have to worry with such considerations and can simply set
\begin{align}
\label{eq:z0-nonconvex-definition}
    z_i^0 = \frac{1}{\n} A_i \bar{x}_{i, T}, \ i=1,\dots, \n.
\end{align}
 We can then proceed to run Algorithm~\ref{alg:BCFW} with this initialization. The algorithm outputs $(\beta^K, z^K)$ for which, by Theorem~\ref{thm:convergence-bcfw-specified}, it holds that
\begin{align}
\label{eq:nonconvex-betaK-zK-rate-randomness-alg2}
\begin{split}
    &\Etwo[\beta^K -d^* ] \leq (d^* - d(\bar{\lambda}_T)) + \frac{2\SqrtAvgDiamSquared}{\sqrt{K }} + \frac{2\n}{K} \sqrt{F((\beta^0, z^0)) - F^*},\\
        &\Etwo\norm{z^K - b }_+ \leq (d^* - d(\bar{\lambda}_T)) + \frac{2\SqrtAvgDiamSquared}{\sqrt{K }} + \frac{2\n}{K} \sqrt{F((\beta^0, z^0)) - F^*}.
    \end{split}
\end{align}
In the next proposition, we take the expectation with respect to the randomness induced by the first algorithm to get a final rate on $(\beta^K, z^K)$. Note that the bound is exactly the same as in the convex case in Theorem~\ref{thm:convergence-2-stage-algorithm}. The proof is also very similar, the only important difference being that we must be extra careful with the initialization scheme \eqref{eq:beta0-nonconvex-definition} for $\beta^0$, which is different than in the convex section. We defer the proof to appendix~\ref{app:nonconvex-results} to ease the flow of the section.
\begin{restatable}{prop}{PropNonConvexTwoStageConvergenceBetaZ}
\label{prop:nonconvex-two-stage-convergence-beta-z}
    Consider the nonconvex problem~\ref{eq:nonconvex-primal}. Suppose we run Algorithm~\ref{alg:stochastic-dual-subgradient-separable-primal} for $T$ iterations with input $\lambda_0 = 0_m$ and constant stepsize $\alpha_t = \alpha = \frac{\alphabar}{\Gtilde \sqrt{T}}$, followed by Algorithm~\ref{alg:BCFW} for $K$ iterations with input $(\beta^0, z^0)$ as described in~\eqref{eq:beta0-nonconvex-definition} and~\eqref{eq:z0-nonconvex-definition}. The output $(\beta^K, z^K)$ then satisfies
\begin{align*}
    &\E[\beta^K -d^* ] \leq \ErrTK(T, K),\\
        &\E\norm{z^K - b }_+ \leq \ErrTK(T, K),
\end{align*}
where
\begin{align*}
        \ErrTK(T, K) := P_1 \frac{1}{\sqrt{T}} + P_2 \frac{1}{\sqrt{K}} + \frac{2 \n}{K} \left( P_3 \sqrt{\frac{\n-1}{T}} + P_4 \frac{1}{\sqrt{T}} + P_5 \frac{\n}{T}\right),
    \end{align*}
where $P_1, P_2, P_3, P_4$ and $P_5$ are the constants defined in Theorem~\ref{thm:convergence-2-stage-algorithm}.
\end{restatable}

\subsection{Characterization of the BCFW output}

As in the convex case, the output can be written as a simple convex combination. Yet, because of the initialization we chose above (equation~\eqref{eq:beta0-nonconvex-definition}), it has a slightly different structure, which we summarize in the next lemma.

\begin{lem}
    The output $(\beta^K, z^K)$ can be written as 
    \begin{align}
    \label{eq:nonconvex-betaK-zK-output}
        \begin{pmatrix}
            \beta^K \\ z^K
        \end{pmatrix} = \sum_{i=1}^\n \sum_{l \in Q_i}  \omega_i^l \frac{1}{\n}\begin{pmatrix}
            f_i(x_i^l) \\ A_i x_i^l
        \end{pmatrix},
    \end{align}
    where, for each $i=1, \dots, \n$, $Q_i$ is an appropriately defined finite set, and it holds that $\omega_i \geq 0$, $\sum_{l \in Q_i} \omega_i^l = 1$ and $x_i^l \in \dom (f_i)$ for all $l \in Q_i$.
\end{lem}
\begin{proof}
    One can see that
    \begin{align*}
        \beta^K = \sum_{i=1}^\n \beta_i^K \ \text{ where } \ 
        \beta_i^K = \omega_i^{-1} \beta^0_i + \sum_{k, i_k=i} \omega_i^k \frac{1}{\n}f_i^{**}(x_{i_k}^{k+1}),
    \end{align*}
    with $\omega_i \geq 0$, $\omega_i^{-1} + \sum_{k, i_k=i} \omega_i^k = 1$. Using once again the properties of our oracle,
    \begin{align*}
        \beta_i^K = \omega_i^{-1} \beta^0_i + \sum_{k, i_k=i} \omega_i^k \frac{1}{\n}f_i(x_{i_k}^{k+1}).
    \end{align*}
    Writing out the definition of $\beta_i^0$ from~\eqref{eq:beta0-nonconvex-definition}, we get
    \begin{align*}
        \beta_i^K = \omega_i^{-1} \left(\frac{1}{I_i + 1} \left( \frac{1}{\n}f_i\left(x_i^*(\lambda_{T-1})\right) +  \sum_{t\in \{0, \dots, T-2\}, \atop i_t = i} \frac{1}{\n}f_i\left( x_i^*(\lambda_t)\right)\right) \right) + \sum_{k, i_k=i} \omega_i^k \frac{1}{\n}f_i(x_{i_k}^{k+1}).
    \end{align*}
    Up to rewriting and redefining the coefficients, this is a convex combination of elements $\frac{1}{\n} f_i(x_{i}^l)$ for some $x_i^l \in \dom(f_i)$. Thus, defining the appropriate set $Q_i$, we can write
    \begin{align*}
        \beta_i^{K} = \sum_{l \in Q_i} \omega_i^l \frac{1}{\n} f_i(x_{i}^l),
    \end{align*}
    where $\omega_i^l \geq 0$ and $\sum_{l \in Q_i} \omega_i^l = 1$. The same argument applies to $z^K$.
\end{proof}

Let us summarize what we have up to this point. After running Algorithm~\ref{alg:stochastic-dual-subgradient-separable-primal} and~\ref{alg:BCFW}, we get $(\beta^K, z^K)$, which can be written as the sum of convex combinations
\begin{align}
\label{eq:nonconvex-bcfw-betaK-zK-output}
    \begin{pmatrix}
        \beta^K \\ z^K
    \end{pmatrix} = \sum_{i=1}^{\n}\begin{pmatrix}\beta_i^K \\ z_i^K \end{pmatrix} = \sum_{i=1}^\n \sum_{l \in Q_i} \omega_i^l \frac{1}{\n} \begin{pmatrix} f_i(x_i^l) \\ A_i x_i^l \end{pmatrix}.
\end{align}
Moreover, from Proposition~\ref{prop:nonconvex-two-stage-convergence-beta-z}, we have convergence guarantees on $(\beta^K, z^K)$ relating it to optimal dual values. In the convex case this was enough to conclude. In the nonconvex case, it is precisely here that we can no longer blindly follow the approach from the convex section. Indeed, this would lead us to consider the vector $\hat{x}^K$ such that $\hat{x}^K_i = \sum_{l \in Q_i} \omega_i^l x_i^l$, which would be a solution of~\eqref{eq:primal} by convexity of the functions $h_i$. As this is not true for the functions $f_i$, we suggest a fix next. 

\subsection{Carath\'eodory decomposition}
\label{sec:nonconvex-caratheodory}

From~\eqref{eq:nonconvex-bcfw-betaK-zK-output}, observe that we can write
\begin{align}
\label{eq:nonconvex-betaK-zK-output-big}
    \begin{pmatrix}
        \beta^K \\ z^K \\ \ones_\n
    \end{pmatrix}  = \sum_{i=1}^\n \sum_{l \in Q_i} \omega_i^l \begin{pmatrix} \frac{1}{\n}f_i(x_i^l) \\ \frac{1}{\n}A_i x_i^l \\ e_i \end{pmatrix}, 
\end{align}
where $e_i$ is the $i$-th vector in the canonical basis of $\R^\n$. In particular, this is a conic combination of the vectors $\left\{ \begin{pmatrix} \frac{1}{\n}f_i(x_i^l) \\ \frac{1}{\n}A_i x_i^l \\ e_i \end{pmatrix} \mid i=1,\dots,\n, \ l \in Q_i \right\} \subset \R^{1 + m + \n}$. Let us now recall the conic Carath\'eodory theorem.

\begin{thm}[Conic Carath\'eodory]
\label{thm:conic-caratheodory}
    Let $V \in \R^p$ and suppose $v \in \R^p$ is in the conic hull of $V$. Then there exists $\{\mu_j\}_{j=1}^p \subset \R_+$ and $\{v_j\}_{j=1}^p \subset V$ such that
    \begin{align*}
        v = \sum_{j=1}^p \mu_j v_j.
    \end{align*}
\end{thm}
\begin{proof}
    See~\cite[Corollary 17.1.2]{rockafellar2015convex}.
\end{proof}

Applying the above result to~\eqref{eq:nonconvex-betaK-zK-output-big}, we get that there exists $\tilde{\omega}_i^l \geq 0$ such that 
\begin{align}
    \begin{pmatrix}
        \beta^K \\ z^K \\ \ones_\n
    \end{pmatrix}  = \sum_{i=1}^\n \sum_{l \in Q_i} \tilde{\omega}_i^l  \begin{pmatrix} \frac{1}{\n}f_i(x_i^l) \\ \frac{1}{\n}A_i x_i^l \\ e_i \end{pmatrix}, 
\end{align}
where at most $1 + m + \n$ of the scalars $\tilde{\omega}_i^l$ are nonzero. By looking at the last $\n$ coordinates, note that it must hold that $\sum_{l \in Q_i} \tilde{\omega}_i^l = 1$ for all $i=1,\dots, \n$. By a pigeonhole argument, this implies that the convex combination is actually trivial for at least $\n - (m + 1)$ indices. In other words, this gives
\begin{align}
\label{eq:nonconvex-betaK-zK-caratheodory-output}
    \begin{pmatrix}
        \beta^K \\ z^K \\
    \end{pmatrix}  = \sum_{i \in S} \sum_{l \in Q_i} \tilde{\omega}_i^l  \begin{pmatrix} \frac{1}{\n}f_i(x_i^l) \\ \frac{1}{\n}A_i x_i^l \end{pmatrix} + \sum_{i \not \in S}  \begin{pmatrix} \frac{1}{\n}f_i(x_i^{l_i}) \\ \frac{1}{\n}A_i x_i^{l_i} \end{pmatrix},
\end{align}
for some $S \subset \{1, \dots, \n\}$ with $|S| \leq m+1$ and some $l_i \in Q_i$ for all $i \not \in S$. We summarize the full method in Algorithm~\ref{alg:2-stage-algorithm-nonconvex}.

\begin{algorithm}
\caption{Two-stage algorithm to solve~\eqref{eq:nonconvex-primal}.}
\label{alg:2-stage-algorithm-nonconvex}
\begin{algorithmic}[1]
\STATE{\textbf{Input}: $\lambda_0 \in \R^m_+$, number of iterations $T$ of stochastic dual subgradient algorithm, stepsize sequence $\{\alpha_t\}_{t=0}^{T-1}$, number of iterations $K$ of BCFW.}
\STATE{Run Algorithm~\ref{alg:stochastic-dual-subgradient-separable-primal} with input $\lambda_0$ for $T$ iterations with stepsize $\{\alpha_t\}_{t=0}^{T-1}$.}
\STATE{Obtain output $\bar{\lambda}_T$.}
\STATE{Run Algorithm~\ref{alg:BCFW} with input $d(\bar{\lambda}_T)$ and $x_i^0 \in \domfi $ such that $f_i(x_i^0) = f_i^{**}(x_i^0)$ for all $i=1, \dots, \n$.}
\STATE{Obtain output $(\beta^K, z^K)$ as in~\eqref{eq:nonconvex-bcfw-betaK-zK-output}.}
\STATE{Compute Carath\'eodory decomposition to obtain decomposition of $(\beta^K, z^K)$ as in~\eqref{eq:nonconvex-betaK-zK-caratheodory-output}.}
\STATE{\textbf{Return:} For $i \in S$, the couples $\{(\tilde{\omega}_i^k,  x_i^k)\}_{k \in J_i}$. For $i\not \in S$, the vector $x_i^{k_i}$.}
\end{algorithmic}
\end{algorithm}

Note that at this point we simply stated the existence of the conic Carath\'eodory decomposition in Theorem~\ref{thm:conic-caratheodory}, but did not discuss algorithms to obtain it. A full description of algorithms to compute such decompositions being outside the scope of this work, we simply assume that we have access to one. We refer the reader to~\cite{beck2017linearly,wirth2024frank,besanccon2025pivoting,dubois2025frank} for algorithms computing an exact decomposition, which typically rely on iteratively solving simple linear systems, leading to high complexity and limiting their effectiveness in large-scale settings. Alternatively, approximate Carath\'eodory decompositions can be computed using fast variants of the Frank-Wolfe algorithm on a well-crafted objective function~\cite{combettes2023revisiting}. In particular, the min-norm point algorithm~\cite{wolfe1970convergence} has been observed to efficiently compute good Carath\'eodory decompositions in practical large-scale applications~\cite{dubois2025frank}.

\subsection{Building final solution}

From the output of the Carath\'eodory decomposition algorithm from~\eqref{eq:nonconvex-betaK-zK-caratheodory-output}, we are ready to build final candidates for the primal problem~\eqref{eq:nonconvex-primal}. We distinguish two cases: when the function domains $\domfi$ are convex and when they are not.

\subsubsection{Convex domains}

When the function domains $\domfi$ are convex, a final solution $\bar{x} \in \R^d$ can be constructed as
\begin{align}
\label{eq:final-reconstruction-convex-domain}
    \bar{x}_i = \begin{cases}
        \sum_{l \in Q_i} \tilde{\omega}_i^l x_i^l &\text{ if } i\in S \\
        x_i^{l_i} &\text{ if } i \not \in S
    \end{cases} \quad i=1,\dots, \n.
\end{align}

We next present the final convergence rate for this construction, which recovers the theoretical duality gap from Aubin and Ekeland (Theorem~\ref{thm:duality-gap-estimation-convex-domains}) up to an error term decreasing with the number of iterations of Algorithms~\ref{alg:dual-subgradient-separable-primal} and~\ref{alg:BCFW}. We give the rate in $O$-notation to simplify the exposition.

\begin{thm}
\label{thm:nonconvex-2-stage-convergence}
    Consider the nonconvex problem~\eqref{eq:nonconvex-primal} with $\domfi$ convex and the nonconvexity $\rho(f_i)$ of $f_i$ defined in~\eqref{eq:function-nonconvexity-definition} for all $i=1, \dots, \n$. Assume we apply Algorithm~\ref{alg:2-stage-algorithm-nonconvex} with input $\lambda_0 = 0_m$ and $x_i^0 \in \domfi$ such that $f_i(x_i^0) = f_i^{**}(x_i^0)$ for all $i$, and stepsize $\alpha_t = \frac{\alphabar}{\Gtilde \sqrt{T}}$. The output $\bar{x}$, as constructed in~\eqref{eq:final-reconstruction-convex-domain}, then satisfies
    \begin{align*}
        &\E\left[\frac{1}{\n} \sum_{i=1}^\n f_i(\bar{x}_i) - d^*\right] \leq \frac{m+1}{\n} \max_i \rho(f_i) + O\left( \frac{1}{\sqrt{T}} + \frac{1}{\sqrt{K}} + \frac{\n}{K} \left( \sqrt{\frac{\n}{T}} + \frac{\n}{T} \right) \right)\\
        & \E\norm{\frac{1}{\n} \sum_{i=1}^\n A_i \bar{x}_i - b}_+ \leq O\left( \frac{1}{\sqrt{T}} + \frac{1}{\sqrt{K}} + \frac{\n}{K} \left( \sqrt{\frac{\n}{T}} + \frac{\n}{T} \right) \right).
    \end{align*}
\end{thm}

\begin{proof}
    For primal suboptimality, we have
    \begin{align*}
    \frac{1}{\n} \sum_{i=1}^\n f_i(\bar{x}_i) &= \frac{1}{\n} \left( \sum_{i\in S} f_i\left(\sum_{l \in Q_i} \tilde{\omega}_i^l x_i^l \right) + \sum_{i \not \in S} f_i(x_i^{l_i}) \right)  \\
        &= \frac{1}{\n} \left( \sum_{i\in S} \left( f_i\left(\sum_{l \in Q_i} \tilde{\omega}_i^l x_i^l \right) - \sum_{l\in Q_i} \tilde{\omega}_{i}^l f_i(x_i^l) + \sum_{l\in Q_i} \tilde{\omega}_{i}^l f_i(x_i^l) \right) + \sum_{i \not \in S} f_i(x_i^{l_i}) \right).
        \end{align*}
        Recalling the definition of $\rho(f_i)$ in~\eqref{eq:function-nonconvexity-definition}, we get
        \begin{align*}
        \frac{1}{\n} \sum_{i=1}^\n f_i(\bar{x}_i) &\leq \frac{1}{\n} \left( \sum_{i\in S} \left( \rho(f_i) + \sum_{l\in Q_i} \tilde{\omega}_{i}^l f_i(x_i^l) \right) + \sum_{i \not \in S} f_i(x_i^{l_i}) \right) \\
        &\leq \frac{m+1}{\n} \max_{i} \rho(f_i) + \frac{1}{\n} \left( \sum_{i\in S} \sum_{l\in Q_i} \tilde{\omega}_{i}^l f_i(x_i^l) + \sum_{i \not \in S} f_i(x_i^{l_i}) \right) \\
        &= \frac{m+1}{\n} \max_{i} \rho(f_i) + \beta^K.
    \end{align*}
    Applying Proposition~\ref{prop:nonconvex-two-stage-convergence-beta-z} and hiding the constants in the $O$-notation, we get
    \begin{align*}
        \E[\frac{1}{\n} \sum_{i=1}^\n f_i(\bar{x}_i) - d^*] &\leq \frac{m+1}{\n} \max_i \rho(f_i) + \E[\beta^K - d^*] \\
        &\leq O\left( \frac{1}{\sqrt{T}} + \frac{1}{\sqrt{K}} + \frac{\n}{K} \left( \sqrt{\frac{\n}{T}} + \frac{1}{\sqrt{T}} + \frac{\n}{T} \right) \right).
    \end{align*}
    For infeasibility, we have
    \begin{align*}
        \frac{1}{\n} \sum_{i=1}^\n A_i \bar{x}_i  - b= \frac{1}{\n} \left( \sum_{i\in S}A_i \sum_{k \in J_i} \tilde{\omega}_i^k x_i^k  + \sum_{i \not \in S} A_i x_i^{k_i}\right) - b = z^K - b.
    \end{align*}
    Applying Proposition~\ref{prop:nonconvex-two-stage-convergence-beta-z} just like above yields the result.
\end{proof}

We can then specify the complexity in terms of a given accuracy $\epsilon > 0$.

\begin{cor}
    Consider the setting of Theorem~\ref{thm:nonconvex-2-stage-convergence} and let $\epsilon > 0$. By setting $T = K$, the solution $\bar{x}$ defined in~\eqref{eq:final-reconstruction-convex-domain} satisfies
    \begin{align*}
        &\E[\frac{1}{\n} \sum_{i=1}^\n f_i(\bar{x}_i) - d^*] \leq \frac{m+1}{\n} \max_i \rho(f_i) + \epsilon \ \text{ and } \ \E\norm{\frac{1}{\n} \sum_{i=1}^\n A_i \bar{x}_i - b}_+ \leq \epsilon,
    \end{align*}
after at most $O\left( \frac{1}{\epsilon^2} + \frac{\n}{\epsilon^{2/3}}\right)$ calls to oracle~\eqref{eq:nonconvex-oracle}. An additional computation of  the Carath\'eodory decomposition is also required.
\end{cor}
\begin{proof}
    From Theorem~\ref{thm:nonconvex-2-stage-convergence} with $T = K$, we have
    \begin{align*}
        &\E[\frac{1}{\n} \sum_{i=1}^\n f_i(\bar{x}_i) - d^*] \leq \frac{m+1}{\n} \max_i \rho(f_i) + O\left( \frac{1}{\sqrt{T}} + \frac{\n^{3/2}}{T^{3/2}} + \frac{\n^2}{T^2} \right),\\
        & \E\norm{\frac{1}{\n} \sum_{i=1}^\n A_i \bar{x}_i - b}_+ \leq O\left( \frac{1}{\sqrt{T}} + \frac{\n^{3/2}}{T^{3/2}} + \frac{\n^2}{T^2} \right).
    \end{align*}
    Thus we need to set $T = O\left(\max(1/\epsilon^2, \n/\epsilon^{2/3}, \n/\epsilon^2)\right) = O\left(1/\epsilon^2 + \n/\epsilon^{2/3}\right)$ in order to reach $\epsilon$-precision.
\end{proof}


\begin{remark}
    Let us compare our complexity bound to the results from~\cite{dubois2025frank}, and particularly to Theorem 3.2 in that work, which recovers the same bound $\frac{m+1}{\n} \max_i \rho(f_i)$, up to an error term $O(\frac{1}{\sqrt{K}})$, where $K$ is the number of iterations of the Frank-Wolfe algorithm they use, which requires computing oracle~\eqref{eq:nonconvex-oracle} for all agents $\n$ at each iteration. Therefore, the overall complexity of their method is $O(\n/\epsilon^2)$ to recover an $\epsilon$-approximate duality gap bound. In contrast, we get a better overall complexity $O(1/\epsilon^2 + \n /\epsilon^{2/3})$, without assuming knowledge of $d^*$.
\end{remark}

\subsubsection{Nonconvex domains}

We end this section by discussing the case of nonconvex function domains. We can no longer build a final solution as in~\eqref{eq:final-reconstruction-convex-domain}, as for $i \in S$ we might get candidates not in the domain of $f_i$. There are several ways to go around this issue. We explore one such, which consists in a simple sampling strategy. Duality gap bounds for nonconvex domains are typically based on the following more general notion of nonconvexity~\cite{bertsekas2014constrained}
\begin{align}
\label{eq:general-function-nonconvexity-definition}
    \gamma(f_i) := \sup_{x_i \in \domfi} f_i(x_i) - \inf_{y_i \in \domfi} f_i(y_i).
\end{align}

Recalling the equation for $(\beta^K, z^K)$ in~\eqref{eq:nonconvex-betaK-zK-caratheodory-output}, our sampling strategy then consists in, for each $i \in S$, sampling $\tilde{l}_i \in Q_i$ with probability $\tilde{\omega}^{\tilde{l}_i}_i$ and setting
\begin{align}
\label{eq:final-reconstruction-nonconvex-domain-general}
    \bar{x}_i = \begin{cases}
        x_i^{\tilde{l}_i} &\text{ for } i\in S \\
        x_i^{l_i} &\text{ for } i \not \in S
    \end{cases} \quad i=1,\dots, \n.
\end{align}
We can then prove the following result, similar in spirit to Theorem~\ref{thm:nonconvex-2-stage-convergence}, except that $\rho(f_i)$ is replaced by $\gamma(f_i)$ and that the infeasibility is assessed component-wise.

\begin{thm}
\label{thm:nonconvex-2-stage-convergence-nonconvex-domains}
    Consider the nonconvex problem~\eqref{eq:nonconvex-primal} with nonconvexity $\gamma(f_i)$ of $f_i$ defined as in~\eqref{eq:general-function-nonconvexity-definition} for all $i=1, \dots, \n$. Assume we apply Algorithm~\ref{alg:2-stage-algorithm-nonconvex} with input $\lambda_0 = 0_m$ and $x_i^0 \in \domfi$ such that $f_i(x_i^0) = f_i^{**}(x_i^0)$ for all $i=1,\dots, \n$, and stepsize $\alpha_t = \alpha = \frac{\alphabar}{\Gtilde \sqrt{T}}$. The output $\bar{x}$, as constructed in~\eqref{eq:final-reconstruction-nonconvex-domain-general}, then satisfies
    \begin{align*}
        &\E\left[\frac{1}{\n} \sum_{i=1}^\n f_i(\bar{x}_i) - d^*\right] \leq \frac{m+1}{\n} \max_i \gamma(f_i) + O\left( \frac{1}{\sqrt{T}} + \frac{1}{\sqrt{K}} + \frac{\n}{K} \left( \sqrt{\frac{\n}{T}} + \frac{\n}{T} \right) \right)\\
        & \E\left[\frac{1}{\n} \sum_{i=1}^\n A_i \bar{x}_i - b \right] \leq   O\left( \frac{1}{\sqrt{T}} + \frac{1}{\sqrt{K}} + \frac{\n}{K} \left( \sqrt{\frac{\n}{T}} + \frac{\n}{T} \right) \right).
    \end{align*}
\end{thm}
\begin{proof}
    For primal suboptimality, we have
    \begin{align*}
        \frac{1}{\n} \sum_{i=1}^\n f_i(\bar{x}_i) &= \frac{1}{\n} \left( \sum_{i\in S} f_i\left(x_i^{\tilde{l}_i} \right) + \sum_{i \not \in S} f_i(x_i^{l_i}) \right)  \\
        &= \frac{1}{\n} \left( \sum_{i\in S} \left( f_i\left(x_i^{\tilde{l}_i} \right) - \sum_{l\in Q_i} \tilde{\omega}_{i}^l f_i(x_i^l) + \sum_{l\in Q_i} \tilde{\omega}_{i}^l f_i(x_i^l) \right) + \sum_{i \not \in S} f_i(x_i^{l_i}) \right) \\
        &\leq \frac{1}{\n} \left( \sum_{i\in S} \left( \gamma(f_i) + \sum_{l\in Q_i} \tilde{\omega}_{i}^l f_i(x_i^l) \right) + \sum_{i \not \in S} f_i(x_i^{l_i}) \right) \\
        &\leq \frac{m+1}{\n} \max_{i} \gamma(f_i) + \frac{1}{\n} \left( \sum_{i\in S} \sum_{l\in Q_i} \tilde{\omega}_{i}^l f_i(x_i^l) + \sum_{i \not \in S} f_i(x_i^{l_i}) \right) \\
        &= \frac{m+1}{\n} \max_{i} \gamma(f_i) + \beta^K.
    \end{align*}
    The proof can then proceed exactly as in the proof of Theorem~\ref{thm:nonconvex-2-stage-convergence}.
    For infeasibility, let $\Ethree$ be the expectation with respect to the randomness induced by~\eqref{eq:final-reconstruction-nonconvex-domain-general}. Note that we have
    \begin{align*}
        \Ethree\left[  \frac{1}{\n} \sum_{i=1}^\n A_i \bar{x}_i  - b\right] = z^K - b \leq \norm{z^K - b}_+,
    \end{align*}
    where the inequality is to be understood component-wise. Taking total expectation, we have
    \begin{align*}
        \E \left[ \frac{1}{\n} \sum_{i=1}^\n A_i \bar{x}_i  - b\right] \leq \E \left[\norm{z^K - b}_+\right].
    \end{align*}
    The proof can then proceed exactly as in the proof of Theorem~\ref{thm:nonconvex-2-stage-convergence}
\end{proof}

To finish, let us point out that the approach proposed in this section also applies to other ways of handling the nonconvexity of the function domains. For example, one might require stronger assumptions on the domains~\cite{vujanic2014large}, or solve a perturbed problem to ensure that the reconstruction is fully feasible~\cite{vujanic2016decomposition,dubois2025frank}.

\section{Numerical results}

We consider the problem of charging a fleet of electric vehicles over a given time horizon, under constraints on the maximum consumption of the overall network. Let us first give a complete description of the problem, inspired by~\cite{vujanic2016decomposition}. The network consists of $\n$ electric vehicles which must be charged over $m$ timesteps of length $\Delta$. The constraints on each vehicle can be summarized as follows. Each vehicle $i$ has an initial level of charge  $E_i^{\text{init}}$, a required level of charge $E_i^{\text{ref}}$ that must be met by the end of the charging session, and a maximum level of energy $E_i^{\text{max}}$. Each vehicle also has a charging rate $P_i$. The variable $x_i^j \in \{0, 1\}$ indicates whether or not vehicle $i$ is charging at timestep $j$. The dynamics of the level of charge $e_i\in \R^m$ is then given by
\begin{align*}
    e_i^{j+1} = e_i^j + P_i \Delta \xi_i x_i^j, 
\end{align*}
where $\xi_i < 1$ is a coefficient reflecting the energy loss.

The coupling constraint forces the total charging procedure not to impose too much demand on the network, i.e.,
\begin{align*}
    \sum_{i=1}^\n P_i x_i^j \leq P^{\text{max}}, \ j=1,\dots, m,
\end{align*}
where $P^{\text{max}}$ is the maximum power that can be allocated to the charging of the fleet. Finally, a vector $C^u \in \R^m$ gives the price of electricity over time. The overall optimization problem can then be written as
\begin{align}
\label{eq:numerical-primal}
\begin{split}
    \mbox{minimize} &\ \frac{1}{\n} \sum_{i=1}^\n f_i(x_i) \\
    \mbox{subject to} & \ \sum_{i=1}^\n P_i x_i \leq P^{\text{max}}, \\
    &\ x_i \in \domfi,  \ i=1,\dots, \n,
    \end{split}
\end{align}
where $f_i(x_i) = \sum_{j=1}^m P_i C^{u}_j x_i^j$ and
\begin{align*}
    \domfi = \left\{ x \in \{0, 1\}^m\ \middle\vert \begin{array}{l}
    E^{\text{init}}_i + \sum_{j=1}^{T}P_i \Delta \xi_i x_i^j \geq E^{\text{ref}}_i, \\
    E^{\text{init}}_i + \sum_{j=1}^{t}P_i \Delta \xi_i x_i^j \leq E^{\text{max}}_i, \ t=1, \dots, T
  \end{array}\right\}.
\end{align*}

This is a nonconvex problem of the form~\eqref{eq:nonconvex-primal}, whose bidual reads
\begin{align}
\label{eq:numerical-bidual}
\begin{split}
    \mbox{minimize} &\ \frac{1}{\n} \sum_{i=1}^\n f_i(x_i) \\
    \mbox{subject to} & \ \sum_{i=1}^\n P_i x_i \leq P^{\text{max}}, \\
    &\ x_i \in \conv(\domfi), \ i=1,\dots, \n.
\end{split}
\end{align}

We shall therefore compare our two-stage algorithm with the dual subgradient method on the convex problem~\eqref{eq:numerical-bidual}, validating the theoretical results from Section~\ref{sec:proposed-algorithm} by doing so. Moreover, the approach to nonconvexity from Section~\ref{sec:nonconvex} also applies to the dual subgradient algorithm, so we will apply a Carath\'eodory algorithm to the outputs of both our two-stage algorithm and the dual subgradient algorithm, and compare the performances on the nonconvex problem~\eqref{eq:numerical-primal}.

\subsection{Data generation and implementation details} We set $\n = 10^4$ and $T = 24$. The full data generation procedure can be found in~\cite[Appendix B]{vujanic2016decomposition}. To compute the optimal dual value, we run the stochastic dual subgradient algorithm for a large amount of time and choose the largest dual value obtained as an approximate optimal dual value. For both our two-stage algorithm and the dual subgradient algorithm, we set the stepsize as $\alpha_t = \frac{\alphabar}{\sqrt{t}}$, and we do a grid search to pick the value of $\alphabar$ achieving the best convergence. We compare the algorithms with the same number of calls to oracle~\eqref{eq:oracle-definition}. Since our two-stage algorithm is stochastic, we average the results over 5 runs. Finally, to recover a solution to the nonconvex primal problem, we use the min-norm point algorithm \cite{wolfe1970convergence} to compute an approximate Carath\'eodory decomposition, and build a final solution as in \eqref{eq:final-reconstruction-nonconvex-domain-general}. The full code to reproduce the experiment is available at: \url{https://github.com/bpauld/SeparableOpt}.

\subsection{Results on convex bidual}

Figure \ref{fig:cvx_convergence} plots the convergence of our algorithm on the dual and the bidual problem. The deterministic subgradient algorithm convergence is plotted in red, while our two-stage algorithm is represented by two curves: the stochastic dual subgradient algorithm in blue followed by the BCFW algorithm in green, along with the standard deviation in shaded color. The top plots show that the stochastic subgradient algorithm significantly outperforms its deterministic counterpart for convergence in the dual, as is expected from the results in Propositions~\ref{prop:epsilon-complexity-dual-subgradient} and~\ref{prop:epsilon-complexity-stochastic-dual-subgradient}. The middle left plot shows the evolution of what we call the bidual gap, which we define as the difference between the cost function of problem~\eqref{eq:numerical-bidual} and its optimal dual value. The reason we do not see some of the curves is that the bidual iterates produced by the algorithms often have a bidual value smaller than the optimal dual value. While this might seem unintuitive, this is due to the fact that those bidual iterates slightly violate the coupling constraints, as can be seen in the middle right plot. Therefore, the most telling plots are the bottom ones, which show the evolution of the sum of the positive part of the bidual gap and the infeasibility as a function of the number of oracle calls. We see that our proposed two-stage algorithm significantly outperforms the deterministic dual subgradient algorithm. It is also interesting to note that the stochastic dual subgradient alone also outperforms its deterministic counterpart.

\begin{figure}
    \centering
    \includegraphics[scale=0.3]{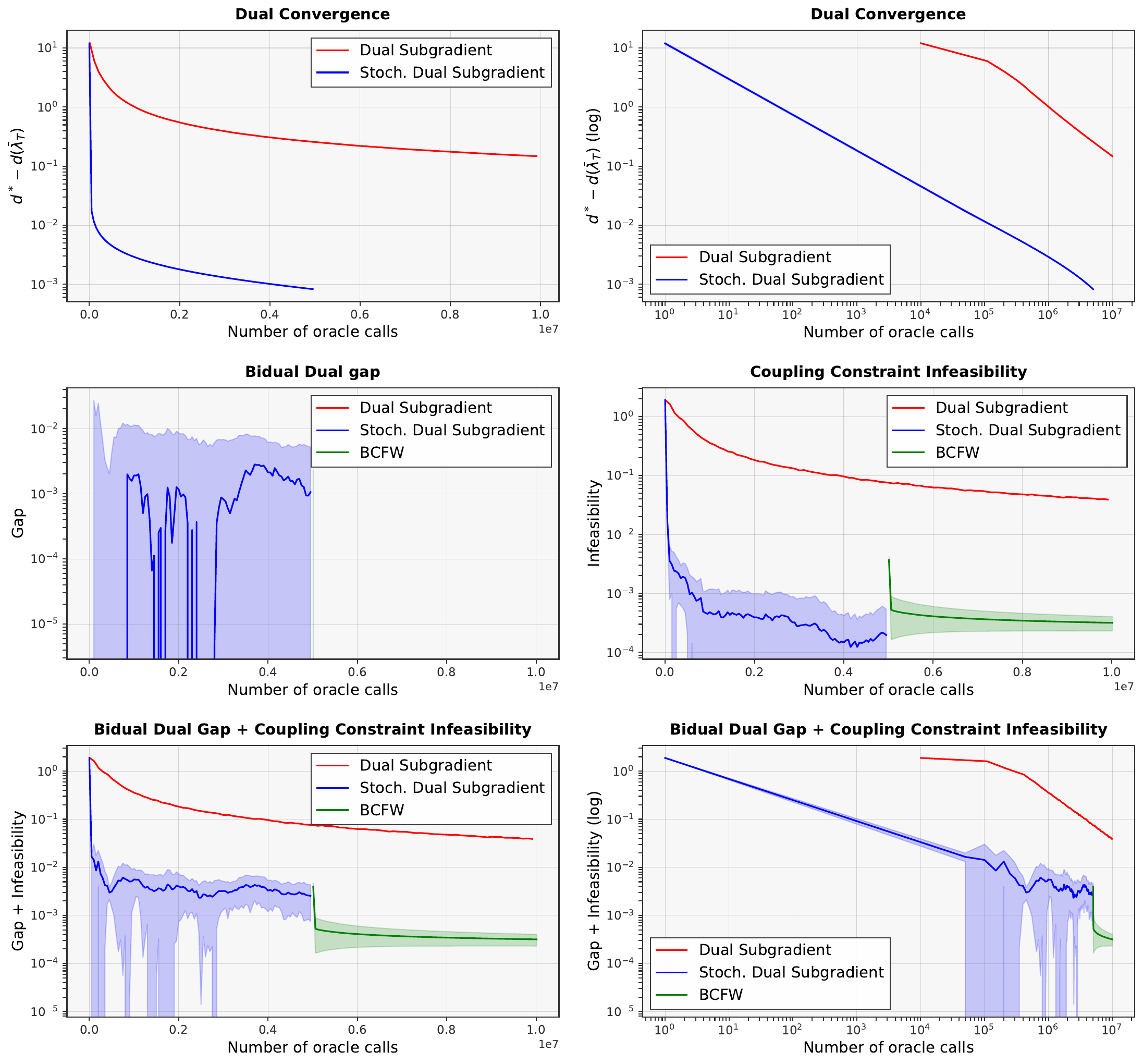}
    \caption{Convergence of the deterministic subgradient algorithm convergence (red) and of our two-stage algorithm (stochastic dual subgradient in blue, BCFW in green). \textbf{Top left}: convergence of the dual objective. \textbf{Top right}: same as top left in log-log scale. \textbf{Middle left}: bidual gap. Some curves are missing as the bidual iterates have a strictly negative bidual gap. This is due to the violation of the coupling constraint. \textbf{Middle right}: violation of the coupling constraint by the bidual iterates. \textbf{Bottom left:} sum of the bidual gap and of the coupling constraint violation. We see that our two-stage algorithm significantly outperforms the deterministic dual subgradient algorithm. \textbf{Bottom right:} same as bottom left in log-log scale.}
    \label{fig:cvx_convergence}
\end{figure}

\subsection{Results on nonconvex primal}

Figure~\ref{fig:noncvx_convergence} plots the convex and nonconvex behavior of our algorithm. More precisely, for different values of the number of oracle calls, blue-green circles represent the performance of the bidual solution of our two-stage algorithm, and black crosses represent the performance of the primal solution after reconstruction using the Carath\'eodory algorithm. Similarly, for the dual subgradient algorithm, the bidual solution is in red circles and the primal solution in red crosses. On the left plot, we see that despite the nonconvexity of the primal problem~\eqref{eq:numerical-primal}, the reconstructed solution is almost as good as the solution to the bidual convex problem~\eqref{eq:numerical-bidual}. This is expected from our theoretical duality gap bounds derived in Theorem~\ref{thm:nonconvex-2-stage-convergence-nonconvex-domains} and the fact that $\n >> m$ in this example. Note that this is also true for the dual subgradient algorithm, although it is harder to see on the plot because of the slower convergence of this algorithm. The right plot is the same as the left one, except that we have removed the dual subgradient curve, have put the x-axis in logarithmic scale and have added a (scaled) $O(1/\sqrt{K})$ curve. This allows us to observe that the convergence of our algorithm on the bidual behaves like the inverse of the square root of the number of calls to the oracle~\eqref{eq:oracle-definition}, as is expected from the results in Theorem~\ref{thm:convergence-2-stage-algorithm}.

\begin{figure}
    \centering
    \includegraphics[scale=0.3]{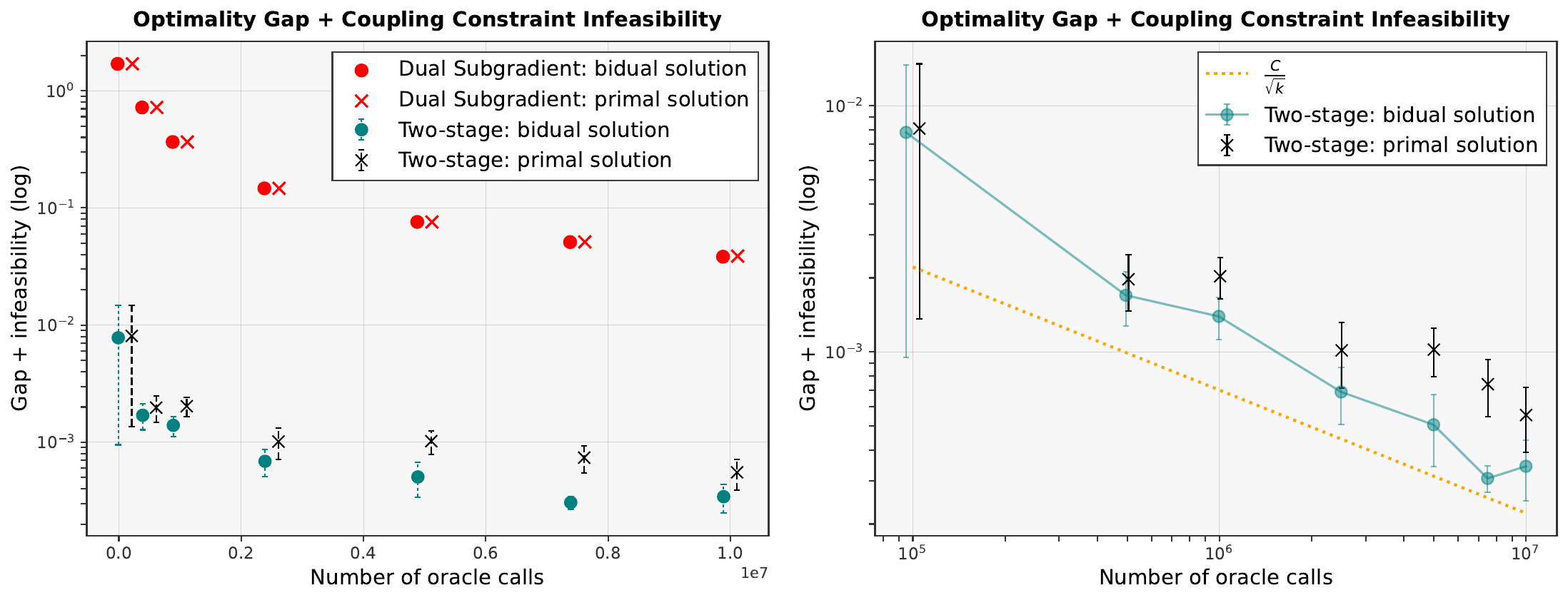}
    \caption{\textbf{Left}. Performance of our two-stage algorithm in the bidual (blue-green circles) and primal (black crosses) compared with the performance of the dual subgradient in the bidual (red circles) and in the primal (red crosses). \textbf{Right}. Same plot in logarithmic x-scale with the additional orange curve being the inverse of the square root of the number of oracle calls. We see that (i) the solution to the nonconvex primal problem is almost as good as the solution to the bidual convex problem, as is expected from the duality gap bounds derived in~\ref{thm:nonconvex-2-stage-convergence-nonconvex-domains} and (ii) the convergence is of the order of the inverse of the square root of the number of oracle calls, as is expected from Theorems~\ref{thm:convergence-2-stage-algorithm} and~\ref{thm:nonconvex-2-stage-convergence-nonconvex-domains}.}
    \label{fig:noncvx_convergence}
\end{figure}

\appendix

\section{Additional proofs}

\subsection{Results on stochastic dual subgradient}
\label{sec:technical-results}
\SDSGDBoundIterates*
\begin{proof}
    For $t=0, \dots, T-2$, we have
    \begin{align*}
        \norm{\lambda_{t+1} - \lambda^*}^2 &=  \norm{\lambda_{t} + \alpha g_t - \lambda^*}^2 \\
        &= \norm{\lambda_t - \lambda^*}^2 + 2\alpha \iprod{g_t}{\lambda_t - \lambda^*} + \alpha^2 \norm{g_t}^2 \\
        &= \norm{\lambda_t - \lambda^*}^2 + 2\alpha \iprod{g_t}{\lambda_t - \lambda^*} + \alpha^2 \norm{A_{i_t}x_{i_t}^*(\lambda_t) - b}^2\\
        &\leq \norm{\lambda_t - \lambda^*}^2 + 2\alpha \iprod{g_t}{\lambda_t - \lambda^*} + \alpha_t^2 \Gtilde^2.
    \end{align*}
    Taking expectation with respect to iteration $t$, observe that we have
    \begin{align*}
        \E[g_t] = \sum_{i=1}^\n \frac{1}{\n} A_i x_i^*(\lambda_t) - b \in \partial d(\lambda_t).
    \end{align*}
    Therefore we have, for $t=0, \dots, T-2$,
    \begin{align}
        \E\norm{\lambda_{t+1} - \lambda^*}^2 &\leq \norm{\lambda_t - \lambda^*}^2 + 2\alpha (d(\lambda_t) - d(\lambda^*)) + \alpha^2 \Gtilde^2 \nonumber \\
        &\leq \norm{\lambda_t - \lambda^*}^2  + \alpha^2 \Gtilde^2
        \label{eq:proof-stochastic-dual-subgradient-bounded-iterates-1}
    \end{align}
    by concavity of $d$ and optimality of $d(\lambda^*)$. Now, for the last iteration $t=T-1$, we have
    \begin{align*}
        \norm{\lambda_{T} - \lambda^*}^2 &=  \norm{\lambda_{T-1} + \alpha g_{T-1} - \lambda^*}^2 \\
        &= \norm{\lambda_{T-1} - \lambda^*}^2 + 2\alpha \iprod{g_t}{\lambda_{T-1} - \lambda^*} + \alpha^2 \norm{g_{T-1}}^2 \\
        &= \norm{\lambda_{T-1} - \lambda^*}^2 + 2\alpha \iprod{g_{T-1}}{\lambda_{T-1} - \lambda^*} + \alpha^2 \norm{\frac{1}{\n}\sum_{i=1}^\n A_i x_i^*(\lambda_t) - b}^2\\ 
        &\leq  \norm{\lambda_{T-1} - \lambda^*}^2 + 2\alpha (d(\lambda_{T-1}) - d(\lambda^*))  + \alpha^2 G^2,
    \end{align*}
    by concavity of $d$. Note that in that case we did not need to take expectation since $g_{T-1} \in \partial d(\lambda_{T-1})$. Since $G \leq \Gtilde$, combining with~\eqref{eq:proof-stochastic-dual-subgradient-bounded-iterates-1} and taking total expectation, we get that for any $t=0,\dots, T-1,$
    \begin{align*}
\E\norm{\lambda_{t+1} - \lambda^*}^2 &\leq \E\norm{\lambda_t - \lambda^*}^2 + \alpha^2 \Gtilde^2.
\end{align*}
Recursing yields the result.
\end{proof}

\subsection{Results concerning the nonconvex framework}
\label{app:nonconvex-results}
\ExtremePointsMatching*

\begin{proof}
For conciseness, we drop the subscripts $i$ throughout the proof and denote as $A$ the linear mapping associated with the matrix $A$, without risk of confusion.
We have $\epi f^{**} = \conv (\epi f)$~\cite[Theorem X.1.3.5 and Lemma X.5.3.1]{hiriart1996convex}. This immediately implies that $\dom(f^{**})= \conv(\dom(f))$, so the minimization problem in \eqref{eq:nonconvex-oracle} is nothing but a minimization of $\gamma f^{**} + \lambda^\top A$.

Assume that $\gamma=0$. We have $(\tilde{x}, f^{**}(\tilde{x})) \in \epi f^{**}$ and we already know that $ \epi f^{**} = \conv (\epi f)$. Therefore, there exist $J \in \mathbb{N}$, $p \in \R^J$, such that
\begin{equation*}
(\tilde{x},f^{**}(\tilde{x}))
= \sum_{j=1}^J p_j (x^j,\beta^j)
\end{equation*}
with $(x^j, \beta^j) \in \epi f$, $p_j > 0$ for all $ j=1,\ldots J$, and $\sum_{j=1}^J p_j= 1$. By assumption, $\tilde{x}$ is an extreme point of $\dom(f)$, so all points $x^j$ are equal to $\tilde{x}$. We deduce that
\begin{equation*}
f^{**}(\tilde{x})
= \sum_{j=1}^J p_j \beta^j
\geq \sum_{j=1}^J p_j f(\tilde{x})
= f(\tilde{x}).
\end{equation*}
Since we also have $f^{**}\leq f$, we conclude that $f^{**}(\tilde{x})= f(\tilde{x})$, as announced.

Let us now assume that $\gamma>0$. Dividing both optimization problems \eqref{eq:nonconvex-oracle} and \eqref{eq:oracle_nc} by $\gamma$, we see that it is sufficient to treat the case where $\gamma=1$. Define $\hat{f} = f + \lambda^\top A$. Direct calculations show that $\hat{f}^{**}= f^{**}+\lambda^\top A$.
 Let us now show that $\tilde{x}$ minimizes $\hat{f}^{**}$.
Since $\tilde{x}$ minimizes $\hat{f}$, it holds $\hat{f}^*(0)= \sup_{y} y^\top 0 - f(y) = -\inf_y f(y) =- \hat{f}(\tilde{x})$. Thus for any $x$, it holds that $\hat{f}^{**}(x) \geq 0^\top x- \hat{f}^*(0)= \hat{f}(\tilde{x})$.
Recalling that $\hat{f}^{**} \leq \hat{f}$, we deduce that $\hat{f}^{**}(\tilde{x})
= \hat{f}(\tilde{x})$. Since $\hat{f}(\tilde{x})$ is a lower bound on $\hat{f}^{**}$, we obtain that $\hat{f}^{**}(\tilde{x}) = \hat{f}(\tilde{x}) = \inf \hat{f}^{**}$, and that $\tilde{x}$ is a solution to \eqref{eq:nonconvex-oracle}.
We also have
$f^{**}(\tilde{x})= \hat{f}^{**}(\tilde{x})- \lambda^\top A \tilde{x}= \hat{f}(\tilde{x})- \lambda^\top A \tilde{x}= f(\tilde{x})$, which concludes the proof.
\end{proof}

\PropNonConvexTwoStageConvergenceBetaZ*
\begin{proof}
The proof closely resembles the proof of Theorem~\ref{thm:convergence-2-stage-algorithm}. Indeed, from equation~\eqref{eq:nonconvex-betaK-zK-rate-randomness-alg2},
\begin{align}
\label{eq:nonconvex-proof-1}
\begin{split}
    &\Etwo[\beta^K -d^* ] \leq (d^* - d(\bar{\lambda}_T)) + \frac{2\SqrtAvgDiamSquared}{\sqrt{K }} + \frac{2\n}{K} \sqrt{F((\beta^0, z^0)) - F^*},\\
        &\Etwo\norm{z^K - b }_+ \leq (d^* - d(\bar{\lambda}_T)) + \frac{2\SqrtAvgDiamSquared}{\sqrt{K }} + \frac{2\n}{K} \sqrt{F((\beta^0, z^0)) - F^*}.
\end{split}
\end{align}
We also already have a bound on $\Eone \left[ d^* - d(\bar{\lambda}_T)\right]$, i.e.
\begin{align}
\label{eq:nonconvex-proof-2}
    \Eone \left[ d^* - d(\bar{\lambda}_T)\right] \leq \frac{\Gtilde \norm{\lambda^*}^2}{2\alphabar\sqrt{T}} + \frac{\Gtilde \alphabar}{2\sqrt{T}}.
\end{align}
The only difference with the result from~\ref{thm:convergence-2-stage-algorithm} is actually the different definition of $\beta^0$. We will prove now that this is actually not an issue. We have
    \begin{align*}
        \Eone \left[ \sqrt{F((\beta^0, z^0)) - F^*}\right] &\leq 
        \Eone \left[ \sqrt{F((\beta^0, z^0))}\right] \\
        &= \Eone  
        \left[\sqrt{\frac{1}{2} \max( \sum_{i=1}^\n \beta_i^0  - d(\bar{\lambda}_T), 0)^2 + \frac{1}{2} \norm{\frac{1}{\n} \sum_{i=1}^\n A_i \bar{x}_{i, T} - b }_+^2} \right] \\
        &\leq \frac{1}{\sqrt{2}} \left(\Eone \left[\max(  \sum_{i=1}^\n \beta_i^0 - d(\bar{\lambda}_T), 0)\right] + \Eone \left[\norm{\frac{1}{\n} \sum_{i=1}^\n A_i \bar{x}_{i, T} - b }_+\right] \right).
    \end{align*}
    To simplify notation, let us observe that from~\eqref{eq:beta0-nonconvex-definition} and~\eqref{eq:z0-nonconvex-definition}, we can write
    \begin{align*}
        \beta_i^0 &= \sum_{t=0}^{T-1} \nu_{i, t} \frac{1}{\n}f_i(x_i^*(\lambda_t))\ \text{ and } \ 
        z_i^0 = \sum_{t=0}^{T-1} \nu_{i, t} \frac{1}{\n}A_ix_i^*(\lambda_t),
    \end{align*}
    with $\nu_{i, t} \geq 0$, $\sum_{t=0}^{T-1} \nu_{i, t} = 1$, and allowing some terms $\nu_{i, t}$ to be $0$. We have
    \begin{align*}
        \sum_{i=1}^\n \beta_i^0  &= \frac{1}{\n} \sum_{i=1}^\n \sum_{t=0}^{T-1} \nu_{i, t} f_i(x_i^*(\lambda_t)) \\
        &= \frac{1}{\n} \sum_{i=1}^\n \sum_{t=0}^{T-1} \nu_{i, t} \left( f_i(x_i^*(\lambda_t)) + (\lambda^*)^\top (A_i x_i^*(\lambda_t) - b) \right) - \frac{1}{\n} \sum_{i=1}^\n \sum_{t=0}^{T-1} \nu_{i, t} (\lambda^*)^\top (A_i x_i^*(\lambda_t) - b).
\end{align*}
Observe that the terms $\sum_{t=0}^{T-1} \nu_{i, t} \left( f_i(x_i^*(\lambda_t)) + (\lambda^*)^\top (A_i x_i^*(\lambda_t) - b)\right)$ can be seen as the expectation of $f_i(x_i^*(\lambda_t)) + (\lambda^*)^\top (A_i x_i^*(\lambda_t) - b)$ with respect to the probability vector $\nu_i = (\nu_{i, 0}, \dots, \nu_{i, T-1})$. In particular, we may write
        \begin{align*}
        \sum_{i=1}^\n \beta_i^0  &= \E_{\nu_1, \dots, \nu_{\n}} \left[ \underbrace{\sum_{i=1}^\n \frac{1}{\n}\left( f_i(x_i^*(\lambda_t)) + (\lambda^*)^\top (A_i x_i^*(\lambda_t) - b) \right)}_{\geq d(\lambda^*)}\right] - (\lambda^*)^\top (z^0 - b) \\
        &\geq d^* - (\lambda^*)^\top \ProjPos{z^0 - b} \tag{since $\lambda^* \geq 0$}.
    \end{align*}
    Therefore,
    \begin{align*}
        \sum_{i=1}^\n \beta_i^0  - d(\bar{\lambda}_T) \geq d^* - d(\bar{\lambda}_T) - (\lambda^*)^\top \ProjPos{z^0 - b} \geq - (\lambda^*)^\top \ProjPos{z^0 - b}.
    \end{align*}
    By Lemma~\ref{lem:max-inequality}, we thus get
    \begin{align*}
        \max\left( \sum_{i=1}^\n \beta_i^0  - d(\bar{\lambda}_T),0\right) &\leq \sum_{i=1}^\n \beta_i^0  - d(\bar{\lambda}_T) + \abs{(\lambda^*)^\top \ProjPos{z^0 - b}} \\
        &\leq \sum_{i=1}^\n \beta_i^0  - d(\bar{\lambda}_T) + \norm{\lambda^*} \norm{z^0 - b}_+ \\
        &= \sum_{i=1}^\n \beta_i^0  - d(\bar{\lambda}_T) + \norm{\lambda^*} \norm{\frac{1}{\n}\sum_{i=1}^\n A_i \bar{x}_{i, T} - b}_+\\
        &= \sum_{i=1}^\n \beta_i^0  - d^* +  (d^* - d(\bar{\lambda}_T)) + \norm{\lambda^*} \norm{\frac{1}{\n}\sum_{i=1}^\n A_i \bar{x}_{i, T} - b}_+.
    \end{align*}
    We already have a bound on the last two terms from Propositions~\ref{prop:dual-convergence-stochastic-dual-subgradient} and~\ref{prop:primal-convergence-stochastic-dual-subgradient}, so it only remains to bound
    \begin{align*}
        \sum_{i=1}^\n \beta_i^0  - d^* &= \sum_{i=1}^\n \left( \frac{1}{I_i + 1} \left( \frac{1}{\n}f_i\left(x_i^*(\lambda_{T-1})\right) +  \sum_{t\in \{0, \dots, T-2\}, \atop i_t = i} \frac{1}{\n}f_i\left( x_i^*(\lambda_t)\right)\right)\right) - d^*.
    \end{align*}
    It may seem at first that we do not have a bound on this quantity from Proposition~\ref{prop:primal-convergence-stochastic-dual-subgradient}, but looking back at the proof of this proposition, and especially equation~\eqref{eq:proof-stochastic-dual-subgradient-primal-convergence-important-result}, we can see that this is precisely this quantity that we end up bounding in the proof. In particular,
    \begin{align*}
        \Eone \left[ \sum_{i=1}^\n \beta_i^0  - d^* \right] \leq  H \sqrt{\frac{\n - 1}{T}} + \frac{\alphabar\Gtilde}{2\sqrt{T}} + \frac{\n \Gtilde(2\norm{\lambda^*}  + \alphabar)}{T}.
    \end{align*}
    Putting everything together,
         \begin{align*}
        \Eone [\sqrt{F((\beta^0, z^0))}] \leq& \ \frac{1}{\sqrt{2}} \left( \Eone \left[ \sum_{i=1}^\n \beta_i^0 - d^* + d^* -  d(\bar{\lambda}_T) +  \norm{\lambda^*} \norm{\frac{1}{\n} \sum_{i=1}^\n A_i \bar{x}_{i,T} - b}_+\right]\right)  \\ & + \frac{1}{\sqrt{2}}\Eone \norm{\frac{1}{\n} \sum_{i=1}^\n A_i \bar{x}_{i,T} - b}_+ \\
        \leq &\ \frac{1}{\sqrt{2}} \left( \frac{\Gtilde \norm{\lambda^*}^2}{2\alphabar \sqrt{T}} + \frac{\Gtilde \alphabar}{2 \sqrt{T}}\right) + \frac{1}{\sqrt{2}} \left( H \sqrt{\frac{N-1}{T}} + \frac{\alphabar \Gtilde }{2 \sqrt{{T}}}  + \frac{\n \Gtilde( 2\norm{\lambda^*} + \alphabar)}{T}\right)   \\ &+ \frac{\norm{\lambda^*} + 1}{\sqrt{2}}  \left( \Gtilde \sqrt{\frac{N-1}{T}} + \frac{ 2\Gtilde \norm{\lambda^*}} {\alphabar\sqrt{T}} + \frac{\Gtilde}{\sqrt{T}}  + \frac{\n \Gtilde }{T} \right).
    \end{align*}
    Combining the above with~\eqref{eq:nonconvex-proof-1} and~\eqref{eq:nonconvex-proof-2} yields the result.
\end{proof}

\subsection{Technical results}

\begin{lem}
\label{lem:max-inequality}
    For any $a,c \in \R$ such that $a \geq c$, it holds that $\max(a, 0) \leq a + \abs{c}$.
\end{lem}
\begin{proof}
    If $a \geq 0$, this is obvious. If $ a < 0$, then also $c < 0$ so that $\abs{c} = -c$. Therefore $a + \abs{c} = a - c \geq 0 = \max(a, 0)$.
\end{proof}

\bibliographystyle{spmpsci}      
\bibliography{bib} 

\begin{thebibliography}{10}
\providecommand{\url}[1]{{#1}}
\providecommand{\urlprefix}{URL }
\expandafter\ifx\csname urlstyle\endcsname\relax
  \providecommand{\doi}[1]{DOI~\discretionary{}{}{}#1}\else
  \providecommand{\doi}{DOI~\discretionary{}{}{}\begingroup \urlstyle{rm}\Url}\fi

\bibitem{aubin1976estimates}
Aubin, J.P., Ekeland, I.: Estimates of the duality gap in nonconvex optimization.
\newblock Mathematics of Operations Research \textbf{1}(3), 225--245 (1976)

\bibitem{baumann2013portfolio}
Baumann, P., Trautmann, N.: Portfolio-optimization models for small investors.
\newblock Mathematical Methods of Operations Research \textbf{77}, 345--356 (2013)

\bibitem{beck20141}
Beck, A., Nedi{\'c}, A., Ozdaglar, A., Teboulle, M.: An o(1/k) gradient method for network resource allocation problems.
\newblock IEEE Transactions on Control of Network Systems \textbf{1}(1), 64--73 (2014)

\bibitem{beck2017linearly}
Beck, A., Shtern, S.: Linearly convergent away-step conditional gradient for non-strongly convex functions.
\newblock Mathematical Programming \textbf{164}, 1--27 (2017)

\bibitem{bertsekas1983optimal}
Bertsekas, D., Lauer, G., Sandell, N., Posbergh, T.: Optimal short-term scheduling of large-scale power systems.
\newblock IEEE Transactions on Automatic Control \textbf{28}(1), 1--11 (1983)

\bibitem{bertsekas2003convex}
Bertsekas, D., Nedic, A., Ozdaglar, A.: Convex analysis and optimization, vol.~1.
\newblock Athena Scientific (2003)

\bibitem{bertsekas1997nonlinear}
Bertsekas, D.P.: Nonlinear programming.
\newblock Journal of the Operational Research Society \textbf{48}(3), 334--334 (1997)

\bibitem{bertsekas2014constrained}
Bertsekas, D.P.: Constrained optimization and Lagrange multiplier methods.
\newblock Academic Press (2014)

\bibitem{besanccon2025pivoting}
Besan{\c{c}}on, M., Pokutta, S., Wirth, E.S.: The pivoting framework: Frank-{W}olfe algorithms with active set size control.
\newblock In: International Conference on Artificial Intelligence and Statistics, pp. 271--279. PMLR (2025)

\bibitem{bi2016refined}
Bi, Y., Tang, A.: Refined {S}hapley-{F}olkman lemma and its application in duality gap estimation.
\newblock arXiv preprint arXiv:1610.05416  (2016)

\bibitem{boyd2011distributed}
Boyd, S., Parikh, N., Chu, E., Peleato, B., Eckstein, J., et~al.: Distributed optimization and statistical learning via the alternating direction method of multipliers.
\newblock Foundations and Trends{\textregistered} in Machine learning \textbf{3}(1), 1--122 (2011)

\bibitem{combettes2023revisiting}
Combettes, C.W., Pokutta, S.: Revisiting the approximate {C}arath{\'e}odory problem via the {F}rank-{W}olfe algorithm.
\newblock Mathematical Programming \textbf{197}(1), 191--214 (2023)

\bibitem{dubois2025frank}
Dubois-Taine, B., d’Aspremont, A.: Frank-{W}olfe meets {S}hapley-{F}olkman: a systematic approach for solving nonconvex separable problems with linear constraints.
\newblock Mathematical Programming pp. 1--51 (2025)

\bibitem{ekeland1999convex}
Ekeland, I., Temam, R.: Convex analysis and variational problems.
\newblock SIAM (1999)

\bibitem{ermol1966methods}
Ermol'ev, Y.M.: Methods of solution of nonlinear extremal problems.
\newblock Cybernetics \textbf{2}(4), 1--14 (1966)

\bibitem{gabay1976dual}
Gabay, D., Mercier, B.: A dual algorithm for the solution of nonlinear variational problems via finite element approximation.
\newblock Computers \& {M}athematics with applications \textbf{2}(1), 17--40 (1976)

\bibitem{glowinski1975approximation}
Glowinski, R., Marroco, A.: Sur l'approximation, par {\'e}l{\'e}ments finis d'ordre un, et la r{\'e}solution, par p{\'e}nalisation-dualit{\'e} d'une classe de probl{\`e}mes de dirichlet non lin{\'e}aires.
\newblock Revue {F}ran{\c{c}}aise d'{A}utomatique, {I}nformatique, {R}echerche {O}p{\'e}rationnelle. Analyse {N}um{\'e}rique \textbf{9}(R2), 41--76 (1975)

\bibitem{gustavsson2015primal}
Gustavsson, E., Patriksson, M., Str{\"o}mberg, A.B.: Primal convergence from dual subgradient methods for convex optimization.
\newblock Mathematical Programming \textbf{150}, 365--390 (2015)

\bibitem{hiriart1996convex}
Hiriart-Urruty, J.B., Lemar{\'e}chal, C.: Convex analysis and minimization algorithms, vol. 305.
\newblock Springer {S}cience \& {B}usiness {M}edia (1996)

\bibitem{kiwiel2001parallel}
Kiwiel, K., Lindberg, P.: Parallel subgradient methods for convex optimization.
\newblock In: Studies in Computational Mathematics, vol.~8, pp. 335--344. Elsevier (2001)

\bibitem{lacoste2013block}
Lacoste-Julien, S., Jaggi, M., Schmidt, M., Pletscher, P.: Block-coordinate {F}rank-{W}olfe optimization for structural {SVM}s.
\newblock In: International Conference on Machine Learning, pp. 53--61. PMLR (2013)

\bibitem{larsson1997lagrangean}
Larsson, T., Liu, Z.: A {L}agrangean relaxation scheme for structured linear programs with application to multicommodity network flows.
\newblock Optimization \textbf{40}(3), 247--284 (1997)

\bibitem{larsson1996ergodic}
Larsson, T., Patriksson, M., Str{\"o}mberg, A.B.: Ergodic results and bounds on the optimal value in subgradient optimization.
\newblock In: Operations Research Proceedings 1995: Selected Papers of the Symposium on Operations Research (SOR’95), pp. 30--35. Springer (1996)

\bibitem{lemarechal2001geometric}
Lemar{\'e}chal, C., Renaud, A.: A geometric study of duality gaps, with applications.
\newblock Mathematical Programming \textbf{90}, 399--427 (2001)

\bibitem{nedic2001distributed}
Nedic, A., Bertsekas, D., Borkar, V.: Distributed asynchronous incremental subgradient methods (2001)

\bibitem{nedic2001incremental}
Nedic, A., Bertsekas, D.P.: Incremental subgradient methods for nondifferentiable optimization.
\newblock SIAM Journal on Optimization \textbf{12}(1), 109--138 (2001)

\bibitem{nedic2009approximate}
Nedi{\'c}, A., Ozdaglar, A.: Approximate primal solutions and rate analysis for dual subgradient methods.
\newblock SIAM Journal on Optimization \textbf{19}(4), 1757--1780 (2009)

\bibitem{nedic2009distributed}
Nedic, A., Ozdaglar, A.: Distributed subgradient methods for multi-agent optimization.
\newblock IEEE Transactions on automatic control \textbf{54}(1), 48--61 (2009)

\bibitem{nemirovski1978cezari}
Nemirovski, A., Yudin, D.: On {C}ezari's convergence of the steepest descent method for approximating saddle point of convex-concave functions.
\newblock In: Soviet Mathematics. Doklady, vol.~19, pp. 258--269 (1978)

\bibitem{nemirovskij1983problem}
Nemirovskij, A.S., Yudin, D.B.: Problem complexity and method efficiency in optimization.
\newblock Wiley-Interscience  (1983)

\bibitem{nesterov2013introductory}
Nesterov, Y.: Introductory {L}ectures on {C}onvex {O}ptimization: A {B}asic {C}ourse, vol.~87.
\newblock Springer Science \& Business Media (2013)

\bibitem{nesterov2018dual}
Nesterov, Y., Shikhman, V.: Dual subgradient method with averaging for optimal resource allocation.
\newblock European Journal of Operational Research \textbf{270}(3), 907--916 (2018)

\bibitem{onnheim2017ergodic}
{\"O}nnheim, M., Gustavsson, E., Str{\"o}mberg, A.B., Patriksson, M., Larsson, T.: Ergodic, primal convergence in dual subgradient schemes for convex programming, ii: the case of inconsistent primal problems.
\newblock Mathematical Programming \textbf{163}(1), 57--84 (2017)

\bibitem{polyak1967general}
Polyak, B.T.: A general method for solving extremal problems.
\newblock In: Soviet Mathematics Doklady, vol.~8, pp. 593--597 (1967)

\bibitem{polyak1987introduction}
Polyak, B.T.: Introduction to Optimization.
\newblock Optimization Software Inc., New York (1987)

\bibitem{rockafellar2015convex}
Rockafellar, R.T.: Convex Analysis:(PMS-28).
\newblock Princeton {U}niversity {P}ress (2015)

\bibitem{ruszczynski1995convergence}
Ruszczy{\'n}ski, A.: On convergence of an augmented {L}agrangian decomposition method for sparse convex optimization.
\newblock Mathematics of Operations Research \textbf{20}(3), 634--656 (1995)

\bibitem{shor2012minimization}
Shor, N.Z.: Minimization methods for non-differentiable functions, vol.~3.
\newblock Springer Science \& Business Media (2012)

\bibitem{starr1969quasi}
Starr, R.M.: Quasi-equilibria in markets with non-convex preferences.
\newblock Econometrica: journal of the Econometric Society pp. 25--38 (1969)

\bibitem{udell2016bounding}
Udell, M., Boyd, S.: Bounding duality gap for separable problems with linear constraints.
\newblock Computational Optimization and Applications \textbf{64}(2), 355--378 (2016)

\bibitem{vujanic2014large}
Vujanic, R., Esfahani, P.M., Goulart, P., Morari, M.: Large scale mixed-integer optimization: A solution method with supply chain applications.
\newblock In: 22nd Mediterranean Conference on Control and Automation, pp. 804--809 (2014)

\bibitem{vujanic2016decomposition}
Vujanic, R., Esfahani, P.M., Goulart, P.J., Mari{\'e}thoz, S., Morari, M.: A decomposition method for large scale milps, with performance guarantees and a power system application.
\newblock Automatica \textbf{67}, 144--156 (2016)

\bibitem{wirth2024frank}
Wirth, E.S.: Frank-Wolfe Algorithms in Polytope Settings.
\newblock Technische Universitaet Berlin (Germany) (2024)

\bibitem{wolfe1970convergence}
Wolfe, P.: Convergence theory in nonlinear programming.
\newblock Integer and nonlinear programming pp. 1--36 (1970)

\bibitem{zhao1999surrogate}
Zhao, X., Luh, P.B., Wang, J.: Surrogate gradient algorithm for {L}agrangian relaxation.
\newblock Journal of optimization Theory and Applications \textbf{100}(3), 699--712 (1999)

\end{thebibliography}
\end{document}